\documentclass{amsart}[11pt]
\usepackage[letterpaper, pdftex, bookmarks=false]{hyperref}
\usepackage{fullpage}
\usepackage{afterpage}
\usepackage{graphicx}
\usepackage{tikz, pgfplots}
\usetikzlibrary{patterns}
\usepackage{amssymb, amsmath, amsthm, mathtools}
\usepackage{algorithmic}
\usepackage{algorithm}
\usepackage[numbers]{natbib}
\usepackage[all]{xy}
\usepackage{caption}
\usepackage{subcaption}
\allowdisplaybreaks

\newtheorem{theorem}{Theorem}[section]
\newtheorem{lemma}{Lemma}[section]

\newcommand{\lefri}[1]{\left(#1\right)}

\newcommand{\dt}[0]{\mbox{d}t}

\newcommand{\galqn}[0]{\tilde{q}_{n}}
\newcommand{\dgalqn}[0]{\dot{\tilde{q}}_{n}}

\newcommand{\bj}[0]{b_{j}}
\newcommand{\cjh}[0]{c_{j}h}
\newcommand{\bnj}[0]{b_{n_{j}}}
\newcommand{\cnjh}[0]{c_{n_{j}}h}

\newcommand{\EDLh}[3]{L_{d}^{E}\lefri{#1,#2,#3}}
\newcommand{\EDLn}[0]{L_{d}^{E}}

\newcommand{\GDL}[2]{L_{d}^{G}\lefri{#1,#2}}
\newcommand{\GDLh}[3]{L_{d}^{G}\lefri{#1,#2,#3}}

\newcommand{\CQ}[0]{C^{2}\lefri{\left[0,h\right],Q}}
\newcommand{\FdFSpace}[0]{\mathbb{M}^{n}\lefri{\left[0,h\right],Q}}

\newcommand{\HoA}[0]{H_{0}^{1}\lefri{\left[0,h\right],\mathfrak{g}}}

\newcommand{\SobNorm}[2]{\left\|#1\right\|_{W^{1,#2}\lefri{\left[0,h\right]}}}
\newcommand{\LNorm}[2]{\left\|#1\right\|_{L^{#2}\lefri{\left[0,h\right]}}}
\newcommand{\MNorm}[2]{\left\|#1\right\|_{#2}}

\newcommand{\ApproxC}[0]{C_{A}}
\newcommand{\ApproxK}[0]{K_{A}}
\newcommand{\SpecK}[0]{K_{s}}
\newcommand{\SpecC}[0]{C_{s}}
\newcommand{\QuadC}[0]{C_{g}}
\newcommand{\QuadK}[0]{K_{g}}
\newcommand{\CoerC}[0]{C_{f}}
\newcommand{\OptC}[0]{C_{op}}

\newcommand{\LagLipC}[0]{L_{\alpha}}

\newcommand{\sigmaC}{C_{\sigma}}
\newcommand{\mixedC}{C_{m}}
\newcommand{\ddlgC}{C_{d}}
\newcommand{\derivC}{C_{p}}

\newcommand{\dboundC}{C_{\dot{\xi}}}
\newcommand{\cboundC}{C_{\xi}}
\newcommand{\potentC}{C_{V}}

\newcommand{\GroupC}[0]{C_{G}}
\newcommand{\AlgeC}[0]{C_{\mathfrak{g}}}
\newcommand{\AlgeGC}[0]{C_{\mathfrak{g}}^{G}}
\newcommand{\ApproxCi}[0]{C_{\mathfrak{A}}}

\newcommand{\CondC}[0]{C_{con}}

\newcommand{\truqargsf}[2]{\substack{q \in \CQ\\q\lefri{0} = #1, q\lefri{h} = #2}} 

\newcommand{\galargsf}[2]{\substack{q_{n} \in \FdFSpace\\q_{n}\lefri{0} = #1, q_{n}\lefri{h} = #2}}

\newcommand{\truqargs}[0]{\substack{q\lefri{0} = q_{k}\\q\lefri{h} = q_{k+1}\\q \in \CQ}} 

\newcommand{\Ggalargsf}[2]{\substack{g_{n} \in \GalSpacef{#1} \\ g_{n}\lefri{0} = #1,g_{n}\lefri{h}=#2}}
\newcommand{\Agalargsf}[5]{\substack{g_{n} \in \GalSpacef{#1} \\ \Phi^{-1}\lefri{L_{\gk^{-1}}#4} = #2,\Phi^{-1}\lefri{L_{\gk^{-1}}#5}=#3}}
\newcommand{\altAgalargsf}[5]{\substack{g_{n} \in \GalSpacef{#1} \\ \Phi^{-1}\lefri{L_{\gkm^{-1}}#4} = #2,\Phi^{-1}\lefri{L_{\gkm^{-1}}#5}=#3}}
\newcommand{\CayGalargsf}[2]{\substack{R_{n} \in \GalSpacef{#1} \\ R_{n}\lefri{0} = #1,R_{n}\lefri{0}=#2}}

\newcommand{\darkD}[1]{\textbf{D}_{#1}}

\newcommand{\SoT}[0]{SO\lefri{3}}
\newcommand{\soT}[0]{\mathfrak{so}\lefri{3}}

\newcommand{\galgn}[0]{\tilde{g}_{n}}
\newcommand{\dgalgn}[0]{\dot{\tilde{g}}_{n}}

\newcommand{\optgn}[0]{\hat{g}_{n}}
\newcommand{\doptgn}[0]{\dot{\hat{g}}_{n}}

\newcommand{\trug}[0]{\bar{g}}
\newcommand{\dtrug}[0]{\dot{\bar{g}}}

\newcommand{\trula}[0]{\bar{\eta}}
\newcommand{\optnla}[0]{\hat{\eta}_{n}}
\newcommand{\optla}[0]{\hat{\eta}}

\newcommand{\dgalnla}[0]{\dot{\tilde{\eta}}_{n}}
\newcommand{\galla}[0]{\tilde{\eta}}
\newcommand{\dgalla}[0]{\dot{\tilde{\eta}}}
\newcommand{\dtrula}[0]{\dot{\bar{\eta}}}
\newcommand{\doptnla}[0]{\dot{\hat{\eta}}_{n}}
\newcommand{\doptla}[0]{\dot{\hat{\eta}}}

\newcommand{\Cay}[1]{\Phi\lefri{#1}}
\newcommand{\GroupE}[2]{e_{g}\lefri{#1,#2}}
\newcommand{\AlgeE}[2]{e_{a}\lefri{#1,#2}}

\newcommand{\LGForm}[1]{L_{g_{k}}\Phi\lefri{#1}}
\newcommand{\LGFormf}[2]{L_{#1}\Phi\lefri{#2}}
\newcommand{\DLGForm}[1]{D_{\Phi\lefri{#1}}L_{g_{0}}D_{#1}\Phi\lefri{\dot{#1}}}
\newcommand{\DLGFormN}[1]{D_{\Phi\lefri{#1_{n}}}L_{g_{0}}D_{#1_{n}}\Phi\lefri{\dot{#1}_{n}}}

\newcommand{\GalSpace}[0]{\mathbb{GM}^{n}\lefri{g_{0}\times \left[0,h\right],G}}
\newcommand{\GalSpacef}[1]{\mathbb{GM}^{n}\lefri{#1\times \left[0,h\right],G}}
\newcommand{\CG}[0]{C^{2}\lefri{\left[0,h\right],G}}
\newcommand{\AFdFSpace}[0]{\mathbb{M}^{n}\lefri{\left[0,h\right],\mathfrak{g}}}

\newcommand{\RMetric}[2]{\left\langle #1, #2 \right\rangle}

\newcommand{\ALd}[0]{\hat{L}_{d}}
\newcommand{\altALd}[0]{\tilde{L}_{d}}

\newcommand{\gzero}[0]{g_{0}}

\newcommand{\baseg}[0]{g_{\beta}}

\newcommand{\gk}[0]{g_{k}}
\newcommand{\gkp}[0]{g_{k+1}}
\newcommand{\gkm}[0]{g_{k-1}}

\newcommand{\alzero}[0]{\xi_{0}}
\newcommand{\alone}[0]{\xi_{1}}

\newcommand{\ak}[0]{\xi_{k}}
\newcommand{\akp}[0]{\xi_{k+1}}
\newcommand{\akm}[0]{\xi_{k-1}}

\newcommand{\altak}[0]{\lambda_{k}}

\newcommand{\altakp}[0]{\lambda_{k+1}}

\newcommand{\inak}[0]{\xi_{k}^{i}}
\newcommand{\ina}[0]{\xi^{i}}
\newcommand{\inakf}[1]{\xi_{#1}^{i}}

\newcommand{\acurve}[1]{\xi\lefri{#1}}
\newcommand{\dacurve}[1]{\dot{\xi}\lefri{#1}}

\newcommand{\lgstat}[0]{\check{\xi}}
\newcommand{\dlgstat}[0]{\dot{\check{\xi}}}

\newcommand{\GalR}{R_{n}}
\newcommand{\dGalR}{\dot{R}_{n}}
\newcommand{\R}[1]{R\lefri{#1}}
\newcommand{\dR}[1]{\dot{R}\lefri{#1}}
\newcommand{\lgc}[0]{\xi\lefri{t}}
\newcommand{\dlgc}[0]{\dot{\xi}\lefri{t}}

\newcommand{\dellgc}[0]{\delta\xi\lefri{t}}
\newcommand{\deldlgc}[0]{\delta\dot{\xi}\lefri{t}}

\newcommand{\qdellgc}[0]{\delta\xi\lefri{\cjh}}
\newcommand{\qdeldlgc}[0]{\delta\dot{\xi}\lefri{\cjh}}

\newcommand{\ilgc}[0]{\eta\lefri{t}}
\newcommand{\dilgc}[0]{\dot{\eta}\lefri{t}}

\newcommand{\ddKdddq}[0]{\frac{\partial^{2} K}{\partial \dot{\xi}^{2}}}
\newcommand{\ddKddq}[0]{\frac{\partial^{2} K}{\partial \xi^{2}}}
\newcommand{\ddKddqdq}[0]{\frac{\partial^{2} K}{\partial \xi \partial \dot{\xi}}} 
\newcommand{\ddt}[0]{\frac{\mbox{d}}{\mbox{d}t}}

\newcommand{\ext}[0]{\operatornamewithlimits{ext}}

\newcommand{\argext}[0]{\operatornamewithlimits{argext}}

\begin{document}
\title{Lie Group Spectral Variational Integrators}

\author[J.~Hall]{James Hall}
\address{Department of Mathematics\\
University of California, San Diego\\
9500 Gilman Drive \#0112\\
La Jolla, California 92093-0112, USA}
\email{j9hall@math.ucsd.edu}

\author[M.~Leok]{Melvin Leok}
\address{Department of Mathematics\\
University of California, San Diego\\
9500 Gilman Drive \#0112\\
La Jolla, California 92093-0112, USA}
\email{mleok@math.ucsd.edu}

\begin{abstract}
We present a new class of high-order variational integrators on Lie groups. We show that these integrators are symplectic, momentum preserving, and can be constructed to be of arbitrarily high-order, or can be made to converge geometrically. Furthermore, these methods are stable and accurate for very large time steps. We demonstrate the construction of one such variational integrator for the rigid body, and discuss how this construction could be generalized to other related Lie group problems. We close with several numerical examples which demonstrate our claims, and discuss further extensions of our work.
\end{abstract}

\maketitle

\section{Introduction}

There is a deep and elegant geometric structure underlying the dynamics of many mechanical systems. Conserved quantities, such as the energy, momentum, and symplectic form offer insight into this structure, and through this, we obtain an understanding of the behavior of these systems that goes beyond what is conventionally available. Conservation laws reveal much about the stability and long term behavior of a system, and can even characterize the entire dynamics of a system when a sufficient number of them exist. Hence, there has been much recent interest in the field of geometric mechanics, which seeks to understand this structure using differential geometric and symmetry techniques.

From this geometric mechanics framework, it is possible to formulate numerical methods which respect much of the geometry of mechanical systems. There are a variety of approaches for constructing such methods, often known as \emph{structure-preserving methods}, including projection methods, splitting methods, symplectic Runge-Kutta methods, B-series expansion methods, to name a few. An extensive introduction can be found in \citet{HaLuWa2006}. One of the powerful frameworks, discrete mechanics, approaches the construction of numerical methods by developing much of the theory of geometric mechanics from a discrete standpoint. This approach has proven highly effective for constructing methods for problems in Hamiltonian and Lagrangian mechanics, specifically because these type of problems arise from a variational principle. Methods that make use of a variational principle and the framework of discrete mechanics are referred to as \emph{variational integrators}, and they have many favorable geometric properties, including conservation of the symplectic form and momentum.

A further advantage of variational integrators is that it is often straightforward to analyze the error of these methods. This has led turn to the development of high-order variational integrators, which can be constructed so that they converge very quickly. In \citet{HaLe2012}, such integrators for vector space problems were presented and analyzed. It was shown that such integrators can be arbitrarily high-order or even exhibit geometric convergence. Furthermore, these integrators are stable and accurate even with extremely large time steps, and using them it is easy to reconstruct highly accurate continuous approximations to the dynamics of the system of interest.

In this paper, we present an extension of that work to Lie group methods. Lie group methods are of particular interest in science and engineering applications. It can be shown that many problems of interest, from the dynamics of rigid bodies to the behavior of incompressible fluids, evolve on Lie groups. Furthermore, if a traditional numerical method is applied to a problem with dynamics in a Lie group, the approximate solution will typically depart from the Lie group, destroying a critical structural property of the solution. Our work gives general a general framework for constructing methods which will always evolve in the Lie group and which will share many of the desirable properties of the vector space type methods. Specifically, we will be able to construct methods of arbitrarily high-order and with geometric convergence, and we will be able to reconstruct high quality continuous approximations from these methods.

Lie group methods have a rich history and remain the subject of significant interest. An extensive introduction can be found in \citet{IsMKNZa2000}, which provides an excellent exposition of both the motivation for Lie group methods and many of the techniques used on Lie groups. Likewise, \citet{CeBy2003}, provide a very helpful general introduction to Lie group methods for the rigid body, which is a prototypical example of an interesting Lie group problem. In this paper, we provide a thorough example of the construction of our method for the rigid body, as this approach can easily be extended to other interesting problems. More recently \citet{BogMa2013} investigated the construction of high-order symplectic Lie group integrators from a discrete Hamilton-Pontryagin principle, and \citet{BuHoMe2011} described applications of such high-order Lie group discretizations, such as interpolation in \(\SoT\).

Galerkin variational integrators were proposed in \citet{MaWe2001}, and expanded on by \citet{Le04}. The concept of a Galerkin Lie group integrator was proposed in \citet{Le04} and expanded in \citet{LeSh2011b}. Our work expands upon this by generalizing both the diffeomorphisms used to construct the natural charts and the approximation spaces used to construct the curve on the Lie group, and establishing convergence results and properties of both the discrete solution and the continuous approximation.

\subsection{Discrete Mechanics}

Since we are working from the perspective of discrete mechanics, we will take a moment to review the fundamentals of the theory here. This will only be a brief summary, and extensive exposition of the theory can be found in \citet{MaWe2001}.

Consider a configuration manifold, \(Q\), which describes the configuration of a mechanical system at a given point in time. In discrete mechanics, the fundamental object is the \emph{discrete Lagrangian}, \(L_{d}:Q\times Q \times \mathbb{R} \rightarrow \mathbb{R}\). The discrete Lagrangian can be viewed as an approximation to the \emph{exact discrete Lagrangian} \(L_{d}^{E}\), where the \(L_{d}^{E}\) is defined to be the action of the Lagrangian on the solution of the Euler-Lagrange equations over a short time interval:%
\begin{align*}
L_{d}\lefri{q_{0},q_{1},h} \approx \EDLh{q_{0}}{q_{1}}{h} = \ext_{\truqargs{q_{0}}{q_{1}}} \int_{0}^{h} L\lefri{q,\dot{q}}\dt.
\end{align*}%
The discrete Lagrangian gives rise to a discrete action sum, which can be viewed as an approximation to the action over a long time interval:%
\begin{align*}
\mathbb{S}\lefri{\left\{q_{k}\right\}_{k=1}^{n}} = \sum_{k=0}^{n-1} L_{d}\lefri{q_{k},q_{k+1}} \approx \int_{t_{0}}^{t_{n}} L\lefri{q,\dot{q}} \dt,
\end{align*}%
and requiring stationarity of this discrete action sum subject to fixed endpoint conditions \(q_{0},q_{n}\), gives rise to the \emph{discrete Euler-Lagrange equations}:%
\begin{align} \label{dEL}
D_{1}L_{d}\lefri{q_{k},q_{k+1},h} + D_{2}L_{d}\lefri{q_{k-1},q_{k},h} = 0,
\end{align}%
where \(D_{i}\) denotes partial differentiation of a function with respect to the \(i\)-th argument. Given a point \(\lefri{q_{k-1},q_{k}}\), these equations implicitly define an update map \(F_{L_{d}}:\lefri{q_{k-1},q_{k}} \rightarrow \lefri{q_{k},q_{k+1}}\), which approximates the solution of the Euler-Lagrange equations for the continuous system. A numerical method which uses the update map \(F_{L_{d}}\) to construct numerical solutions to ODEs is referred to as a \emph{variational integrator}.

The power of discrete mechanics is derived from the discrete variational structure. Since the update map \(F_{L_{d}}\) is induced from a discrete analogue of the variational principle, much of the geometric structure from continuous mechanics can be extended to discrete mechanics. The discrete Lagrangian gives rise to \emph{discrete Legendre Transforms} \(\mathbb{F}L^{\pm}:Q \times Q\rightarrow T^{*}Q\):%
\begin{align*}
\mathbb{F}L_{d}^{+}\lefri{q_{k},q_{k+1}} &= \lefri{q_{k+1},D_{2}L_{d}\lefri{q_{k},q_{k+1}}},\\
\mathbb{F}L_{d}^{-}\lefri{q_{k},q_{k+1}} &= \lefri{q_{k},-D_{1}L_{d}\lefri{q_{k},q_{k+1}}},
\end{align*}
which lead to the extension of other classical geometric structures. It is important to note that, while there are two different discrete Legendre transforms, (\ref{dEL}) guarantees that \(\mathbb{F}L_{d}^{-}\lefri{q_{k},q_{k+1}} = \mathbb{F}L_{d}^{+}\lefri{q_{k-1},q_{k}}\), and thus they can be used interchangeably when defining the discrete geometric structure. By their construction, variational integrators induce a discrete symplectic form, i.e. \(\Omega_{L_{d}} = \lefri{\mathbb{F}^{\pm}L_{d}}^{*}\Omega\) which is conserved by the update map \(F^{*}_{L_{d}}\Omega_{L_{d}} = \Omega_{L_{d}}\), and a discrete analogue of Noether's Theorem, which states that if a discrete Lagrangian is invariant under a diagonal group action on \(\lefri{q_{k},q_{k+1}}\), it induces a discrete momentum map \(J_{L_{d}} = \lefri{\mathbb{F}L_{d}^{\pm}}^{*}J\), which is preserved under the update map: \(F^{*}_{L_{d}}J_{L_{d}} = J_{L_{d}}\). The existence of these discrete geometric conservation laws gives a systematic framework to construct powerful numerical methods which preserve structure.

The discrete Legendre transforms also allow us to define an update map through phase space \(\tilde{F}_{L_{d}}:T^{*}Q \rightarrow T^{*}Q\),%
\begin{align*}
\tilde{F}_{L_{d}}\lefri{q_{k},p_{k}} = \lefri{q_{k+1},p_{k+1}},
\end{align*}%
which is given by%
\begin{align*}
\tilde{F}_{L_{d}}\lefri{q_{k},p_{k}} = \mathbb{F}^{+}L_{d}\lefri{\lefri{\mathbb{F}^{-}L_{d}}^{-1}\lefri{q_{k},p_{k}}},
\end{align*}
known as the Hamiltonian flow map. As long as the discrete Lagrangian is sufficiently smooth, the Hamiltonian flow map and the Lagrangian flow map are compatible, and the geometric structure of discrete flow can be understood from either perspective, just as in the continuous theory.

The following commutative diagram illustrates the relationship between the discrete Legendre transforms, the Lagrangian flow map, the Hamiltonian flow map, and the discrete Lagrangian.%
\begin{align*}
\xymatrix{
& \lefri{q_{k},p_{k}}  \ar[rr]^{\tilde{F}_{L_{d}}} & & \lefri{q_{k+1},p_{k+1}} & \\
& & & & \\
\lefri{q_{k-1},q_{k}} \ar[uur]^{\mathbb{F}^{+}L_{d}} \ar[rr]_{F_{L_{d}}} & & \lefri{q_{k},q_{k+1}} \ar[rr]_{F_{L_{d}}} \ar[uur]^{\mathbb{F}^{+}L_{d}} \ar[uul]_{\mathbb{F}^{-}L_{d}}& &\lefri{q_{k+1},q_{k+2}} \ar[uul]_{\mathbb{F}^{-}L_{d}}
}
\end{align*}
A further consequence of the discrete mechanics framework is that it provides a natural mechanism for analyzing the order of accuracy of a variational integrator. Specifically, it can be shown that the variational integrator induced by the exact discrete Lagrangian produces an exact sampling of the true flow. Based on this, we have the following theorem which is critical for the error analysis of variational integrators:%
\begin{theorem}\label{ErrorThm}{\emph{Variational Order Analysis (Theorem 2.3.1 of \citet{MaWe2001}).}} If a discrete Lagrangian \(L_{d}\) approximates the exact discrete Lagrangian \(L_{d}^{E}\) to order p, i.e. \(L_{d}\lefri{q_{0},q_{1},h} = L_{d}^{E}\lefri{q_{0},q_{1},h} + \mathcal{O}\lefri{h^{p+1}}\), then the variational integrator induced by \(L_{d}\) is order \(p\) accurate.
\end{theorem}%
This theorem allows for greatly simplified \emph{a priori} error estimates of variational integrators, and is a fundamental tool for the development and analysis of high-order variational integrators.

\section{Construction}

\subsection{General Galerkin Variational Integrators}

Lie group Galerkin variational integrators are an extension of Galerkin variational integrators to Lie groups. As such, we will briefly review the construction of general Galerkin variational integrators.

The driving idea behind Galerkin variational integrators is approaching the construction of a discrete Lagrangian as the approximation of a variational problem. We know from discrete mechanics that the exact discrete Lagrangian \(\EDLn:Q\times Q \times \mathbb{R} \rightarrow \mathbb{R}\),%
\begin{align*}
\EDLh{q_{0}}{q_{1}}{h} = \ext_{\truqargsf{q_{0}}{q_{1}}}\int_{0}^{h}L\lefri{q,\dot{q}}\dt,
\end{align*}%
induces a variational integrator that produces an exact sampling of the true flow, and the accuracy with which a variational integrator approximates the true solution is the same as the accuracy to which the discrete Lagrangian used to construct it approximates the exact discrete Lagrangian. Hence, to construct a highly accurate discrete Lagrangian, we construct a discrete approximation%
\begin{align*}
\GDLh{q_{0}}{q_{1}}{h} = \ext_{\galargsf{q_{0}}{q_{1}}} h\sum_{j=1}^{m}b_{j}L\lefri{q_{n}\lefri{c_{j}h},\dot{q}_{n}\lefri{c_{j}h}} \approx \ext_{\truqargsf{q_{0}}{q_{1}}}\int_{0}^{h}L\lefri{q,\dot{q}}\dt
\end{align*}%
by replacing the function space \(\CQ\) with a finite-dimensional subspace \(\FdFSpace \subset \CQ\) and the integral with a quadrature rule, \(h\sum_{j=1}^{m}b_{j}f\lefri{c_{j}h} \approx \int_{0}^{h}f\dt\). Finding the extremizer of the discrete action is computationally feasible, and by computing this extremizer we can construct the variational integrator that results from the discrete Lagrangian. Because this approach of replacing the function space \(\CQ\) with a finite-dimensional subspace is inspired by Galerkin methods for partial differential equations, we refer to variational integrators constructed in this way as \emph{Galerkin variational integrators}.

In \citet{HaLe2012}, we studied Galerkin variational integrators on linear spaces. Specifically, we obtained several significant results, including that Galerkin variational integrators over linear spaces are in a certain sense order-optimal, and that by enriching the function space \(\FdFSpace\), as opposed to shortening the time step \(h\), we can construct variational integrators that converge geometrically. Furthermore, we established that it is easy to recover a continuous approximation to the trajectory over the time step \(\left[0,h\right]\), and that the convergence of this continuous approximation is related to the rate of convergence of the variational integrator. Finally, we established an error bound on Noether quantities evaluated on this continuous approximation which is independent of the number of steps taken.

\subsection{Lie Group Galerkin Variational Integrators}

The construction and analysis in \citet{HaLe2012} relied on the linear structure of the spaces involved. At their heart, Galerkin variational integrators make use of a Galerkin curve%
\begin{align*}
\galqn\lefri{t} = \sum_{i=1}^{n}q^{i}\phi_{i}\lefri{t}
\end{align*}%
for some set of points \(\left\{q^{i}\right\}_{i=1}^{n} \subset Q\) and basis functions \(\left\{\phi_{i}\right\}_{i=1}^{n}\). While for linear spaces, \(\galqn\lefri{t} \in Q\) for any choice of \(t\), in nonlinear spaces this will not be the case. However, when \(Q\) is a Lie group, it is possible to extend this construction in a way that keeps the curve \(\galqn\lefri{t}\) in \(Q\).

\subsubsection{Natural Charts}

To generalize Galerkin variational integrators to Lie groups, we will make use of the linear nature of the Lie algebra associated with the Lie group. Specifically, given a Lie group \(G\) and its associated Lie algebra \(\mathfrak{g}\), we choose a local diffeomorphism \(\Phi:\mathfrak{g}\rightarrow G\). Then, given a set of points in the Lie group \(\left\{g_{i}\right\}_{i=1}^{n} \subset G\) and a set of associated interpolation times \(t_{i}\), we can construct an interpolating curve \(g: G^{n} \times \mathbb{R}\rightarrow G\) such that \(g\lefri{\left\{g_{i}\right\}_{i=1}^{n},t_{i}} = g_{i}\), given by%
\begin{align*}
g\lefri{\left\{g_{i}\right\}_{i=1}^{n},t} = \LGFormf{g_{1}}{\sum_{i=1}^{n}\Phi^{-1}\lefri{L_{g_{1}^{-1}}g_{i}}\phi_{i}\lefri{t}},
\end{align*}%
where \(L_{g}\) is the left group action of \(g\) and \(\phi_{i}\lefri{t}\) is the Lagrange interpolation polynomial for \(t_{i}\). A key feature of this type of curve is that is \emph{Lie group equivariant}, that is, \(g\lefri{\left\{L_{g_{0}}g_{i}\right\}_{i=1}^{n},t} = L_{g_{0}}g\lefri{\left\{g_{i}\right\}_{i=1}^{n},t}\), as we shall show in the following lemma.

\begin{lemma}
The curve \(g\lefri{\left\{g_{i}\right\}_{i=1}^{n},t}\) is Lie group equivariant.
\end{lemma}

\begin{proof}
The proof is a direct calculation.%
\begin{align*}
g\lefri{\left\{L_{g_{0}}g_{i}\right\},t} &= L_{L_{g_{0}}g_{1}}\Phi\lefri{\sum_{i=1}^{n}\Phi^{-1}\lefri{L_{\lefri{L_{g_{0}}g_{1}}^{-1}}L_{g_{0}}g_{i}}\phi_{i}\lefri{t}} \\
&= L_{g_{0}}L_{g_{1}}\Phi\lefri{\sum_{i=1}^{n}\Phi^{-1}\lefri{L_{g_{1}^{-1}}L_{g_{0}^{-1}}L_{g_{0}}g_{i}}\phi_{i}\lefri{t}}\\
&= L_{g_{0}}L_{g_{1}}\Phi\lefri{\sum_{i=1}^{n}\Phi^{-1}\lefri{L_{g_{1}^{-1}}g_{i}} \phi_{i}\lefri{t}} \\
&= L_{g_{0}}g\lefri{\left\{g_{i}\right\}_{i=1}^{n},t}.
\end{align*}%
\end{proof}%
This property will be important for ensuring that the Lie group Galerkin discrete Lagrangian inherits the symmetries of the continuous Lagrangian; these inherited symmetries give rise to the structure-preserving properties of the resulting variational integrator.
 
Throughout this paper, we will consider the function spaces composed of curves of this form. We note that \(\Phi^{-1}\lefri{L_{g_{1}^{-1}}g_{i}} \in \mathfrak{g}\), and for any \(\xi \in \mathfrak{g}\), \(L_{g_{1}}\Phi\lefri{\xi} \in G\), so we can construct interpolation curves on the group in terms of interpolation curves in the Lie algebra. In light of this, we define%
\begin{align*}
\GalSpace := \left\{g\lefri{\left\{\ina\right\}_{i=1}^{n},t} \vphantom{ g\lefri{\left\{\ina\right\}_{i=1}^{n},t} = L_{\gzero}\Phi\lefri{\sum_{i=1}^{n}\ina \phi_{i}\lefri{t}}, \ina \in \mathfrak{g}}\right. \left|\; g\lefri{\left\{\ina\right\}_{i=1}^{n},t} = \LGFormf{g_{0}}{\sum_{i=1}^{n}\ina \phi_{i}\lefri{t}}, \ina \in \mathfrak{g}\right\}
\end{align*}%
where \(\left\{\phi_{i}\lefri{t}\right\}_{i=1}^{n}\) forms the basis for a finite dimensional-approximation space in \(\mathbb{R}\), for example, Lagrange interpolation polynomials, which is what we will use in our explicit construction in \S \ref{CayCon} and numerical examples in \S \ref{NumExp}. We refer to the space of finite dimensional curves in the Lie algebra as%
\begin{align*}
\AFdFSpace = \left\{\lgc \vphantom{\lgc = \sum_{i=1}^{n} \ina \phi_{i}\lefri{t}, \ina \in \mathfrak{g}, \phi_{i}:\left[0,h\right] \rightarrow \mathbb{R}} \right.\left|\; \vphantom{\lgc} \lgc = \sum_{i=1}^{n} \ina \phi_{i}\lefri{t}, \ina \in \mathfrak{g}, \phi_{i}:\left[0,h\right] \rightarrow \mathbb{R}\right\}.
\end{align*}

Because we are identifying every point in a neighborhood of the Lie group with a point in the Lie algebra, which is a vector space, it is natural to think of this construction as choosing a set of coordinates for a neighborhood in the Lie group. Thus, we can consider this construction as choosing a chart for a neighborhood of the Lie group, and because it makes use of the ``natural'' relationship between the Lie group \(G\), its Lie algebra \(\mathfrak{g}\), and the tangent space of the Lie group \(TG\), we call the function \(\varphi_{g_{0}}:G\rightarrow \mathfrak{g}\), \(\varphi_{g_{0}}\lefri{\cdot}= \Phi^{-1}\lefri{L_{g_{0}^{-1}}\lefri{\cdot}}\) a ``natural chart.''

\subsubsection{Discrete Lagrangian}

Now that we have introduced a Lie group approximation space, we can define a compatible discrete Lagrangian for Lie group problems. We take a similar approach to the construction for vector spaces; we construct an approximation to the action of the Lagrangian over \(\left[0,h\right]\) by replacing \(\CG\) with a finite-dimensional approximation space and the integral with a quadrature rule, and then compute its extremizer. Specifically, given a Lagrangian on the tangent space of a Lie group \(L:TG \rightarrow \mathbb{R}\), the associated Lie group Galerkin discrete Lagrangian is defined to be:%
\begin{align*}
L_{d}\lefri{\gk,\gkp} = \ext_{\Ggalargsf{\gk}{\gkp}} h \sum_{j=1}^{m} \bj L\lefri{g_{n}\lefri{\cjh},\dot{g}_{n}\lefri{\cjh}}.
\end{align*}%

\subsubsection{Internal Stage Discrete Euler-Poincar\'{e} Equations} \label{ISDEPSubsection}

This discrete Lagrangian involves solving an optimization problem, namely: find \(\galgn\lefri{t} \in \GalSpacef{g_{k}}\) such that \(\galgn\lefri{0} = g_{k}\), \(\galgn\lefri{h} = g_{k+1}\), and %
\begin{align}
h \sum_{j=1}^{m} \bj L\lefri{\galgn\lefri{\cjh},\dgalgn\lefri{\cjh}}= \ext_{\Ggalargsf{\gk}{\gkp}} h \sum_{j=1}^{m} \bj L\lefri{g_{n}\lefri{\cjh},\dot{g}_{n}\lefri{\cjh}}. \label{ExtCond}
\end{align}%
While this problem can be solved using standard methods of numerical optimization, it is also possible to reduce it to a root finding problem. Since each curve \(\galgn\lefri{t} \in \GalSpace\) is parametrized by a finite number of Lie algebra points \(\{\ina\}_{i=1}^{n}\), by taking discrete variations of the discrete Lagrangian with respect to these points, we can derive stationarity conditions for the extremizer. Specifically, if we denote%
\begin{align*}
\acurve{t} = \sum_{i=1}^{n} \ina \phi_{i}\lefri{t}
\end{align*}
then a straightforward computation reveals the stationarity condition:%
\begin{align*}
&h\sum_{j=1}^{m} b_{j} \left(\darkD{1}L \circ \darkD{\Phi\lefri{\acurve{\cjh}}}L_{\gk} \circ \darkD{\acurve{\cjh}}\Phi \circ \lefri{\sum_{i=1}^{n}\darkD{\ina}\acurve{\cjh}\cdot \delta \ina} +  \right. \\
& \hspace{3em} \left. \darkD{2}L \circ \darkD{\lefri{\Phi\circ \acurve{\cjh}, \darkD{\dacurve{\cjh}}\Phi\circ \dacurve{\cjh}}} \darkD{\Phi \circ \acurve{\cjh}} L_{\gk} \circ \darkD{\lefri{\acurve{\cjh},\dacurve{\cjh}}}\Phi \circ \lefri{\sum_{i=1}^{n}\darkD{\ina}\dacurve{\cjh} \cdot \delta \ina}\right) = 0
\end{align*}%
for arbitrary \(\left\{\delta \ina\right\}_{i=1}^{n}\). Using standard calculus of variations arguments, this reduces to%
\begin{align*}
&h\sum_{j=1}^{m} b_{j} \left(\darkD{1}L \circ \darkD{\Phi\lefri{\acurve{\cjh}}}L_{\gk} \circ \darkD{\acurve{\cjh}}\Phi \circ \darkD{\ina}\acurve{\cjh}\cdot \delta \ina + \vphantom{\darkD{2}L \circ \darkD{\lefri{\Phi\circ \acurve{\cjh}, \darkD{\dacurve{\cjh}}\Phi\circ \dacurve{\cjh}}} \darkD{\Phi \circ \acurve{\cjh}} L_{\gk} \circ \darkD{\lefri{\acurve{\cjh},\dacurve{\cjh}}}\Phi \circ \darkD{\ina}\dacurve{\cjh} \cdot \delta \ina} \right. \\
& \hspace{3em} \left. \darkD{2}L \circ \darkD{\lefri{\Phi\circ \acurve{\cjh}, \darkD{\dacurve{\cjh}}\Phi\circ \dacurve{\cjh}}} \darkD{\Phi \circ \acurve{\cjh}} L_{\gk} \circ \darkD{\lefri{\acurve{\cjh},\dacurve{\cjh}}}\Phi \circ \darkD{\ina}\dacurve{\cjh} \cdot \delta \ina\right) = 0
\end{align*}%
for \(i = 2,..,n-1\) (note that the sum of the Lie algebra elements has disappeared). Now using the linearity of one-forms, we can collect terms to further simplify this expression to%
\begin{align*}
&h\sum_{j=1}^{m} b_{j} \left(\left[\darkD{1}L \circ \darkD{\Phi\lefri{\acurve{\cjh}}}L_{\gk} \circ \darkD{\acurve{\cjh}}\Phi \circ \darkD{\ina}\acurve{\cjh}+  \vphantom{\darkD{2}L \circ \darkD{\lefri{\Phi\circ \acurve{\cjh}, \darkD{\dacurve{\cjh}}\Phi\circ \dacurve{\cjh}}} \darkD{\Phi \circ \acurve{\cjh}} L_{\gk} \circ \darkD{\lefri{\acurve{\cjh},\dacurve{\cjh}}}\Phi \circ \darkD{\ina}\dacurve{\cjh}\cdot \delta \ina}\right.\right.  \\
& \hspace{3em} \left. \left. \darkD{2}L \circ \darkD{\lefri{\Phi\circ \acurve{\cjh}, \darkD{\dacurve{\cjh}}\Phi\circ \dacurve{\cjh}}} \darkD{\Phi \circ \acurve{\cjh}} L_{\gk} \circ \darkD{\lefri{\acurve{\cjh},\dacurve{\cjh}}}\Phi \circ \darkD{\ina}\dacurve{\cjh}\right] \cdot \delta \ina\right) = 0
\end{align*}%
for \(i = 2,...,n-1\). Since \(\delta \ina\) is arbitrary, this implies that %
\begin{align}
&h\sum_{j=1}^{m} b_{j} \left(\darkD{1}L \circ \darkD{\Phi\lefri{\acurve{\cjh}}}L_{\gk} \circ \darkD{\acurve{\cjh}}\Phi \circ \darkD{\ina}\acurve{\cjh}\cdot + \vphantom{\darkD{2}L \circ \darkD{\lefri{\Phi\circ \acurve{\cjh}, \darkD{\dacurve{\cjh}}\Phi\circ \dacurve{\cjh}}} \darkD{\Phi \circ \acurve{\cjh}} L_{\gk} \circ \darkD{\lefri{\acurve{\cjh},\dacurve{\cjh}}}\Phi \circ \darkD{\ina}\dacurve{\cjh} \cdot \delta \ina}\right. \label{InterEPEquations} \\
& \hspace{3em} \left. \darkD{2}L \circ \darkD{\lefri{\Phi\circ \acurve{\cjh}, \darkD{\dacurve{\cjh}}\Phi\circ \dacurve{\cjh}}} \darkD{\Phi \circ \acurve{\cjh}} L_{\gk} \circ \darkD{\lefri{\acurve{\cjh},\dacurve{\cjh}}}\Phi \circ \darkD{\ina}\dacurve{\cjh}\right) = 0 \nonumber
\end{align}%
for \(i = 2,...,n-1\). These equations, which we shall refer to as the \emph{internal stage discrete Euler-Poincar\'{e} equations}, combined with the standard momentum matching condition,%
\begin{align}
D_{2}L_{d}\lefri{g_{k-1},g_{k}} + D_{1}L_{d}\lefri{g_{k},g_{k+1}} = 0, \label{DEPEquations}
\end{align}%
which we will discuss in more detail in the \S \ref{DEPSection}, can be easily solved with an iterative nonlinear equation solver. The result is a curve \(\galgn\lefri{t}\) which satisfies condition (\ref{ExtCond}). The next step of the one-step map is given by \(\galgn\lefri{h} = g_{k+1}\), which gives the variational integrator.

It should be noted that while the internal stage discrete Euler-Poincar\'{e} equations can be computed by deriving all of the various differentials in the chosen coordinates, it is often much simpler to form the discrete action%
\begin{align*}
\mathbb{S}_{d}\lefri{\left\{\ina\right\}_{i=1}^{n}} = h \sum_{j=1}^{m} b_{j} L \lefri{L_{g_{k}}\Phi\lefri{\sum_{i=1}^{n} \ina \phi_{i}\lefri{\cjh}}, \ddt \lefri{L_{g_{k}} \Phi\lefri{\sum_{i=1}^{n} \ina \phi_{i}\lefri{\cjh}}}}
\end{align*}%
explicitly and then compute the stationarity conditions directly in coordinates, rather than a step by step computation of the different maps in (\ref{InterEPEquations}). This is the approach we take when deriving the integrator for the rigid body in \S \ref{CayCon}, and it appears to be the much simpler approach in this case. However, the two approaches are equivalent, so if done carefully either will suffice to give the internal stage Euler-Poincar\'{e} equations.

\subsubsection{Momentum Matching Condition} \label{DEPSection}

A difficulty in the derivation of the discrete Euler-Poincar\'{e} equations is the computation of the discrete momentum terms%
\begin{align*}
p_{k,k+1}^{-} &= -D_{1}L_{d}\lefri{\gk,\gkp} \\
p_{k-1,k}^{+} &= D_{2}L_{d}\lefri{\gkm,\gk}
\end{align*}%
which are used in the discrete Euler-Poincar\'{e} equations (\ref{DEPEquations}),%
\begin{align*}
D_{1}L_{d}\lefri{\gk,\gkp} + D_{2}L_{d}\lefri{\gkm,\gk} = 0
\end{align*}
or
\begin{align*}
p_{k-1,k}^{+} = p_{k,k+1}^{-}.
\end{align*}
The difficulty arises because the discrete Lagrangian makes use of a local left trivialization. Through the local charts, we reduce the discrete Lagrangian to a function of algebra elements, and because the corresponding group elements are recovered through a complicated computation, working with the group elements directly to compute the discrete Euler-Poincar\'{e} equations is difficult. Because of this, to compute the discrete Euler-Poincar\'{e} equations, it is more natural to think of the discrete Lagrangian as a function of two Lie algebra elements. If we define a discrete Lagrangian on the Lie algebra \(\ALd: \mathfrak{g} \times \mathfrak{g} \rightarrow \mathbb{R}\) as%
\begin{align*}
\ALd\lefri{\ak,\akp} = \ext_{\Agalargsf{\gk}{\ak}{\akp}{g_{n}\lefri{0}}{g_{n}\lefri{h}}}h \sum_{j=1}^{m} \bj L\lefri{g_{n}\lefri{\cjh},\dot{g}_{n}\lefri{\cjh}}
\end{align*}
and compare it to the discrete Lagrangian on the Lie group,
\begin{align*}
L_{d}\lefri{\gk,\gkp} = \ext_{\Ggalargsf{\gk}{\gkp}}h\sum_{j=1}^{m}\bj L\lefri{g_{n}\lefri{\cjh},\dot{g}_{n}\lefri{\cjh}}
\end{align*}%
it can be seen that there is a simple one-to-one correspondence through the natural charts between points in \(G \times G\)  and points in \(\mathfrak{g} \times \mathfrak{g}\), and that if \(\lefri{\Phi^{-1}\lefri{L_{g_{k}^{-1}}g_{0}},\Phi^{-1}\lefri{L_{g_{k}^{-1}}g_{1}}} = \lefri{\alzero,\alone}\), then%
\begin{align*}
L\lefri{g_{0},g_{1}} = \ALd\lefri{\alzero,\alone}.
\end{align*}%
Hence, for every sequence \(\left\{g_{k}\right\}_{k=1}^{N}\), there exists a unique sequence \(\left\{\ak\right\}_{i=1}^{N}\) such that%
\begin{align}
\sum_{k=1}^{N-1}L_{d}\lefri{\gk,\gkp} = \sum_{k=1}^{N-1}\ALd\lefri{\ak,\akp}, \label{DiscreteActionSum}
\end{align}%
and vice versa. Thus, we can find the sequence \(\left\{g_{k}\right\}_{k=1}^{N}\) that makes the sum on the left hand side of (\ref{DiscreteActionSum}) stationary by finding the sequence \(\left\{\ak\right\}_{k=1}^{N}\) that makes the sum on the right hand side of (\ref{DiscreteActionSum}) stationary. 

It can easily be seen that the stationarity condition of the action sum on the right is%
\begin{align}
D_{2}\ALd\lefri{\akm,\ak} + D_{1}\ALd\lefri{\ak,\akp} = 0, \label{DEPStationary}.
\end{align}%
However, from the definition of \(\ALd\), this implicitly assumes that \(\lefri{\akm,\ak}\) and \(\lefri{\ak,\akp}\) are in the same natural chart. Unfortunately, in our construction \(\lefri{\akm,\ak}\) and \(\lefri{\ak,\akp}\) are in different natural charts. This is because the construction of the Lie group interpolating curve%
\begin{align*}g\lefri{t} = L_{\baseg}\Phi\lefri{\sum_{i=1}^{n}\inak\phi_{i}\lefri{t}}\end{align*}%
requires the choice of a base point for the natural chart \(\baseg \in G\). If a consistent choice of base point was made for each time step, then the above equations could be directly computed without difficulty. However, because many natural chart functions contain coordinate singularities, our construction uses a different base point, and thus a different natural chart, at each time step. Specifically, on the interval \(\left[kh,(k+1)h\right]\), we choose \(\baseg = \gk\) and define%
\begin{align*}
g\lefri{t} = L_{\gk}\Phi\lefri{\sum_{i=1}^{n}\inak\phi_{i}\lefri{t}}.
\end{align*}%
Thus%
\begin{align*}
g\lefri{t} &= L_{\gkm}\Phi\lefri{\sum_{j=1}^{n}\inakf{k-1}\phi_{i}\lefri{t}}, t \in \left[\lefri{k-1}h,kh\right]\\
g\lefri{t} &= L_{\gk}\Phi\lefri{\sum_{j=1}^{n}\inak\phi_{i}\lefri{t}}, t \in \left[kh,\lefri{k+1}h\right],
\end{align*}%
where we now denote internal stage points \(\inak\) with the subscript \(k\) to denote in which interval they occur. While the change in natural chart is expedient for the construction, it creates a difficulty for the computation of the discrete Euler-Poincar\'{e} equations, in that now we are using discrete Lagrangians with different natural charts for the different time steps, and hence we cannot compute the discrete Euler-Poincar\'{e} equations using (\ref{DEPStationary}). This problem can be resolved by expressing \(g\lefri{t}\), and hence \(\lefri{\akm,\ak}\) and \(\lefri{\ak,\akp}\), in the same natural chart for \(t\in\left[\lefri{k-1}h,\lefri{k+1}h\right]\). Rewriting%
\begin{align*}
g_{n}\lefri{t} = L_{\gk}\Phi\lefri{\sum_{i=1}^{n}\inak \phi_{i}\lefri{t}} = L_{\gkm}\Phi\lefri{\Phi^{-1}\lefri{L_{\gkm^{-1}}L_{\gk}\Phi\lefri{\sum_{i=1}^{n}\inak \phi_{i}\lefri{t}}}}, t \in \left[kh,\lefri{k+1}h\right],
\end{align*}%
(note that \(g_{n}\lefri{t}\) is still in \(\GalSpacef{g_{k}}\)), and defining
\begin{align*}
\lambda\lefri{t} = \Phi^{-1}\lefri{L_{g_{k-1}^{-1}}L_{g_{k}}\Phi\lefri{\sum_{i=1}^{n}\inak \phi_{i}\lefri{t}}} = \Phi^{-1}\lefri{L_{\Phi\lefri{\ak}}\Phi\lefri{\sum_{i=1}^{n}\inak \phi_{i}\lefri{t}}}
\end{align*}
we can reexpress the discrete Lagrangian as %
\begin{align*}
\altALd\lefri{\altak,\altakp} = \ext_{\altAgalargsf{g_{k}}{\altak}{\altakp}{g_{n}\lefri{0}}{g_{n}\lefri{h}}}h \sum_{j=1}^{m} \bj L\lefri{g_{n}\lefri{\cjh},\dot{g}_{n}\lefri{\cjh}}.
\end{align*}
Note that if \(L_{g_{k-1}}\Phi\lefri{\altak} = L_{g_{k}}\Phi\lefri{\ak}\) and \(L_{g_{k-1}}\Phi\lefri{\altakp} = L_{g_{k}}\Phi\lefri{\akp}\) that%
\begin{align*}
\ALd\lefri{\ak,\akp} = \altALd\lefri{\altak,\altakp}.
\end{align*}%
Furthermore, \(\lefri{\altak,\altakp}\) are in the same chart as \(\lefri{\akm,\ak}\), and hence the discrete Euler-Poincar\'{e} equations are
\begin{align*}
D_{2}\ALd\lefri{\akm,\ak} + D_{1}\altALd\lefri{\altak,\altakp} = 0.
\end{align*}
It remains to compute \(\altak\) as a function of \(\ak\). If we consider the definition of \(\lambda\lefri{t}\), then%
\begin{align*}
\altak &= \lambda\lefri{0} =\Phi^{-1}\lefri{L_{\Phi\lefri{\ak}}\Phi\lefri{\sum_{i=1}^{n}\inak\phi_{i}\lefri{0}}}
\end{align*}%
and%
\begin{align}
\ak = \Phi^{-1}\lefri{L_{g_{k}^{-1}}g_{n}\lefri{0}} = \Phi^{-1}\lefri{L_{\Phi\lefri{\ak}^{-1}}\Phi\lefri{\altak}}. \label{CoordChange}
\end{align}%
This is simply a change of coordinates, and hence computing the discrete Euler Lagrange equations amounts to using the change of coordinates map to transform the algebra elements into the same chart. Thus, %
\begin{align}
D_{2}\ALd\lefri{\akm,\ak} &= \frac{\partial L_{d}}{\partial \ak} \nonumber\\
D_{1}\altALd\lefri{\altak,\altakp} &= \frac{\partial L_{d}}{\partial \ak}\frac{\partial \ak}{\partial \altak} \label{PMinus}
\end{align}%
where (\ref{CoordChange}) can be used to compute \(\frac{\partial \ak}{\partial \altak}\). An explicit example is presented in section \S \ref{CayCon}.

There are several features of this computation that should be noted. First, since we are considering specific choices of natural charts, we may think of \(\ak\) and \(\altak\) as corresponding to a specific coordinate choice, and hence it is natural to use standard partial derivatives as opposed to coordinate free notation. Second, because \(\altak\) is a function of \(\ak\), which is in turn a function of \(\inak\), this is still a root finding problem over \(\inak\), and hence may be solved concurrently with the internal stage Euler-Poincar\'{e} equations (\ref{InterEPEquations}).

\section{Convergence}

Thus far, we have discussed the construction of Lie group Galerkin variational integrators. Now we will prove several theorems related to their convergence. Unlike traditional integrators, we will achieve convergence in two different ways; the first will be the standard shortening of the time step \(\left[0,h\right]\), which we refer to as \(h\)-refinement. In practice, we refer to methods that achieve convergence through \(h\)-refinement as Lie group Galerkin variational integrators, after the method used to construct them. The second way that we achieve convergence is by enriching the function space \(\GalSpace\) and holding the time step \(\left[0,h\right]\) constant. Because enriching \(\GalSpace\) involves increasing the number of basis functions, and hence the value of \(n\), we refer to this as \(n\)-refinement. Because this approach of enriching the function space is inspired by classical spectral methods, as in \citet{Tr2000}, when we use \(n\)-refinement to achieve convergence we will refer to the the resulting method as a \emph{Lie group spectral variational integrator}.

\subsection{Geometric and Optimal Convergence}

Naturally, the goal of applying the spectral paradigm to the construction of Galerkin variational integrators is to construct methods which achieve geometric convergence. In this section, we will prove that under certain assumptions about the behavior of the Lagrangian and the approximation space, Lie group spectral variational integrators achieve geometric convergence. Additionally, the argument that establishes geometric convergence can be easily modified to show that the convergence of Lie group Galerkin integrators is, in a certain sense, optimal.

The proof of the rate of convergence Galerkin Lie group variational integrators is superficially similar to the proof of the rate of convergence of Galerkin variational integrators, which was established in \citep{HaLe2012}. The specific major difference is the need to quantify the error between two different curves on the Lie group. Unlike a normed vector space, there may not be a simple method of quantifying this error. For the moment, we will avoid this difficulty by simply assuming that the error between two curves that share a common point in a Lie group can be characterized through the error between curves in the Lie algebra. Specifically, we will make the following ``natural chart conditioning'' assumption:%
\begin{align}
  \GroupE{\LGFormf{g_{0}}{\lgc}}{\LGFormf{g_{0}}{\ilgc}} &\leq \GroupC \RMetric{\lgc - \ilgc}{\lgc - \ilgc}^{\frac{1}{2}} \label{CharCond1}\\
  \AlgeE{\ddt \LGFormf{g_{0}}{\lgc}}{\ddt \LGFormf{g_{0}}{\ilgc}} & \leq \AlgeC \RMetric{\dlgc - \dilgc}{\dlgc - \dilgc}^{\frac{1}{2}}  \label{CharCond2}\\
& \hspace{5em} + \AlgeGC \RMetric{\lgc - \ilgc}{\lgc - \ilgc}^{\frac{1}{2}} \nonumber
\end{align}%
for some functions \(\GroupE{\cdot}{\cdot}\) and \(\AlgeE{\cdot}{\cdot}\), which are chosen to measure the error in the Lie group and tangent bundle of the Lie group, respectively, and for some choice of Riemannian metric \(\RMetric{\cdot}{\cdot}\) on the Lie algebra. It is important to note that while the error function may be chosen to be the length of the geodesic curve that connects \(\LGForm{\xi}\) and \(\LGForm{\eta}\), there are other valid choices. This will greatly simplify error calculations; for example, in \S \ref{CayCon} we choose the error function to be the matrix two-norm, \(\MNorm{\cdot}{2}\), which is quickly and easily computed and will obey this inequality for the Riemannian metric we use.

\subsubsection{Optimal Convergence} We will begin by proving optimal convergence of Lie group Galerkin variational integrators. In this case, we take ``optimal'' to mean that the Lie group Galerkin variational integrator will converge at the same rate as the best possible approximation in the approximation space used to construct it.
%
%
\begin{theorem} \label{LG_OptConv} Given an interval \(\left[0,h\right]\), and a Lagrangian \(L:TG \rightarrow \mathbb{R}\), suppose that \(\trug\lefri{t}\) solves the Euler-Lagrange equations on that interval exactly. Furthermore, suppose that the exact solution \(\trug\lefri{t}\) falls within the range of the natural chart, that is:%
\begin{align*}
\trug\lefri{t} = \LGForm{\trula\lefri{t}}
\end{align*}
for some \(\trula\lefri{t} \in C^{2}\lefri{\left[0,h\right],\mathfrak{g}}\). For the function space \(\GalSpacef{g_{0}}\) and the quadrature rule \(\mathcal{G}\), define the Galerkin discrete Lagrangian \(\GDL{g_{0}}{g_{1}} \rightarrow \mathbb{R}\) as%
\begin{align}
\GDLh{g_{0}}{g_{1}}{h} = \ext_{\Ggalargsf{g_{0}}{g_{1}}} h \sum_{j=1}^{m}b_{j}L\lefri{g_{n}\lefri{c_{j}h},\dot{g}_{n}\lefri{c_{j}h}} = h \sum_{h=1}^{m}b_{j}L\lefri{\galgn\lefri{c_{j}h},\dgalgn\lefri{c_{j}h}} \label{LG_GalDiscAct}
\end{align}%
where \(\galgn\lefri{t}\) is the extremizing curve in \(\GalSpace\). If:
\begin{enumerate}
\item there exists an approximation \(\optnla \in \AFdFSpace\) such that,%
\begin{align*} 
\RMetric{\trula\lefri{t} - \optnla\lefri{t}}{\trula\lefri{t} - \optnla\lefri{t}}^{\frac{1}{2}} & \leq \ApproxC h^{n}\\
\RMetric{\dtrula\lefri{t} - \doptnla\lefri{t}}{\dtrula\lefri{t} - \doptnla\lefri{t}}^{\frac{1}{2}} &\leq \ApproxCi h^{n},
\end{align*}%
for some constants \(\ApproxC \geq 0\) and \(\ApproxCi \geq 0\) independent of \(h\),
\item the Lagrangian \(L\) is Lipschitz in the chosen norms in both its arguments, that is:%
\begin{align*}
\left|L\lefri{g_{1},\dot{g}_{1}} - L\lefri{g_{2},\dot{g}_{2}}\right| \leq \LagLipC \lefri{\GroupE{g_{1}}{g_{2}} + \AlgeE{\dot{g}_{1}}{\dot{g}_{2}}},
\end{align*}%
\item the chart function \(\Phi\) is well-conditioned in \(\GroupE{\cdot}{\cdot}\) and \(\AlgeE{\cdot}{\cdot}\), that is (\ref{CharCond1}) and (\ref{CharCond2}) hold,
\item for the quadrature rule \(\mathcal{G}\lefri{f} = h \sum_{j=1}^{m}b_{j}f\lefri{c_{j}h} \approx \int_{0}^{h} f\lefri{t}\dt\), there exists a constant \(\QuadC \geq 0\) such that,%
\begin{align*}
\left|\int_{0}^{h}L\lefri{g_{n}\lefri{t},\dot{g}_{n}\lefri{t}}\dt - h \sum_{j=1}^{m}b_{j}L\lefri{g_{n}\lefri{c_{j}h},\dot{g}_{n}\lefri{c_{j}h}}\right| \leq \QuadC h^{n+1}
\end{align*}%
for any \(g_{n}\lefri{t} = L_{g_{0}}\Phi\lefri{\xi\lefri{t}}\) where \(\xi \in \AFdFSpace\),
\item the stationary points of the discrete action and the continuous action are minimizers,
\end{enumerate}
then the variational integrator induced by \(\GDL{g_{0}}{g_{1}}\) has error \(\mathcal{O}\lefri{h^{n+1}}\).
\end{theorem}%
%
%
\begin{proof}%
We begin by rewriting the exact discrete Lagrangian and the Galerkin discrete Lagrangian:%
\begin{align*}
\left|\EDLh{g_{0}}{g_{1}}{h} - \GDLh{g_{0}}{g_{1}}{h}\right| &= \left|\int_{0}^{h}L\lefri{\trug,\dtrug} \dt - h\sum_{j=1}^{m}b_{j}L\lefri{\galgn\lefri{c_{j}h},\dgalgn\lefri{c_{j}h}}\right|,
\end{align*}%
where we have introduced \(\galgn\lefri{t}\), which is the stationary point of the local Galerkin action (\ref{LG_GalDiscAct}). We introduce the solution in the approximation space which takes the form \(\optgn\lefri{t} = \LGForm{\optnla\lefri{t}}\), and compare the action on the exact solution to the action on this solution:%
\begin{align*}
\left|\int_{0}^{h}L\lefri{\trug,\dtrug}\dt - \int_{0}^{h}L\lefri{\optgn,\doptgn}\dt\right| &= \left|\int_{0}^{h}L\lefri{\trug,\dtrug} - L\lefri{\optgn,\doptgn}\dt\right| \\
& \leq \int_{0}^{h}\left|L\lefri{\trug,\dtrug} - L\lefri{\optgn,\doptgn}\right|\dt.
\end{align*}%
Now, we use the Lipschitz assumption to establish the bound  
\begin{align*}
\int_{0}^{h}\left|L\lefri{\trug,\dtrug} - L\lefri{\optgn,\doptgn}\right|\dt & \leq \int_{0}^{h}\LagLipC \lefri{\GroupE{\trug}{\optgn} + \AlgeE{\dtrug}{\doptgn}}\dt \\
& = \int_{0}^{h} \LagLipC \left(\GroupE{\LGForm{\trula}}{\LGForm{\optnla}} + \right. \\
& \hspace{5em} \left. \AlgeE{\DLGForm{\trula}}{\DLGFormN{\optla}} \right) \dt,
\end{align*}%
and the chart conditioning assumptions to see
\begin{align*}
\int_{0}^{h}\left|L\lefri{\trug,\dtrug} - L\lefri{\optgn,\doptgn}\right|\dt & \leq \int_{0}^{h} \LagLipC \left(\GroupC\RMetric{\trula - \optnla}{\trula - \optnla}^{\frac{1}{2}} + \AlgeC\RMetric{\dtrula - \doptnla}{\dtrula - \doptnla}^{\frac{1}{2}} + \right. \\
& \hspace{5em} \left. \phantom{\RMetric{\dtrula - \doptnla}{\dtrula - \doptnla}^{\frac{1}{2}}} \AlgeGC\RMetric{\trula - \optla}{\trula - \optla}^{\frac{1}{2}}\right) \dt \\
& \leq \int_{0}^{h} \LagLipC \lefri{\GroupC \ApproxC h^{n} + \AlgeC \ApproxCi h^{n} + \AlgeGC \ApproxC h^{n}} \dt \\
& = \LagLipC\lefri{\lefri{\GroupC + \AlgeGC}\ApproxC + \AlgeC \ApproxCi}h^{n+1}.
\end{align*}%
This establishes a bound between the action evaluated on the exact discrete Lagrangian and the optimal solution in the approximation space, \(\optgn\). Considering the Galerkin discrete action,
\begin{align}
h\sum_{j=1}^{m} b_{j}L\lefri{\galgn,\galgn} & \leq h \sum_{j=1}^{m} b_{j}L\lefri{\optgn,\doptgn} \nonumber \\  
& \leq \int_{0}^{h}L\lefri{\optgn,\doptgn}\dt + \QuadC h^{n+1} \nonumber \\
& \leq \int_{0}^{h}L\lefri{\trug,\dtrug}\dt + \QuadC h^{n+1} + \LagLipC\lefri{\lefri{\GroupC + \AlgeGC}\ApproxC + \AlgeC\ApproxCi}h^{n+1} \label{OptIneq1}
\end{align}%
where we have used the assumption that the Galerkin approximation \(\galgn\) minimizes the Galerkin discrete action and the assumption on the accuracy of the quadrature. Now, using the fact that \(\trug\lefri{t}\) minimizes the action and that \(\GalSpace \subset \CG\),
\begin{align}
h\sum_{j=1}^{m}b_{j}L\lefri{\galgn,\dgalgn} & \geq \int_{0}^{h}L\lefri{\galgn,\dgalgn}\dt - \QuadC h^{n+1} \nonumber \\
& \geq \int_{0}^{h} L\lefri{\trug,\trug}\dt - \QuadC h^{n+1} \label{OptIneq2}
\end{align}%
Combining inequalities (\ref{OptIneq1}) and (\ref{OptIneq2}), we see that,%
\begin{align*}
\int_{0}^{h}L\lefri{\trug,\dtrug} \dt - \QuadC h^{n+1} \leq h \sum_{j=1}^{m}b_{j}L\lefri{\galgn,\dgalgn} \leq \int_{0}^{h} L\lefri{\trug,\dtrug} \dt + \QuadC h^{n+1} + \LagLipC\lefri{\lefri{\GroupC + \AlgeGC} \ApproxC + \AlgeC \ApproxCi}h^{n+1}
\end{align*}%
which implies%
\begin{align}
\left|\int_{0}^{h} L\lefri{\trug,\dtrug} \dt - h \sum_{j=1}^{m}L\lefri{\galgn,\dgalgn}\right| \leq \lefri{\QuadC + \LagLipC\lefri{\lefri{\GroupC + \AlgeGC}\ApproxC + \AlgeC\ApproxCi}} h^{n+1} \label{OptIneq3}
\end{align}%
The left hand side of (\ref{OptIneq3}) is exactly \(\left|\EDLh{g_{0}}{g_{1}}{h} - \GDLh{g_{0}}{g_{1}}{h}\right|\), and thus%
\begin{align*}
\left|\EDLh{g_{0}}{g_{1}}{h} - \GDLh{g_{0}}{g_{1}}{h}\right| \leq \OptC h^{n+1}.
\end{align*}%
where%
\begin{align*}
\OptC &= \QuadC + \LagLipC\lefri{\lefri{\GroupC + \AlgeGC}\ApproxC + \AlgeC\ApproxCi}.
\end{align*}
This states that the Galerkin discrete Lagrangian approximates the exact discrete Lagrangian with error \(\mathcal{O}\lefri{h^{n+1}}\), and by Theorem (\ref{ErrorThm}) this further implies that the Lagrangian update map, and hence the Lie group Galerkin variational integrator has error \(\mathcal{O}\lefri{h^{n+1}}\).
\end{proof}%
%
%
\subsubsection{Geometric Convergence} \label{SecGeoConv} Under similar assumptions, we can demonstrate that Lie group spectral variational integrators will converge geometrically with \(n\)-refinement, that is, enrichment of the function space \(\GalSpace\) as opposed to the shortening of the time step, \(h\).
\begin{theorem}\label{LG_GeoConv} Given an interval \(\left[0,h\right]\), and a Lagrangian \(L:TG \rightarrow \mathbb{R}\), suppose that \(\trug\lefri{t}\) solves the Euler-Lagrange equations on that interval exactly. Furthermore, suppose that the exact solution \(\trug\lefri{t}\) falls within the range of the natural chart, that is:%
\begin{align*}
\trug\lefri{t} = \LGForm{\trula\lefri{t}}
\end{align*}
for some \(\trula \in C^{2}\lefri{\left[0,h\right],\mathfrak{g}}\). For the function space \(\AFdFSpace\) and the quadrature rule \(\mathcal{G}\), define the Galerkin discrete Lagrangian \(\GDL{g_{0}}{g_{1}} \rightarrow \mathbb{R}\) as%
\begin{align}
\GDLh{g_{0}}{g_{1}}{h} = \ext_{\Ggalargsf{g_{0}}{g_{1}}} h \sum_{j=1}^{m}b_{j}L\lefri{g_{n}\lefri{c_{j}h},\dot{g}_{n}\lefri{c_{j}h}} = h \sum_{h=1}^{m}b_{j}L\lefri{\galgn\lefri{c_{j}h},\dgalgn\lefri{c_{j}h}} \label{DiscAct}
\end{align}%
where \(\galgn\lefri{t}\) is the extremizing curve in \(\GalSpace\). If:
\begin{enumerate}
\item there exists an approximation \(\optnla \in \AFdFSpace\) such that,%
\begin{align*} 
\RMetric{\trula - \optnla}{\trula - \optnla}^{\frac{1}{2}} & \leq \ApproxC \ApproxK^{n}\\
\RMetric{\dtrula - \doptnla}{\dtrula - \doptnla}^{\frac{1}{2}} &\leq \ApproxCi \ApproxK^{n},
\end{align*}%
for some constants \(\ApproxC \geq 0\) and \(\ApproxCi \geq 0\), \(0 < \ApproxK < 1\) independent of \(n\),
\item the Lagrangian \(L\) is Lipschitz in the chosen error norm in both its arguments, that is:%
\begin{align*}
\left|L\lefri{g_{1},\dot{g}_{1}} - L\lefri{g_{2},\dot{g}_{2}}\right| \leq \LagLipC \lefri{\GroupE{g_{1}}{g_{2}} + \AlgeE{\dot{g}_{1}}{\dot{g}_{2}}} 
\end{align*}%
\item the chart function \(\Phi\) is well conditioned in \(\GroupE{\cdot}{\cdot}\) and \(\AlgeE{\cdot}{\cdot}\), that is (\ref{CharCond1}) and (\ref{CharCond2}) hold,
\item there exists a sequence of quadrature rules \(\left\{\mathcal{G}_{n}\right\}_{n=1}^{\infty}\), \(\mathcal{G}_{n}\lefri{f} = h \sum_{j=1}^{m_{n}}\bnj f\lefri{\cnjh} \approx \int_{0}^{h} f\lefri{t}\dt\), and there exists a constant \(0 < \QuadK < 1\) independent of \(n\) such that,%
\begin{align*}
\left|\int_{0}^{h}L\lefri{g_{n}\lefri{t},\dot{g}_{n}\lefri{t}}\dt - h \sum_{j=1}^{m}b_{j}L\lefri{g_{n}\lefri{c_{j}h},\dot{g}_{n}\lefri{c_{j}h}}\right| \leq \QuadC \QuadK^{n}
\end{align*}%
for any \(g_{n}\lefri{t} = L_{g_{0}}\Phi\lefri{\xi\lefri{t}}\) where \(\xi \in \AFdFSpace\),
\item the stationary points of the discrete action and the continuous action are minimizers,
\end{enumerate}
then the variational integrator induced by \(\GDL{g_{0}}{g_{1}}\) has error \(\mathcal{O}\lefri{K^{n}}\).
\end{theorem}

The proof for this theorem is very similar to that for Theorem \ref{LG_OptConv}, using the modified assumptions in the obvious way. It would be tedious to repeat it here, but it has been included in the appendix for completeness.

These proofs may seem quite strong in their assumptions. However, as we shall see in \S \ref{CayCon}, for many Lagrangians, there are many reasonable choices of function spaces, natural chart functions, quadrature rules and error norms such that the assumptions are satisfied. We will specifically examine Lagrangians over \(\SoT\) of the form:%
\begin{align*}
L\lefri{R,\dot{R}} = \mbox{tr}\lefri{\dot{R}^{T}RJ_{d}R^{T}\dot{R}} - V\lefri{R}, 
\end{align*}%
which is the rigid body under the influence of a potential. We will show that for Lie group Galerkin variational integrators, stationary points of the discrete action are minimizers under a certain time step restriction. In addition we will give a specific construction of a Lie group Galerkin variational integrator for this type of problem, and demonstrate the expected convergence on several example problems. 

\subsection{Stationary Points are Minimizers}

A major assumption in both Theorem \ref{LG_OptConv} and Theorem \ref{LG_GeoConv} is that the stationary point of the discrete action is a minimizer. While in general this may not hold, we can show that given a time step restriction on \(h\), that this condition holds for problems on \(\SoT\) for Lagrangians of the form%
\begin{align*}
L\lefri{R,\dot{R}} = \mbox{tr}\lefri{\dot{R}^{T}RJ_{d}R^{T}\dot{R}} - V\lefri{R}. 
\end{align*}%
This includes a broad range of problems. Furthermore, we establish a similar result for problems in vector space in \citet{HaLe2012}, and it may be possible to combine these two results to include a large class of problems, including those that evolve on the special Euclidean group \(SE\lefri{3} = \mathbb{R}^{3} \ltimes \SoT\).

\begin{lemma} \label{StaMini} Consider a Lagrangian on \(\SoT\) of the form%
\begin{align*} 
L\lefri{R,\dot{R}} = \mbox{tr}\lefri{\dot{R}^{T}RJ_{d}R^{T}\dot{R}} - V\lefri{R}. 
\end{align*}
If a Lie group Galerkin variational integrator is constructed with \(\left\{\phi_{i}\right\}_{i=1}^{n}\) forming the basis for polynomials of degree \(n+1\) and the quadrature rule is of order at least \(2n + 2\), then the stationary points of the discrete action are minimizers.
\end{lemma}%
\begin{proof} We begin by noting that we can identify every element of \(\soT\), the Lie algebra associated with \(\SoT\), with an element of \(\mathbb{R}^{3}\) using the \emph{hat map} \(\hat{\cdot}:\mathbb{R}^{3} \rightarrow \soT\),
\begin{align}
\widehat{\left(\begin{array}{c} a\\b\\c \end{array}\right)} = \left(\begin{array}{ccc} 0 & -c & b \\ c & 0 & -a\\ -b & a & 0\end{array}\right). \label{hatmap}
\end{align}
Hence, it is natural to consider the discrete action as a function on \(H^{1}\lefri{\left[0,h\right],\mathbb{R}^{3}}\),%
\begin{align*}
\mathbb{S}_{d}\lefri{\xi\lefri{t},\dot{\xi}\lefri{t}} = h \sum_{j=1}^{m} b_{j} L\lefri{L_{\gk}\Phi\lefri{\hat{\xi}\lefri{\cjh}},\frac{d}{\dt}L_{\gk}\Phi\lefri{\hat{\xi}\lefri{\cjh}}}.
\end{align*}%
Let \(\lgstat\lefri{t}\) be the stationary point of \(\mathbb{S}_{d}\). Now, consider a perturbation to \(\lgstat\lefri{t}\), \(\lgstat\lefri{t} + \delta \xi\lefri{t}\). Since \(\lgstat\lefri{t}\) is the extremizer over curves \(\xi\lefri{t}\) subject to the constraints \(\xi\lefri{0} = \xi_{0}\), \(\xi\lefri{h} = \xi_{1}\), we know \(\delta \xi\lefri{0} = 0\) and \(\delta \xi \lefri{h} = 0\), but it is otherwise arbitrary. Hence, we consider an arbitrary perturbation \(\delta\xi\lefri{t} \in H^{1}_{0}\lefri{\left[0,h\right],\mathbb{R}^{3}}\). Since \(\mathbb{S}_{d}\) is a function on \(H^{1}\lefri{\left[0,h\right],\mathbb{R}^{3}}\), we can Taylor expand around the stationary point:%
\begin{align*}
\mathbb{S}_{d}\lefri{\lgstat + \delta \xi,\dlgstat + \delta \dot{\xi}} &= \mathbb{S}_{d}\lefri{\lgstat,\dlgstat} + D\mathbb{S}_{d}\lefri{\lgstat,\dlgstat}\left[\lefri{\delta \xi, \delta \dot{\xi}}\right] + \frac{1}{2}D^{2} \mathbb{S}_{d}\lefri{\eta,\dot{\eta}}\left[\lefri{\delta \xi,\delta \dot \xi}\right]\left[\lefri{\delta \xi,\delta \dot \xi}\right]
\end{align*}%
where \(\eta\lefri{t} = \lambda\lefri{t}\xi_{0}\lefri{t} + \lefri{1 - \lambda\lefri{t}}\delta \xi\lefri{t}\) for some \(\lambda\lefri{t}: \left[0,h\right] \rightarrow \left[0,1\right]\) and \(D \mathbb{S}_{d}\), \(D^{2} \mathbb{S}_{d}\) are the first and second Frechet derivative of \(\mathbb{S}_{d}\), respectively. Thus
\begin{align*}
\mathbb{S}_{d}\lefri{\lgstat + \delta \xi,\dlgstat + \delta \dot{\xi}} - \mathbb{S}_{d}\lefri{\lgstat,\dlgstat} &=  D \mathbb{S}_{d}\lefri{\lgstat,\dlgstat}\left[\lefri{\delta \xi, \delta \dot{\xi}}\right] + \frac{1}{2}D^{2} \mathbb{S}_{d}\lefri{\eta,\dot{\eta}}\left[\lefri{\delta \xi,\delta \dot \xi}\right]\left[\lefri{\delta \xi,\delta \dot \xi}\right].
\end{align*}%
Now, note that%
\begin{align*}
D \mathbb{S}_{d}\lefri{\lgstat,\dlgstat}\left[\lefri{\delta \xi, \delta \dot{\xi}}\right] = 0
\end{align*}
is exactly the stationarity conditions for the internal stage discrete Euler-Poincar\'{e} equations. Thus,%
\begin{align*}
\mathbb{S}_{d}\lefri{\lgstat + \delta \xi,\dlgstat + \delta \dot{\xi}} - \mathbb{S}_{d}\lefri{\lgstat,\dlgstat} &= \frac{1}{2}D^{2} \mathbb{S}_{d}\lefri{\eta,\dot{\eta}}\left[\lefri{\delta \xi,\delta \dot \xi}\right]\left[\lefri{\delta \xi,\delta \dot \xi}\right].
\end{align*}%
We will examine \(D^{2} \mathbb{S}_{d}\). The second Frechet derivative of the discrete action is given by
\begin{align*}
D^{2}\mathbb{S}_{d}\lefri{\xi,\dot{\xi}}\left[\lefri{\delta \xi_{a},\delta \dot{\xi}_{a}}\right]\left[\lefri{\delta \xi_{b}, \delta \dot{\xi}_{b}}\right] = h\sum_{j=1}^{m}\bj \nabla^{2}L\lefri{\xi,\dot{\xi}}\left[\lefri{\delta \xi_{a}\lefri{\cjh},\delta \dot{\xi}_{a}\lefri{\cjh}}\right]\left[\lefri{\delta \xi_{b}\lefri{\cjh}, \delta \dot{\xi}_{b}\lefri{\cjh}}\right].
\end{align*}
In order to examine the second Frechet derivative, we must examine the Hessian of the Lagrangian. We will do this term-wise. The Lagrangian has the form%
\begin{align*}
L\lefri{\xi,\dot{\xi}} = K\lefri{\xi,\dot{\xi}} - V\lefri{\xi}
\end{align*}
where%
\begin{align*}
K\lefri{\lgc,\dlgc} = \dR{\lgc}^{T}\R{\lgc}J_{d}\R{\lgc}^{T}\dR{\lgc}. 
\end{align*}%
is the kinetic energy and \(V\) is the potential energy. Considering \(K\), note that
\begin{align*}
\R{\lgc}^{T}\dR{\lgc} &= \Phi\lefri{\lgc}^{T}\nabla\Phi\lefri{\lgc}\dlgc
\end{align*}%
and hence as a function of \(\dlgc\),%
\begin{align*}
K\lefri{\dlgc} &= \dlgc^{T}\nabla\Phi\lefri{\lgc}^{T}\Phi\lefri{\lgc}J_{d}\Phi\lefri{\lgc}^{T}\nabla\Phi\lefri{\lgc}\dlgc.
\end{align*}
\(J_{d}\) is a diagonal matrix with \(\lefri{J_{1},J_{2},J_{3}}\) on the diagonal, and because \(\Phi\lefri{\lgc}\) is an orthogonal matrix, \(\Phi\lefri{\lgc}J_{d}\Phi\lefri{\lgc}^{T}\) has the eigenvalues \(\lefri{J_{1},J_{2},J_{3}}\). Furthermore, \(\Phi\lefri{\cdot}\) is a diffeomorphism, which implies \(\nabla \Phi\lefri{\cdot}\) is non-singular, so %
\begin{align*}
\dlgc^{T}\nabla\Phi\lefri{\lgc}^{T}\Phi\lefri{\lgc}J_{d}\Phi\lefri{\lgc}^{T}\nabla\Phi\lefri{\lgc}\dlgc \geq J_{\min} \MNorm{\nabla \Phi\lefri{\lgc} \dlgc}{2}^{2} \geq J_{\min}\left|\sigma_{\min}\lefri{t}\right|\MNorm{\dlgc}{2}^{2}
\end{align*}
where \(J_{\min} = \min\lefri{\left\{J_{1},J_{2},J_{3}\right\}}\) and \(\sigma_{\min}\lefri{t}\) is the value of \(\nabla \Phi\lefri{\dlgc}\) with smallest magnitude. Since \(\left|\sigma_{\min}\lefri{t}\right|\) is a continuous function of \(t\) and \(\left|\sigma_{\min}\lefri{t}\right| > 0\) for all \(t\)  over the compact interval \(\left[0,h\right]\), there exists a constant \(\sigmaC > 0\) such that \(\left|\sigma_{\min}\lefri{t}\right| > \sigmaC\) for all \(t \in \left[0,h\right]\). Finally, we note%
\begin{align*}
\ddKdddq\lefri{\ilgc,\dilgc} \left[\delta \dot{\xi}_{a}\right]\left[\delta \dot{\xi}_{b}\right] = 2\delta\dot{\xi}^{T}_{a}\nabla\Phi\lefri{\ilgc}^{T}\Phi\lefri{\ilgc}J_{d}\Phi\lefri{\ilgc}^{T}\nabla\Phi\lefri{\ilgc}\delta\dot{\xi}_{b},
\end{align*}%
and hence
\begin{align}
\ddKdddq\lefri{\ilgc,\dilgc}\left[\delta \dot{\xi}\right]\left[\delta \dot{\xi}\right] \geq 2J_{\min}\sigmaC \delta \dot{\xi}^{T}\delta\dot{\xi}. \label{MinBound_ddq}
\end{align}%
Now, considering the full term \(\nabla^{2}K\lefri{\ilgc,\dilgc} \left[\lefri{\delta \xi, \delta \dot{\xi}}\right] \left[\lefri{\delta \xi, \delta \dot{\xi}}\right]\), we see that:%
\begin{align}
 \nabla^{2}K\lefri{\ilgc,\dilgc} \left[\lefri{\delta \xi, \delta \dot{\xi}}\right] \left[\lefri{\delta \xi, \delta \dot{\xi}}\right] &= \ddKddq\lefri{\ilgc,\dilgc} \left[\delta \xi\right] \left[\delta \xi \right] + 2 \ddKddqdq\lefri{\ilgc,\dilgc} \left[\delta \xi\right] \left[\delta \dot{\xi}\right]  \label{KineticExpansion}\\
&\hspace{5em} +\ddKdddq\lefri{\ilgc,\dilgc} \left[\delta \dot{\xi}\right]\left[\delta \dot{\xi}\right], \nonumber
\end{align}%
where we have made use of the symmetry of mixed second derivatives. \(K\) is smooth in all of its components, and hence there exists \(\mixedC > 0\), \(\ddlgC > 0\) such that%
\begin{align}
\ddKddqdq\lefri{\ilgc,\dilgc} \left[\delta \xi\right]\left[\delta \dot{\xi}\right] &\geq -\mixedC \delta \xi^{T} \delta \dot{\xi} \label{MaxBound_dqddq}\\
\ddKddq\lefri{\ilgc,\dilgc} \left[\delta \xi\right]\left[\delta \xi\right] &\geq -\ddlgC \delta \xi^{T} \delta \xi \label{MaxBound_dq}.
\end{align}%
Defining \(\derivC = 2J_{\min}C_{\sigma}\), inserting (\ref{MinBound_ddq}), (\ref{MaxBound_dqddq}), and (\ref{MaxBound_dq}) into (\ref{KineticExpansion}) gives%
\begin{align*}
 \nabla^{2}K\lefri{\ilgc,\dilgc} \left[\lefri{\delta \xi, \delta \dot{\xi}}\right] \left[\lefri{\delta \xi, \delta \dot{\xi}}\right] &\geq \derivC\deldlgc^{T}\deldlgc - \mixedC\deldlgc^{T}\dellgc - \ddlgC\dellgc^{T}\dellgc \\
&= \frac{\derivC}{2}\deldlgc^{T}\deldlgc + \frac{\derivC}{2}\deldlgc^{T}\deldlgc - \mixedC\deldlgc^{T}\dellgc \\
& \hspace{5em} - \ddlgC\dellgc^{T}\dellgc.
\end{align*}%
Completing the square, we see that%
\begin{align*}
&\frac{\derivC}{2}\deldlgc^{T}\deldlgc + \frac{\derivC}{2}\deldlgc^{T}\deldlgc - \mixedC\deldlgc^{T}\dellgc - \ddlgC\dellgc^{T}\dellgc  \\
=&\frac{\derivC}{2}\deldlgc^{T}\deldlgc + \lefri{\frac{\sqrt{\derivC}}{\sqrt{2}}\deldlgc - \frac{\sqrt{2}\mixedC}{2\sqrt{\derivC}}\dellgc}^{T}\lefri{\frac{\sqrt{\derivC}}{\sqrt{2}}\deldlgc - \frac{\sqrt{2}\mixedC}{2\sqrt{\derivC}}\dellgc} \\
& \hspace{5em} - \lefri{\ddlgC + \frac{\mixedC^{2}}{2\derivC}}\dellgc^{T}\dellgc\\
=&\frac{\derivC}{2}\MNorm{\deldlgc}{2}^{2} + \MNorm{\frac{\sqrt{\derivC}}{\sqrt{2}}\deldlgc - \frac{\sqrt{2}\mixedC}{2\sqrt{\derivC}}\dellgc}{2}^{2} - \lefri{\ddlgC + \frac{\mixedC^{2}}{2\derivC}}\MNorm{\dellgc}{2}^{2}.
\end{align*}
Making use of the trivial bound that for any \(a, b \in \mathbb{R}^{3}\), \(\MNorm{a - b}{2}^{2} \geq 0\), we see%
\begin{align*}
&\frac{C_{p}}{2}\MNorm{\deldlgc}{2}^{2} + \MNorm{\frac{\sqrt{\derivC}}{\sqrt{2}}\deldlgc - \frac{\sqrt{2}\mixedC}{2\sqrt{\derivC}}\dellgc}{2}^{2} -\lefri{\ddlgC + \frac{\mixedC^{2}}{2\derivC}}\MNorm{\dellgc}{2}^{2}\\
\geq &\frac{\derivC}{2}\MNorm{\deldlgc}{2}^{2} - \lefri{\ddlgC + \frac{\mixedC^{2}}{2\derivC}}\MNorm{\dellgc}{2}^{2}\\
=& \dboundC \deldlgc^{T}\deldlgc - \cboundC \dellgc^{T}\dellgc
\end{align*}%
for constants \(\dboundC > 0\), \(\cboundC > 0\), where
\begin{align*}
\dboundC &= \frac{\derivC}{2}\\
\cboundC &= \ddlgC + \frac{\mixedC^{2}}{2\derivC}.
\end{align*}
This bound allows us to conclude%
\begin{align}
 \nabla^{2}K\lefri{\ilgc,\dilgc} \left[\lefri{\delta \xi, \delta \dot{\xi}}\right] \left[\lefri{\delta \xi, \delta \dot{\xi}}\right] &\geq \dboundC\dlgc^{T}\dlgc - \cboundC\lgc^{T}\lgc. \label{KinBound}
\end{align}%
We now turn our attention to the potential term, \(V\lefri{\R{\lgc}}\). Since \(V\) and \(\R{\cdot}\) are both smooth we know that the second partial derivatives of \(V\lefri{\R{\cdot}}\) are bounded, and since \(V\) does not depend on \(\dlgc\),%
\begin{align}
\nabla^{2} V\lefri{R\lefri{\ilgc}}\left[\lefri{\dellgc,\deldlgc}\right]\left[\lefri{\dellgc,\deldlgc}\right] &\leq \potentC \dellgc^{T}\dellgc \label{PotentBound}
\end{align}%
for a constant \(\potentC\). Thus, combining (\ref{KinBound}) and (\ref{PotentBound}), we can bound \(\nabla^{2}\mathbb{S}_{d}\),
\begin{align*}
&\nabla^{2}\mathbb{S}_{d}\lefri{\ilgc,\dilgc}\left[\lefri{\dellgc,\deldlgc}\right]\left[\lefri{\dellgc,\deldlgc}\right] \geq  \\
& \hspace{5em} h\sum_{j=1}^{m} \bj \dboundC \qdeldlgc^{T}\qdeldlgc - \lefri{\cboundC + \potentC} \qdellgc^{T}\qdellgc. 
\end{align*}%
Since, by assumption, \(\dellgc\) and \(\deldlgc\) are polynomials of degree at most \(n+1\), \(\dellgc^{T}\dellgc\) and \(\deldlgc^{T}\deldlgc\) is a polynomial of degree at most \(2n + 2\), so the quadrature rule is exact, and thus%
\begin{align}
&h \sum_{j=1}^{m}\bj \dboundC \qdeldlgc^{T}\qdeldlgc - \lefri{\cboundC + \potentC}\qdellgc^{T}\qdellgc = \label{QuadtoInt} \\
& \hspace{5em} \dboundC \int_{0}^{h}\deldlgc^{T}\deldlgc\dt - \lefri{\cboundC + \potentC}\int_{0}^{h}\dellgc^{T}\dellgc \dt. \nonumber
\end{align}%
\(\delta \lgc \in H_{0}^{1}\lefri{\left[0,h\right],\mathbb{R}^{3}}\), so we can apply the Poincar\'{e} inequality to see%
\begin{align*}
&\dboundC \int_{0}^{h}\deldlgc^{T}\deldlgc\dt - \lefri{\cboundC + \potentC}\int_{0}^{h}\dellgc^{T}\dellgc \dt \\
\geq &\frac{\dboundC \pi}{h^{2}}\int_{0}^{h}\dellgc^{T}\dellgc \dt -\lefri{\cboundC + \potentC} \int_{0}^{h}\dellgc^{T}\dellgc \dt \\
=& \lefri{\frac{\dboundC \pi}{h^{2}} - \lefri{\cboundC + \potentC}}\int_{0}^{h} \lgc^{T}\lgc \dt 
\end{align*}%
which is positive so long as \(h < \sqrt{\frac{\dboundC \pi}{\cboundC + \potentC}}\). Thus, given that \(h < \sqrt{\frac{\dboundC \pi}{\cboundC + \potentC}}\), for arbitrary \(\lefri{\dellgc,\deldlgc}\)%
\begin{align*}
\mathbb{S}_{d}\lefri{\lgstat\lefri{t} + \dellgc, \dlgstat\lefri{t} + \deldlgc} - \mathbb{S}_{d}\lefri{\lgstat\lefri{t},\dlgstat\lefri{t}} > 0%
\end{align*}
which demonstrates that \(\lefri{\lgstat\lefri{t},\dlgstat\lefri{t}}\) minimizes the action.
\end{proof}

It should be noted that the only use of the assumption that the approximation space is polynomials of order at least \(n\) is when we use the order of the quadrature rule to change the quadrature to the exact integral (\ref{QuadtoInt}). Thus, this proof can easily be generalized to other approximation spaces, so long as the quadrature rule used is exact for the product of any two elements of the approximation space and the product of any two derivatives of the elements of the approximation space.

\subsection{Convergence of Galerkin Curves}

Lie group Galerkin variational integrators require the construction of a curve%
\begin{align*}
\galgn\lefri{t} \in \GalSpacef{\gk}
\end{align*}%
such that%
\begin{align*}
\galgn\lefri{t} = \argext_{\Ggalargsf{g_{k}}{g_{k+1}}} \sum_{j=1}^{m} \bj L\lefri{g_{n}\lefri{\cjh},\dot{g}_{n}\lefri{\cjh}}.
\end{align*}%
This curve, which we shall refer to as the \emph{Galerkin curve}, is a finite-dimension approximation to the true solution of the Euler-Poincar\'{e} equations over the interval \(\left[0,h\right]\). For the one-step map, we are only concerned with the right endpoint of the Galerkin curve, as%
\begin{align*}
g_{k+1} = g_{n}\lefri{h}.
\end{align*}%
However, the curve itself has excellent approximation properties as a continuous approximation to the solution of the Euler-Poincar\'{e} equations over the interval \(\left[0,h\right]\). Because Lie group Galerkin variational integrators are capable of taking very large time steps, the dynamics during these time steps may be of interest, and hence the quality of the approximation by these Galerkin curves is also of particular interest.

Ideally, these curves would have the same order of error as the one-step map. Unfortunately, we can only establish error estimates with lower orders of approximation. We established similar results in the vector space case, see \citet{HaLe2012}, and observed that at high enough accuracy, there is indeed greater error in the Galerkin curve than the one-step map. However, when comparing these curves to the true solution, typically the error introduced by the inaccuracies in \(\lefri{g_{k},g_{k+1}}\) dominates the error from the Galerkin curve, and thus this lower rate of convergence is not observable in practice.

Before we formally establish the rates of convergence for the Galerkin curves, we will briefly review the norms we will use in our theorems and proofs. First, recall the \(L_{p}\) norm for functions over the interval \(\left[0,h\right]\) given by%
\begin{align*}
\LNorm{f}{p} = \lefri{\int_{0}^{h}\left|f\right|^{p}\dt}^{\frac{1}{p}}
\end{align*}%
and next, the Sobolev norm \(\SobNorm{\cdot}{p}\) for functions on the interval \(\left[0,h\right]\), given by:%
\begin{align*}
\SobNorm{f}{p} = \lefri{\LNorm{f}{p}^{p} + \LNorm{\dot{f}}{p}^{p}}^{\frac{1}{p}}.
\end{align*}
Also, note that for curves \(\lgc \in \mathfrak{g}\), \(\left|\lgc\right| = \RMetric{\lgc}{\lgc}^{\frac{1}{2}}\). We will make extensive use of these definitions in the next three theorems.

\begin{theorem}\label{LG_GeoConv_GalCurves}

Under the same assumptions as Theorem \ref{LG_GeoConv}, consider the action as a function of the local left trivialization of the Lie group curve and its derivative,%
\begin{align*}
\mathfrak{S}_{\mathfrak{g}}\lefri{\trula\lefri{t},\dtrula\lefri{t}} = \int_{0}^{h}L\lefri{L_{g}\Phi\lefri{\trula\lefri{t}},\ddt L_{g}\Phi\lefri{\trula\lefri{t}}} \dt,
\end{align*}%
where \(L_{g}\Phi\lefri{\trula\lefri{t}}\) satisfies the Euler-Poincar\'{e} equations exactly. If at \(\lefri{\trula\lefri{t},\dtrula\lefri{t}}\) the action \(\mathfrak{S}_{\mathfrak{g}}\lefri{\cdot,\cdot}\) is twice Frechet differentiable and the second Frechet derivative is coercive in variations of the Lie algebra, that is, %
\begin{align*}
\left|D^{2}\mathfrak{S}_{\mathfrak{g}}\lefri{\lefri{\trula\lefri{t},\dtrula\lefri{t}}}\left[\lefri{\delta \lgc,\delta \dlgc}\right]\left[\lefri{\delta \lgc, \delta \dlgc}\right]\right| \geq \CoerC \SobNorm{\delta \lgc}{1}^{2}
\end{align*}%
for all \(\dellgc \in \HoA\), then if the one-step map has error \(\mathcal{O}\lefri{K^{n}}\), the Galerkin curves have error \(\mathcal{O}\lefri{\sqrt{K}^{n}}\) in Sobolev norm \(\SobNorm{\cdot}{1}\).
\end{theorem}

\begin{proof}
We start with the bound (\ref{GeoIneq3}), given at the end of the proof of Theorem \ref{LG_GeoConv} in the appendix,%
\begin{align*}
\left|\EDLh{\gk}{\gkp}{h} - \GDLh{\gk}{\gkp}{n}\right| \leq \SpecC \SpecK^{n},
\end{align*}
and expand using the definitions of \(\EDLh{\gk}{\gkp}{h}\) and \(\GDLh{\gk}{\gkp}{n}\),%
\begin{align*}
\SpecC \SpecK^{n} &\geq \left|\EDLh{\gk}{\gkp}{h} - \GDLh{\gk}{\gkp}{n}\right| \\
& \geq \left|\int_{0}^{h}L\lefri{\LGForm{\trula\lefri{t}},\ddt \LGForm{\trula\lefri{t}}}\dt - \sum_{j=1}^{m}\bj L\lefri{\LGForm{\galla\lefri{t}},\ddt\LGForm{\galla\lefri{t}}}\right| \\
& \geq \left|\int_{0}^{h}L\lefri{\LGForm{\trula\lefri{t}},\ddt \LGForm{\trula\lefri{t}}}\dt - \int_{0}^{h}L\lefri{\LGForm{\galla\lefri{t}},\ddt\LGForm{\galla\lefri{t}}}\dt\right| - \QuadC \QuadK^{n} \\
& = \left|\mathfrak{S}_{\mathfrak{g}}\lefri{\trula\lefri{t},\dtrula\lefri{t}} - \mathfrak{S}_{\mathfrak{g}}\lefri{\galla\lefri{t},\dgalla\lefri{t}}\right| - \QuadC \QuadK^{n},
\end{align*}%
and since \(\QuadK \leq \SpecK\), this implies%
\begin{align}
\lefri{\SpecC + \QuadC}\SpecK^{n} \geq \left|\mathfrak{S}_{\mathfrak{g}}\lefri{\trula\lefri{t},\dtrula\lefri{t}} - \mathfrak{S}_{\mathfrak{g}}\lefri{\galla\lefri{t},\dgalnla\lefri{t}}\right|. \label{LG_GalCuIneq1}
\end{align}%
We now Taylor expand around the exact solution \(\lefri{\trula\lefri{t},\dtrula\lefri{t}}\)%
\begin{align}
\mathfrak{S}_{\mathfrak{g}}\lefri{\galla\lefri{t},\dgalla\lefri{t}} =& \mathfrak{S}_{\mathfrak{g}}\lefri{\trula\lefri{t},\dtrula\lefri{t}} + D\mathfrak{S}_{\mathfrak{g}}\lefri{\trula\lefri{t},\dtrula\lefri{t}}\left[\trula\lefri{t} - \galla\lefri{t},\dtrula\lefri{t} - \dgalla\lefri{t}\right] \label{LG_GalCuIneq2} \\
& \hspace{1em} + \frac{1}{2}D^{2}\mathfrak{S}_{\mathfrak{g}}\lefri{\nu\lefri{t},\dot{\nu}\lefri{t}}\left[\lefri{\trula\lefri{t} - \galla\lefri{t},\dtrula\lefri{t} - \dgalla\lefri{t}}\right]\left[\lefri{\trula\lefri{t} - \galla\lefri{t},\dtrula\lefri{t} - \dgalla\lefri{t}}\right], \nonumber
\end{align}%
where \(\nu\lefri{t}\) is a curve in \(\mathfrak{g}\). Now, note that \(D\mathfrak{S}_{\mathfrak{g}}\lefri{\trula\lefri{t},\dtrula\lefri{t}} = 0\) is exactly the stationarity condition of the Euler-Poincar\'{e} equations. Thus, inserting (\ref{LG_GalCuIneq2}) into (\ref{LG_GalCuIneq1}) yields%
\begin{align}
\lefri{\SpecC + \QuadC}\SpecK^{n} &\geq \frac{1}{2}\left|D^{2}\mathfrak{S}_{\mathfrak{g}}\lefri{\nu\lefri{t},\dot{\nu}\lefri{t}}\left[\lefri{\trula\lefri{t} - \galla\lefri{t},\dtrula\lefri{t} - \dgalla\lefri{t}}\right]\left[\lefri{\trula\lefri{t} - \galla\lefri{t},\dtrula\lefri{t} - \dgalla\lefri{t}}\right]\right| \nonumber\\
& \geq \frac{\CoerC}{2} \SobNorm{\trula\lefri{t} - \galla\lefri{t}}{1}^{2} \label{LG_GalCuIneq3}
\end{align}%
where we have made use of the coercivity of the second derivative of the action. Simplifying (\ref{LG_GalCuIneq3}) yields
\begin{align*}
\sqrt{\frac{2\lefri{\SpecC + \QuadC}}{\CoerC}}\sqrt{\SpecK}^{n} \geq \SobNorm{\trula\lefri{t} - \galla\lefri{t}}{1},
\end{align*}%
which establishes convergence in the Sobolev norm.
\end{proof}%
Just as we proved an order optimality theorem, Theorem \ref{LG_OptConv}, that was analogous to the geometric convergence theorem, Theorem \ref{LG_GeoConv}, we can establish an analogous convergence theorem for Galerkin curves with \(h\)-refinement.%
\begin{theorem} \label{LG_OptConv_GalCurves}

Under the same assumptions as Theorem \ref{LG_OptConv}, consider the action as a function of the local left trivialization of the Lie group curve and its derivative,%
\begin{align*}
\mathfrak{S}_{\mathfrak{g}}\lefri{\trula,\dtrula} = \int_{0}^{h}L\lefri{L_{g}\Phi\lefri{\trula},\ddt L_{g}\Phi\lefri{\trula}} \dt.
\end{align*}%
If at \(\lefri{\trula,\dtrula}\) the action \(\mathfrak{S}_{\mathfrak{g}}\lefri{\cdot,\cdot}\) is twice Frechet differentiable and the second Frechet derivative is coercive in variations of the Lie algebra as in Theorem \ref{LG_GeoConv_GalCurves}, then if the one-step map has error \(\mathcal{O}\lefri{h^{n+1}}\), then the Galerkin curves have error \(\mathcal{O}\lefri{h^{\frac{n+1}{2}}}\) in the Sobolev norm \(\SobNorm{\cdot}{1}\).
\end{theorem}%
The proof Theorem \ref{LG_OptConv_GalCurves} is nearly identical to that of Theorem \ref{LG_GeoConv_GalCurves}, the only difference being that the bounds containing \(\SpecK^{n}\) are replaced with bounds containing \(h^{n+1}\) in the obvious way.

Like the assumption that the stationary point of the discrete action is a minimizer in Theorems \ref{LG_GeoConv} and \ref{LG_OptConv}, the assumption that the second Frechet derivative of the action is coercive might seem quite strong. However, we can show that for Lagrangians on \(\SoT\) of the form%
\begin{align*}
L\lefri{R,\dot{R}} = \mbox{tr}\lefri{\dot{R}^{T}RJ_{d}R^{T}\dot{R}} - V\lefri{R},
\end{align*}%
the second Frechet derivative of the action is coercive, subject to a time-step restriction on \(h\).
\begin{lemma} \label{LG_CoerAction}
For Lagrangians on \(\SoT\) of the form%
\begin{align*}
L\lefri{R,\dot{R}} = \mbox{tr}\lefri{\dot{R}^{T}RJ_{d}R^{T}\dot{R}} - V\lefri{R},
\end{align*}%
there exists a \(C > 0\) such that for \(h < C\), the second Frechet derivative of \(\mathfrak{S}_{\mathfrak{g}}\lefri{\cdot,\cdot}\) at \(\lefri{\trula\lefri{t},\dtrula\lefri{t}}\) is coercive on the interval \(\left[0,h\right]\).
\end{lemma}%
\begin{proof}

First, we note that for this Lagrangian%
\begin{align*}
&D^{2}\mathfrak{S}_{\mathfrak{g}}\lefri{\trula\lefri{t},\dtrula\lefri{t}}\left[\lefri{\dellgc,\deldlgc}\right]\left[\lefri{\dellgc,\deldlgc}\right] \\
&= \int_{0}^{h}\nabla^{2}L\lefri{\LGForm{\trula\lefri{t}},\ddt \LGForm{\trula\lefri{t}}}\left[\lefri{\dellgc,\deldlgc}\right]\left[\lefri{\dellgc,\deldlgc}\right] \dt
\end{align*}%
From the proof of Lemma \ref{StaMini}, we know that%
\begin{align*}
&\nabla^{2}L\lefri{\LGForm{\trula\lefri{t}},\ddt \LGForm{\trula\lefri{t}}}\left[\lefri{\dellgc, \deldlgc}\right]\left[\lefri{\dellgc, \deldlgc}\right] \\
&\geq  \dboundC \deldlgc^{T}\deldlgc - \lefri{\cboundC + \potentC} \dellgc^{T} \dellgc\\
&= \frac{\dboundC}{2} \deldlgc^{T}\deldlgc + \frac{\dboundC}{2} \deldlgc^{T}\deldlgc - \lefri{\cboundC + \potentC} \dellgc^{T} \dellgc,
\end{align*}%
and hence%
\begin{align}
&D^{2}\mathfrak{S}_{\mathfrak{g}}\lefri{\trula\lefri{t},\dtrula\lefri{t}}\left[\lefri{\dellgc,\deldlgc}\right]\left[\lefri{\dellgc,\deldlgc}\right] \nonumber \\
&\geq  \int_{0}^{h}\frac{\dboundC}{2} \deldlgc^{T}\deldlgc \dt +  \int_{0}^{h}\frac{\dboundC}{2} \deldlgc^{T}\deldlgc - \lefri{\cboundC + \potentC} \dellgc^{T} \dellgc \dt \label{CoerTerm}.
\end{align}%
Applying Poincar\'{e}'s inequality, we see that%
\begin{align}
&\int_{0}^{h} \frac{\dboundC}{2} \deldlgc^{T}\deldlgc \dt - \int_{0}^{h} \lefri{\cboundC + \potentC} \dellgc^{T}\dellgc \dt \nonumber \\
&\geq  \frac{\dboundC\pi^{2}}{2h^{2}} \int_{0}^{h}\dellgc^{T}\dellgc \dt - \lefri{\cboundC + \potentC} \int_{0}^{h}\dellgc^{T}\dellgc \dt \nonumber\\
& = \lefri{\frac{\dboundC\pi^{2}}{2h^{2}} - \lefri{\cboundC + \potentC}}\int_{0}^{h} \dellgc^{T}\dellgc \dt. \label{CoerPoinc}
\end{align}%
Replacing the last two terms in (\ref{CoerTerm}) with (\ref{CoerPoinc}), we see%
\begin{align*}
&D^{2}\mathfrak{S}_{\mathfrak{g}}\lefri{\lefri{\trula,\dtrula}}\left[\lefri{\dellgc,\deldlgc}\right]\left[\lefri{\dellgc,\deldlgc}\right] \\ 
&\geq \frac{\dboundC}{2} \int_{0}^{h}\deldlgc^{T}\deldlgc \dt + \lefri{\frac{\dboundC\pi^{2}}{2h^{2}} - \lefri{\cboundC + \potentC}}\int_{0}^{h} \dellgc^{T}\dellgc \dt \\
&= \frac{\dboundC}{2}\LNorm{\deldlgc}{2}^{2} + \lefri{\frac{\dboundC\pi^{2}}{2h^{2}} - \lefri{\cboundC + \potentC}}\LNorm{\dellgc}{2}^{2}.
\end{align*}%
We now apply H\"{o}lder's inequality%
\begin{align*}
\LNorm{fg}{1} \leq \LNorm{f}{2} \LNorm{g}{2}
\end{align*}
to derive the bounds%
\begin{align*}
\LNorm{\deldlgc}{1} &\leq \sqrt{h} \LNorm{\deldlgc}{2}\\
\LNorm{\dellgc}{1} &\leq \sqrt{h} \LNorm{\dellgc}{2},
\end{align*}
and hence,
\begin{align*}
&D^{2}\mathfrak{S}_{\mathfrak{g}}\lefri{\lefri{\trula,\dtrula}}\left[\lefri{\dellgc,\deldlgc}\right]\left[\lefri{\dellgc,\deldlgc}\right] \\
&\geq \frac{\dboundC }{2h}\LNorm{\deldlgc}{1}^{2} + \frac{1}{h}\lefri{\frac{\dboundC\pi^{2}}{2h^{2}} - \lefri{\cboundC + \potentC}}\LNorm{\dellgc}{1}^{2} \\
&\geq \min\lefri{\frac{\dboundC}{2h},\frac{1}{h}\lefri{\frac{\dboundC \pi^{2}}{2h^{2}} - \lefri{\cboundC + \potentC}}}\lefri{\LNorm{\dellgc}{1}^{2} + \LNorm{\deldlgc}{1}^{2}} \\
&\geq \min\lefri{\frac{\dboundC}{2h},\frac{1}{h}\lefri{\frac{\dboundC \pi^{2}}{2h^{2}} - \lefri{\cboundC + \potentC}}}\lefri{\frac{1}{2}}\lefri{\LNorm{\dellgc}{1} + \LNorm{\deldlgc}{1}}^{2} \\
&\geq \min\lefri{\frac{\dboundC}{4h},\frac{1}{2h}\lefri{\frac{\dboundC \pi^{2}}{2h^{2}} - \lefri{\cboundC + \potentC}}}\SobNorm{\dellgc}{1}^{2}
\end{align*}%
which establishes the required coercivity result so long as \(0 < h < \sqrt{\frac{\dboundC \pi^{2}}{2 \lefri{\cboundC + \potentC}}}\).
\end{proof}

\section{Cayley Transform Based Method on \(SO\lefri{3}\)} \label{CayCon}

Because the construction of a Lie group Galerkin variational integrator can be involved, we will provide an example of an integrator based on the Cayley transform for the rigid body on \(\SoT\) and related problems. We will first construct the method and then verify that it satisfies the hypotheses of Theorems \ref{LG_OptConv} and \ref{LG_GeoConv}, and in \S \ref{NumExp} we will demonstrate numerically that it exhibits the expected convergence.

Additionally, discretizing the rigid body amounts to discretizing a kinetic energy term that can be used in many different applications. It appears that discretizing the kinetic energy term of the rigid body is more painstaking than the potential term, so we provide a detailed description so that others will not have to repeat the derivation of this discretization for future applications.

\subsection{Rigid Body on \(\SoT\)}

The Lagrangian:%
\begin{align}
L\lefri{R,\dot{R}} &= \mbox{tr}\lefri{\dot{R}^{T}RJ_{d}R^{T}\dot{R}} \label{RigidBodyLagrangian}\\
J_{d} &= \frac{1}{2} \mbox{tr}\left[J\right]I_{3\times3} - J \nonumber\\
J &= \mbox{tr}\left[J_{d}\right]I_{3\times3} - J_{d},
\end{align}%
where \(R \in \SoT\) and \(J\) are the moments of inertia in the reference coordinate frame, gives rise to the equations of motion for the rigid body. The rigid body has a rich geometric structure, which is discussed in \citet{LeMcLe2005}, \citet{CeBy2003}, and \citet{MaRa1999}. In addition to being an interesting example of a non-canonical Lagrangian system, it is a standard model problem for discretization for numerical methods on Lie groups, and an overview of integrators applied to the rigid body can be found in \citet{HaLuWa2006}. 

\subsection{Construction}

To construct the Lie group Galerkin variational integrator, we will have to choose:%
\begin{enumerate}
\item a map \(\Phi\lefri{\cdot}:\soT \rightarrow \SoT\),
\item a finite dimensional function space \(\AFdFSpace\), and
\item a quadrature rule,
\end{enumerate}%
and to complete the error analysis, we must also choose%
\begin{enumerate}
\item a metric on \(\soT\) \(\RMetric{\cdot}{\cdot}\),
\item error functions \(\GroupE{\cdot}{\cdot}\) and \(\AlgeE{\cdot}{\cdot}\).
\end{enumerate}%
For our construction, we will make use of the Cayley transform for our map \(\Phi\lefri{\cdot}\) and Lagrange interpolation polynomials through \(\soT\) for the finite-dimensional function space \(\AFdFSpace\), that is,%
\begin{align*}
\AFdFSpace = \left\{\lgc \vphantom{\lgc = \sum_{i=1}^{n} \widehat{q^{i}\phi_{i}\lefri{t}}, q^{i} \in \mathbb{R}^{3}, \phi_{i}\lefri{t}\mbox{ is the Lagrange interpolation polynomial for \(t_{i}\)}}\right| \left. \lgc = \sum_{i=1}^{n} \widehat{q^{i}\phi_{i}\lefri{t}}, q^{i} \in \mathbb{R}^{3}, \phi_{i}\lefri{t}\mbox{ is the Lagrange interpolation polynomial for \(t_{i}\)}\right\},
\end{align*}%
where \(\hat{\cdot}\) is that hat map described by (\ref{hatmap}). For the error analysis we will choose:%
\begin{align*}
\RMetric{\hat{\eta}}{\hat{\nu}} &= \eta^{T}\nu, \\
\GroupE{G_{1}}{G_{2}} &= \MNorm{G_{1} - G_{2}}{2}\\
\AlgeE{\hat{\eta}}{\hat{\nu}} &= \MNorm{\hat{\eta} - \hat{\nu}}{2},
\end{align*}%
for arbitrary \(G_{1},G_{2} \in \SoT\) and \(\eta, \nu \in \mathbb{R}^{3}\), where the \(\MNorm{\cdot}{2}\) norm is understood as arising from the \(\MNorm{\cdot}{2}\) from the embedding space \(R^{3\times3}\). We will discuss these below, and elaborate on the motivation for these choices in our construction.%
\subsubsection{The Cayley Transform}

To construct our Lie group Galerkin variational Integrator, we will make use of the Cayley Transform, \(\Cay{\cdot}:\soT \rightarrow \SoT\) which is given by:%
\begin{align*}
\Cay{q} = \lefri{I - Q}\lefri{I + Q}^{-1}.
\end{align*}%
The reader should note that we are using an unscaled version of the Cayley transform, but for the purposes of constructing the natural chart, different versions of the Cayley transform should result in equivalent methods. Furthermore, the choice of the Cayley transform for the integrator is certainly not necessary; different choices of maps, such as the exponential map, would result in equally valid methods. We make use of the Cayley transform simply because it is easy to manipulate and compute, is its own inverse, and because it satisfies our chart conditioning assumptions, as we will establish shortly.

\begin{lemma}
For \(\eta, \nu \in \soT \), so long as%
\begin{align}
2\MNorm{\eta}{2} + \MNorm{\nu}{2} < 1, \label{StrongAssumpt}
\end{align}%
the natural chart constructed by the Cayley transform locally satisfies chart conditioning assumption, that is:
\begin{align*}
\MNorm{\Cay{\eta} - \Cay{\nu}}{2} &\leq \GroupC \RMetric{\eta - \nu}{\eta - \nu}^{\frac{1}{2}} \\
\MNorm{D_{\eta}\Cay{\dot{\eta}} - D_{\nu}\Cay{\dot{\nu}}}{2} &\leq \AlgeC\RMetric{\eta - \nu}{\eta - \nu}^{\frac{1}{2}} + \AlgeGC \RMetric{\dot{\eta} - \dot{\nu}}{\dot{\eta} - \dot{\nu}}^{\frac{1}{2}}.
\end{align*}%
If \(\MNorm{\eta - \nu}{2} < \epsilon\), assumption (\ref{StrongAssumpt}) can be relaxed to%
\begin{align*}
\MNorm{\eta}{2} + \epsilon < 1.
\end{align*}%
\end{lemma}

\begin{proof} Throughout the proof of this lemma, we will make extensive use of two inequalities. The first is the bound:
\begin{align}
\MNorm{\lefri{I + E}^{-1}}{p} \leq \lefri{1 - \MNorm{E}{p}}^{-1}, \label{InvIneq1}
\end{align}%
if \(\MNorm{E}{p} < 1\), and the second is the bound:
\begin{align}
\MNorm{\lefri{A + E}^{-1} - A^{-1}}{p} \leq \MNorm{E}{p}\MNorm{A^{-1}}{p}^{2} \lefri{1 - \MNorm{A^{-1}E}{p}}^{-1} \label{InvIneq2}
\end{align}%
which generalizes (\ref{InvIneq1}). We begin with%
\begin{align}
\MNorm{\Cay{\eta} - \Cay{\nu}}{2} & = \MNorm{\lefri{I - \eta}\lefri{I + \eta}^{-1} - \lefri{I - \nu}\lefri{I + \nu}^{-1}}{2} \nonumber \\
& = \MNorm{\lefri{I - \eta}\lefri{I + \eta}^{-1} - \lefri{I - \eta}\lefri{I + \nu}^{-1} + \lefri{I - \eta}\lefri{I + \nu}^{-1} - \lefri{I - \nu}\lefri{I + \nu}^{-1}}{2} \nonumber \\
& = \MNorm{\lefri{I - \eta}\left[\lefri{I + \eta}^{-1} - \lefri{I + \nu}^{-1}\right] + \left[\lefri{I - \eta} - \lefri{I - \nu}\right]\lefri{I + \nu}^{-1}}{2} \nonumber\\
& = \MNorm{\lefri{I - \eta}\left[\lefri{I + \eta}^{-1} - \lefri{I + \nu}^{-1}\right] + \left[\nu - \eta\right]\lefri{I + \nu}^{-1}}{2}\nonumber\\
& \leq \MNorm{\lefri{I - \eta}\left[\lefri{I + \eta}^{-1} - \lefri{I + \nu}^{-1}\right]}{2} + \MNorm{\left[\nu - \eta\right]\lefri{I + \nu}^{-1}}{2}. \label{CayleyChartIneq}
\end{align}%
Considering the term \(\left[\nu - \eta\right]\lefri{I + \nu}^{-1}\), we make use of (\ref{InvIneq1}) to see%
\begin{align}
\MNorm{\left[\nu - \eta\right]\lefri{I + \nu}^{-1}}{2} & \leq \MNorm{\nu - \eta}{2}\MNorm{\lefri{I + \nu}^{-1}}{2} \nonumber \\
& \leq \lefri{1 - \MNorm{\nu}{2}}^{-1}\MNorm{\eta - \nu}{2}. \label{ChartTermOne}
\end{align}%
Next, considering the term \(\MNorm{\lefri{I - \eta}\left[\lefri{I + \eta}^{-1} - \lefri{I + \nu}^{-1}\right]}{2}\),%
\begin{align}
\MNorm{\lefri{I - \eta}\left[\lefri{I + \eta}^{-1} - \lefri{I + \nu}^{-1}\right]}{2} & \leq \MNorm{I - \eta}{2} \MNorm{\lefri{I + \eta}^{-1} - \lefri{I + \nu}^{-1}}{2}\nonumber \\
& = \MNorm{I - \eta}{2} \MNorm{\lefri{I + \nu + \lefri{\eta - \nu}}^{-1} - \lefri{I + \nu}^{-1}}{2}. \label{LG_Denom_Ineq5}
\end{align}%
Applying (\ref{InvIneq2}), with \(E = \eta - \nu\) and \(A = I + \nu\),
\begin{align}
\MNorm{\lefri{I + \nu +\lefri{\eta - \nu}}^{-1} - \lefri{I + \nu}^{-1}}{2} & \leq \MNorm{\eta - \nu}{2}\MNorm{\lefri{I + \nu}^{-1}}{2}^{2}\lefri{1 - \MNorm{\lefri{I + \nu}^{-1}\lefri{\eta - \nu}}{2}}^{-1}. \label{LG_Denom_Ineq3}
\end{align}%
But%
\begin{align*}
1 - \MNorm{\lefri{I + \nu}^{-1}\lefri{\eta - \nu}}{2} &\geq 1 - \MNorm{\lefri{I + \nu}^{-1}}{2}\MNorm{\lefri{\eta - \nu}}{2}\\
& \geq 1 - \lefri{1 - \MNorm{\nu}{2}}^{-1}\MNorm{\lefri{\eta - \nu}}{2}
\end{align*}%
which implies%
\begin{align*}
\lefri{1 - \MNorm{\lefri{I + \nu}^{-1}\lefri{\eta - \nu}}{2}}^{-1} \leq \lefri{1 - \lefri{1 - \MNorm{\nu}{2}}^{-1}\MNorm{\lefri{\eta - \nu}}{2}}^{-1} 
\end{align*}
and%
\begin{align*}
\MNorm{\lefri{I + \nu}^{-1}}{2}^{2} \leq \lefri{1 - \MNorm{\nu}{2}}^{-2},
\end{align*}%
so%
\begin{align}
\MNorm{\lefri{I + \nu}^{-1}}{2}^{2}\lefri{1 - \MNorm{\lefri{I + \nu}^{-1}\lefri{\eta - \nu}}{2}}^{-1} & \leq \lefri{1 - \MNorm{\nu}{2}}^{-2}\lefri{1 - \lefri{1 - \MNorm{\nu}{2}}^{-1}\MNorm{\lefri{\eta - \nu}}{2}}^{-1} \nonumber \\
& \leq \lefri{1 - \MNorm{\nu}{2}}^{-1}\lefri{\lefri{1 - \MNorm{\nu}{2}} - \MNorm{\eta - \nu}{2}}^{-1} \nonumber \\
& = \lefri{1 - \MNorm{\nu}{2}}^{-1}\lefri{1 - \MNorm{\nu}{2} - \MNorm{\eta - \nu}{2}}^{-1}  \label{LG_Diff_Ineq}.
\end{align}%
The triangle inequality gives%
\begin{align*}
\MNorm{\eta - \nu}{2} \leq \MNorm{\eta}{2} + \MNorm{\nu}{2}
\end{align*}%
and thus%
\begin{align}
1 - \MNorm{\eta}{2} - \MNorm{\eta - \nu}{2} &\geq 1 - 2\MNorm{\eta}{2} - \MNorm{\nu}{2}\nonumber \\
\lefri{1 - \MNorm{\eta}{2} - \MNorm{\eta - \nu}{2}}^{-1} &\leq \lefri{1 - 2\MNorm{\eta}{2} - \MNorm{\nu}{2}}^{-1}. \label{LG_Denom_Ineq1}
 \end{align}%
So applying (\ref{LG_Denom_Ineq1}) to (\ref{LG_Diff_Ineq}) gives,
\begin{align}
\MNorm{\lefri{I + \nu}^{-1}}{2}^{2}\lefri{1 - \MNorm{\lefri{I + \nu}^{-1}\lefri{\eta - \nu}}{2}}^{-1} & \leq \lefri{1 - \MNorm{\nu}{2}}^{-1}\lefri{1 - 2\MNorm{\eta}{2} - \MNorm{\nu}{2}}^{-1}. \label{LG_Denom_Ineq2},
\end{align}%
then applying (\ref{LG_Denom_Ineq2}) to (\ref{LG_Denom_Ineq3}) gives,
\begin{align}
\MNorm{\lefri{I + \nu +\lefri{\eta - \nu}}^{-1} - \lefri{I + \nu}^{-1}}{2} & \leq \MNorm{\eta - \nu}{2}\lefri{1 - \MNorm{\nu}{2}}^{-1}\lefri{1 - 2\MNorm{\eta}{2} - \MNorm{\nu}{2}}^{-1} \label{LG_Denom_Ineq4},
\end{align}
and finally applying (\ref{LG_Denom_Ineq4}) to (\ref{LG_Denom_Ineq5}) yields
\begin{align}
\MNorm{\lefri{I - \eta}\left[\lefri{I + \eta}^{-1} - \lefri{I + \nu}^{-1}\right]}{2} &\leq \MNorm{I - \eta}{2}\lefri{1 - \MNorm{\nu}{2}}^{-1}\lefri{1 - 2\MNorm{\eta}{2} - \MNorm{\nu}{2}}^{-1}\MNorm{\eta - \nu}{2} \nonumber \\
& \leq \lefri{1 - \MNorm{\eta}{2}}^{-1}\lefri{1 - \MNorm{\nu}{2}}^{-1}\lefri{1 - 2\MNorm{\eta}{2} - \MNorm{\nu}{2}}^{-1}\MNorm{\eta - \nu}{2} \label{ChartTermTwo}
\end{align}
Substituting (\ref{ChartTermOne}) and (\ref{ChartTermTwo}) into (\ref{CayleyChartIneq}), we see%
\begin{align*}
\MNorm{\Cay{\eta} - \Cay{\nu}}{2} \leq \left[\lefri{1 - \MNorm{\nu}{2}}^{-1} + \lefri{1 - \MNorm{\eta}{2}}^{-1}\lefri{1 - \MNorm{\nu}{2}}^{-1}\lefri{1 - 2\MNorm{\eta}{2} - \MNorm{\nu}{2}}^{-1}\right]\MNorm{\eta - \nu}{2}.
\end{align*}%
Hence, so as long as \(2\left\|\eta\right\| + \left\|\nu\right\| < \CondC < 1 \)%
\begin{align*}
\MNorm{\Cay{\eta} - \Cay{\nu}}{2} \leq \GroupC \MNorm{\eta - \nu}{2}.
\end{align*}%
where%
\begin{align*}
\GroupC = \left[\lefri{1 - \MNorm{\nu}{2}}^{-1} + \lefri{1 - \MNorm{\eta}{2}}^{-1}\lefri{1 - \MNorm{\nu}{2}}^{-1}\lefri{1 - 2\MNorm{\eta}{2} - \MNorm{\nu}{2}}^{-1}\right].
\end{align*}%
If we make the stronger assumption that \(\MNorm{\eta - \nu}{2} < \epsilon\), we can weaken the assumption to simply \(\MNorm{\eta}{2} + \epsilon < \CondC < 1\) and \(\MNorm{\nu}{2} + \epsilon < \CondC < 1\). As we expect the error between the two curves in the Lie algebra to be orders of magnitude smaller than the magnitude of the Lie algebra elements, this is a reasonable assumption to make.

Next, to examine \(\MNorm{D_{\eta}\Cay{\dot{\eta}} - D_{\nu}\Cay{\dot{\nu}}}{2}\), we consider the definition%
\begin{align*}
D_{X}\Cay{Y} = -Y\lefri{I + X}^{-1} - \lefri{I - X}\lefri{I + X}^{-1}Y\lefri{I + X}^{-1}.
\end{align*}%
Using this definition,%
\begin{align*}
\MNorm{D_{\eta}\Cay{\dot{\eta}} - D_{\nu}\Cay{\dot{\nu}}}{2} &= \left\|-\dot{\eta}\lefri{I + \eta}^{-1} - \lefri{I - \eta}\lefri{I + \eta}^{-1}\dot{\eta}\lefri{I + \eta}^{-1} + \right.\\
& \left.\hspace{5em}\dot{\nu}\lefri{I + \nu}^{-1} + \lefri{I - \nu}\lefri{I + \nu}^{-1}\dot{\nu}\lefri{I + \nu}^{-1}\right\|_{2} \\
& \leq \MNorm{\dot{\nu}\lefri{I + \nu}^{-1} - \dot{\eta}\lefri{I + \eta}^{-1}}{2} + \\
& \hspace{5em} \MNorm{\lefri{I - \nu}\lefri{I + \nu}^{-1}\dot{\nu}\lefri{I + \nu}^{-1} - \lefri{I - \eta}\lefri{I + \eta}^{-1}\dot{\eta}\lefri{I + \eta}^{-1}}{2}.
\end{align*}%
Considering
\begin{align*}
\MNorm{\dot{\nu}\lefri{I + \nu}^{-1} - \dot{\eta}\lefri{I + \eta}^{-1}}{2} &= \MNorm{\dot{\nu}\lefri{I + \nu}^{-1} - \dot{\eta}\lefri{I + \nu}^{-1} + \dot{\eta}\lefri{I + \nu}^{-1} - \dot{\eta}\lefri{I + \eta}^{-1}}{2} \\
& = \MNorm{\lefri{\dot{\nu} - \dot{\eta}}\lefri{I + \nu}^{-1} + \dot{\eta}\left[\lefri{I + \nu}^{-1} - \lefri{I + \eta}^{-1}\right]}{2} \\
& \leq \MNorm{\lefri{I + \nu}^{-1}}{2}\MNorm{\dot{\eta} - \dot{\nu}}{2} + \MNorm{\dot{\eta}}{2}\MNorm{\lefri{I + \nu}^{-1} - \lefri{I + \eta}^{-1}}{2} \\
& \leq \lefri{1 - \MNorm{\nu}{2}}^{-1}\MNorm{\dot{\eta} - \dot{\nu}}{2} + \MNorm{\dot{\eta}}{2}\lefri{\lefri{1 - \MNorm{\nu}{2}}\lefri{1 - \MNorm{\nu}{2} - \MNorm{\nu - \eta}{2}}}^{-1}\MNorm{\eta - \nu}{2},
\end{align*}%
where we have made use of (\ref{LG_Denom_Ineq4}) to bound \(\MNorm{\lefri{I + \nu}^{-1} - \lefri{I + \eta}^{-1}}{2}\). Now, considering the second term, we first note,
\begin{align*}
\MNorm{\lefri{I - \nu}\lefri{I + \nu}^{-1}\dot{\nu}\lefri{I + \nu}^{-1} - \lefri{I - \eta}\lefri{I + \eta}^{-1}\dot{\eta}\lefri{I + \eta}^{-1}}{2} = \MNorm{\Cay{\nu}\dot{\nu}\lefri{I + \nu}^{-1} - \Cay{\eta}\dot{\eta}\lefri{I + \eta}^{-1}}{2}.
\end{align*}%
Using this, we see%
\begin{align}
\MNorm{\Cay{\nu}\dot{\nu}\lefri{I + \nu}^{-1} - \Cay{\eta}\dot{\eta}\lefri{I + \eta}^{-1}}{2} & = \left\|\Cay{\nu}\dot{\nu}\lefri{I + \nu}^{-1} - \Cay{\nu}\dot{\nu}\lefri{I + \eta}^{-1} \right. \nonumber \\
& \hspace{5em} + \left.\Cay{\nu}\dot{\nu}\lefri{I + \eta}^{-1} - \Cay{\eta}\dot{\eta}\lefri{I + \eta}^{-1}\right\|_{2}. \nonumber\\
& \leq \MNorm{\Cay{\nu}\dot{\nu}\lefri{I + \nu}^{-1} - \Cay{\nu}\dot{\nu}\lefri{I + \eta}^{-1}}{2} \nonumber \\
& \hspace{5em} + \MNorm{\Cay{\nu}\dot{\nu}\lefri{I + \eta}^{-1} - \Cay{\eta}\dot{\eta}\lefri{I + \eta}^{-1}}{2} \label{DChartTermOne}
\end{align}%
For the first term in (\ref{DChartTermOne}),%
\begin{align}
\MNorm{\Cay{\nu}\dot{\nu}\lefri{I + \nu}^{-1} - \Cay{\nu}\dot{\nu}\lefri{I + \eta}^{-1}}{2} &= \MNorm{\Cay{\nu}\dot{\nu}\left[\lefri{I + \nu}^{-1} - \lefri{I + \eta}^{-1}\right]}{2} \nonumber\\
& \leq \MNorm{\Cay{\nu}}{2} \MNorm{\dot{\nu}}{2} \MNorm{\lefri{I + \nu}^{-1} - \lefri{I + \eta}^{-1}}{2} \nonumber\\
& \leq \MNorm{\dot{\nu}}{2} \lefri{\lefri{1 - \MNorm{\nu}{2}}\lefri{1 - \MNorm{\nu}{2} - \MNorm{\nu - \eta}{2}}}^{-1}\MNorm{\eta - \nu}{2}. \label{DChartTermTwo}
\end{align}%
where we once again have made use of (\ref{LG_Denom_Ineq4}) to bound \(\MNorm{\lefri{I + \nu}^{-1} - \lefri{I + \eta}^{-1}}{2}\) and the fact that \(\Cay{\nu}\) is orthogonal to set \(\MNorm{\Cay{\nu}}{2} = 1\). Now, considering the second term in (\ref{DChartTermOne}),
\begin{align}
\MNorm{\Cay{\nu}\dot{\nu}\lefri{I + \eta}^{-1} - \Cay{\eta}\dot{\eta}\lefri{I + \eta}^{-1}}{2} & = \MNorm{\lefri{\Cay{\nu}\dot{\nu} - \Cay{\eta}\dot{\eta}}\lefri{I + \eta}^{-1}}{2} \nonumber\\
& \leq \MNorm{\Cay{\nu}\dot{\nu} - \Cay{\eta}\dot{\eta}}{2}\MNorm{\lefri{I + \eta}^{-1}}{2} \label{DChartTermThree} 
\end{align}%
and additionally,%
\begin{align}
\MNorm{\Cay{\nu}\dot{\nu} - \Cay{\eta}\dot{\eta}}{2} &= \MNorm{\Cay{\nu}\dot{\nu} - \Cay{\eta}\dot{\nu} + \Cay{\eta}\dot{\nu} - \Cay{\eta}\dot{\eta}}{2}\nonumber \\
&\leq \MNorm{\Cay{\nu} - \Cay{\eta}}{2} \MNorm{\dot{\nu}}{2} + \MNorm{\Cay{\eta}}{2}\MNorm{\dot{\nu} - \dot{\eta}}{2}\nonumber\\
& \leq \GroupC \MNorm{\dot{\nu}}{2} \MNorm{\eta - \nu}{2} + \MNorm{\dot{\eta} - \dot{\nu}}{2}. \label{DChartTermFour}
\end{align}%
Combining (\ref{DChartTermTwo}), (\ref{DChartTermThree}), (\ref{DChartTermFour}) in (\ref{DChartTermOne}) yields%
\begin{align*}
\MNorm{D_{\eta}\Cay{\dot{\eta}} - D_{\nu}\Cay{\dot{\nu}}}{2} \leq \AlgeC\MNorm{\eta - \nu}{2} + \AlgeGC\MNorm{\dot{\eta} - \dot{\nu}}{2}
\end{align*}%
with constants%
\begin{align*}
\AlgeC &= \lefri{1 - \MNorm{\nu}{2} - \MNorm{\nu - \eta}{2}}^{-1}\lefri{\MNorm{\dot{\eta}}{2}\lefri{1 - \MNorm{\eta}{2}}^{-1} + \MNorm{\dot{\nu}}{2}\lefri{1 - \MNorm{\nu}{2}}^{-1}} + \GroupC\MNorm{\dot{\nu}}{2}\\
\AlgeGC & = 1 + \lefri{1 - \MNorm{\nu}{2}}^{-1}.
\end{align*}%
To complete the proof of the lemma, we need to establish a bound on the matrix two norm from the metric on the Lie algebra. For arbitrary algebra element \(\xi\), standard vector and matrix norm equivalences yield%
\begin{align*}
\MNorm{\hat{\xi}}{2} \leq \sqrt{3}\MNorm{\hat{\xi}}{1} \leq \sqrt{3} \MNorm{\xi}{1} \leq 3 \MNorm{\xi}{2} = 3 \RMetric{\hat{\xi}}{\hat{\xi}}^{\frac{1}{2}}
\end{align*}
which completes the proof. 
\end{proof}

It should be noted that as \(\MNorm{\nu}{2}\) or \(\MNorm{\eta}{2}\) approaches \(1\), \(\GroupC\), \(\AlgeC\) and \(\AlgeGC\) increase without bound. This amounts to a time step restriction for the method; if the configuration changes too dramatically during the time step, the chart will become poorly conditioned and the numerical solution will degrade. However, as long as \(\MNorm{\nu}{2} < \CondC\) and \(\MNorm{\eta}{2} < \CondC\) on each time step for some \(\CondC < 1\) which is independent of the number of the time step, these constants will remain bounded and the natural chart will be well conditioned. 

\subsubsection{Choice of Basis Functions}

The final feature of the construction of our Cayley transform Lie group Galerkin variational integrator is the choice of function space \(\AFdFSpace\) for approximation of curves in the Lie algebra. Since the curves in the Lie algebra \(\soT\) that we use have a natural correspondence with curves in \(\mathbb{R}^{3}\) through the hat map, constructing these curves reduces to choosing an approximation space for curves in \(\mathbb{R}^{3}\).

We make the choice of polynomials of degree at most \(n + 1\) for \(\AFdFSpace\). We choose polynomials because approximation theory and particularly the theory of spectral numerical methods, see \citet{Tr2000}, tells us that polynomials have excellent convergence under both \(h\) and \(n\) refinement to smooth curves, and in particular, analytic curves. For the basis functions \(\left\{\phi_{i}\lefri{t}\right\}_{i=1}^{n}\), we choose \(\phi_{i}\lefri{t}\) to be the Lagrange interpolation polynomial for the \(i\)-th of \(n\) Chebyshev points rescaled to the interval \(\left[0,h\right]\), that is%
\begin{align*}
\phi_{i}\lefri{t} = \frac{\prod_{j=1,j\neq i}^{n}\lefri{t- t_{j}}}{\prod_{j=1,j \neq i}^{n}\lefri{t_{i} - t_{j}}}
\end{align*}%
for \(t_{i} = \frac{h}{2}\cos\lefri{\frac{i \pi}{n}} + \frac{h}{2}\). While our convergence theory does not depend on the choice of polynomial basis, there are two major benefits for this choice of basis functions. The first is that these polynomials interpolate \(0\) and \(h\), which greatly simplifies the computation of \(D_{1}\lefri{R_{k},R_{k+1}}\) and \(D_{2}\lefri{R_{k-1},R_{k}}\). The second is that this choice of basis function leads to methods which are more stable than other choices of interpolation points, most likely because of the excellent stability properties of the interpolation polynomials that are constructed from them. The interested reader is referred to \citet{Tr2000} and \citet{Bo2001} and the references therein for more details on spectral numerical methods.

\subsubsection{Choice of Quadrature Rule}

The final selection we must make when constructing the integrator is a choice of quadrature rule. We choose to use Gaussian quadrature, mostly because this quadrature rule is optimally accurate in the number of points and because it is simple to compute higher order Gaussian quadrature points and weights by solving a small eigenvalue problem. However, it is possible to use other rules, and we make no claim that our choice is the best for our choice of parameters.

\subsection{Discrete Euler Poincar\'{e} Equations}

While in \S \ref{ISDEPSubsection} we presented a general form of the internal stage discrete Euler-Poincar\'{e} equations in coordinate-free notation, direct construction of these equations is probably not the easiest way to formulate a numerical method. This is because it requires the computation and composition of many different functions, some of which may be complicated (for example, working out \(\darkD{\alpha,\dot{\alpha}}\darkD{\alpha}\Phi\lefri{\darkD{\mathfrak{q}_{i}}\alpha,\darkD{\mathfrak{q}_{i}}\dot{\alpha}}\) for the Cayley transform is straightforward, but also slightly obnoxious). An alternative approach, to which we alluded in \S \ref{ISDEPSubsection}, is to compute the discrete action in coordinates, and then explicitly compute the stationarity conditions for this discrete action. We do this here for the rigid body equations.

For the construction of the Lie group Galerkin variational integrator for the rigid body, we make use of the following functions:%
\begin{align}
\GalR\lefri{\left\{\xi^{i}\right\}_{i=1}^{n},t} &= R_{k}\Phi\lefri{\sum_{i=1}^{n}\hat{\xi}^{i}\phi_{i}\lefri{t}} \nonumber\\
L\lefri{R\lefri{t},\dot{R}\lefri{t}} &= \mbox{tr}\lefri{\dot{R}\lefri{t}^{T}R\lefri{t}J_{d}R\lefri{t}^{T}\dot{R}\lefri{t}} \nonumber\\
L_{d}\lefri{R_{k},R_{k+1}} &= \ext_{\CayGalargsf{R_{k}}{R_{k+1}}} h \sum_{j=1}^{m} b_{j}L\lefri{\GalR\lefri{\cjh},\dGalR\lefri{\cjh}} \label{CayDEP}
\end{align} %
where \(\xi^{i} \in \mathbb{R}^{3}\). Since the curve \(\GalR\lefri{t}\) is a function on \(n\) points in \(\mathbb{R}^{3}\), denoting \(\xi^{i} = \lefri{\xi^{i}_{a},\xi^{i}_{b},\xi^{i}_{c}}\), we can write (\ref{CayDEP}) as%
\begin{align*}
L_{d}\lefri{R_{k},R_{k+1}} &= \ext_{\xi^{0} = 0, \hat{\xi}^{n} = \Phi^{-1}\lefri{R_{k}^{T}R_{k+1}}} h \sum_{j=1}^{m} b_{j} \frac{2}{\lefri{1 + \MNorm{\xi\lefri{\cjh}}{2}^{2}}^{2}}\left(I_{1}\varphi\lefri{\xi_{c},\xi_{a},\xi_{b}}^{2}  + \right.\\
& \hspace{18em} \left. I_{2} \varphi\lefri{\xi_{b},\xi_{c},\xi_{a}}^{2} + I_{3} \varphi\lefri{\xi_{a},\xi_{b},\xi_{c}}^{2}\right)
\end{align*}%
where%
\begin{align*}
\varphi\lefri{\xi_{a},\xi_{b},\xi_{c}} & =  \dot{\xi}_{a}\lefri{\cjh} + \xi_{b}\lefri{\cjh}\dot{\xi}_{c}\lefri{\cjh} - \xi_{c}\lefri{\cjh}\dot{\xi}_{b}\lefri{\cjh},
\end{align*}%
\(\varphi\lefri{\xi_{c},\xi_{a},\xi_{b}}\) and \(\varphi\lefri{\xi_{b},\xi_{c},\xi_{a}}\) are defined analogously, and 
\begin{align*}
I_{i} &= \sum_{j \neq i} \lefri{J_{d}}_{jj} \\
\xi_{x}\lefri{t} &= \sum_{i}^{n} \xi_{x}^{i} \phi_{i}\lefri{t} \\
\xi\lefri{t} &= \lefri{\xi_{a}\lefri{t},\xi_{b}\lefri{t}, \xi_{c}\lefri{t}}.
\end{align*}%
Forming the action sum as a function of the \(\xi^{i}\),%
\begin{align*}
\mathbb{S}_{d}\lefri{\left\{\xi^{i}\right\}_{i=1}^{n}} = h \sum_{j=1}^{m} b_{j} \frac{2}{\lefri{1 + \MNorm{\xi\lefri{\cjh}}{2}^{2}}^{2}}\left(I_{1}\varphi\lefri{\xi_{c},\xi_{a},\xi_{b}}^{2}  + I_{2} \varphi\lefri{\xi_{b},\xi_{c},\xi_{a}}^{2} + I_{3} \varphi\lefri{\xi_{a},\xi_{b},\xi_{c}}^{2}\right)
\end{align*}
computing its variational derivative from \(\xi^{i}\) directly and setting it equal to \(0\),%
\begin{align*}
\left.\frac{\mbox{d}}{\mbox{d}\epsilon}\mathbb{S}_{d}\lefri{\left\{\xi^{i} + \epsilon \delta \xi^{i}\right\}_{i=1}^{n}} \right|_{\epsilon = 0}= 0,
\end{align*}
gives the internal stage discrete Euler-Poincar\'{e} equations,%
\begin{subequations}
\begin{align}
h\sum_{j=1}^{m}\bj4\lefri{1 + \MNorm{\xi}{2}^{2}}^{-2} & \left[\lefri{I_{3}\varphi\lefri{\xi_{c},\xi_{a},\xi_{b}}}\lefri{-2\lefri{1+ \MNorm{\xi}{2}^{2}}^{-1}\varphi\lefri{\xi_{c},\xi_{a},\xi_{b}}\xi_{a}\phi_{i} + \dot{\xi}_{b}\phi_{i} - \xi_{b}\dot{\phi}_{i}} \right. + & \label{DEPone} \\
& \lefri{I_{2}\varphi\lefri{\xi_{b},\xi_{c},\xi_{a}}}\lefri{-2\lefri{1+ \MNorm{\xi}{2}^{2}}^{-1}\varphi\lefri{\xi_{b},\xi_{c},\xi_{a}}\xi_{a}\phi_{i} + \xi_{c}\dot{\phi}_{i} - \dot{\xi}_{c}\phi_{i}} + &  \nonumber\\
& \left. \lefri{I_{1}\varphi\lefri{\xi_{a},\xi_{b},\xi_{c}}}\lefri{-2\lefri{1+ \MNorm{\xi}{2}^{2}}^{-1}\varphi\lefri{\xi_{a},\xi_{b},\xi_{c}}\xi_{a}\phi_{i} + \dot{\phi}_{i}} \right] & = 0 \nonumber\\
h\sum_{j=1}^{m}\bj4\lefri{1 + \MNorm{\xi}{2}^{2}}^{-2} & \left[\lefri{I_{3}\varphi\lefri{\xi_{c},\xi_{a},\xi_{b}}}\lefri{-2\lefri{1+ \MNorm{\xi}{2}^{2}}^{-1}\varphi\lefri{\xi_{c},\xi_{a},\xi_{b}}\xi_{b}\phi_{i} + \xi_{a}\dot{\phi}_{i} - \dot{\xi}_{a}\phi_{i}} \right. + & \label{DEPtwo}  \\
& \lefri{I_{2}\varphi\lefri{\xi_{b},\xi_{c},\xi_{a}}}\lefri{-2\lefri{1+ \MNorm{\xi}{2}^{2}}^{-1}\varphi\lefri{\xi_{b},\xi_{c},\xi_{a}}\xi_{b}\phi_{i} + \dot{\phi}_{i}} + &  \nonumber \\
& \left. \lefri{I_{1}\varphi\lefri{\xi_{a},\xi_{b},\xi_{c}}}\lefri{-2\lefri{1+ \MNorm{\xi}{2}^{2}}^{-1}\varphi\lefri{\xi_{a},\xi_{b},\xi_{c}}\xi_{b}\phi_{i} + \dot{\xi}_{c}\phi_{i} - \xi_{c}\dot{\phi}_{i}} \right] & = 0 \nonumber \\
h\sum_{j=1}^{m}b_{j}4\lefri{1 + \MNorm{\xi}{2}^{2}}^{-2} & \left[\lefri{I_{3}\varphi\lefri{\xi_{c},\xi_{a},\xi_{b}}}\lefri{-2\lefri{1+ \MNorm{\xi}{2}^{2}}^{-1}\varphi\lefri{\xi_{c},\xi_{a},\xi_{b}}\xi_{c}\phi_{i} + \dot{\phi}_{i}} \right. + & \label{DEPthree}\\
& \lefri{I_{2}\varphi\lefri{\xi_{b},\xi_{c},\xi_{a}}}\lefri{-2\lefri{1+ \MNorm{\xi}{2}^{2}}^{-1}\varphi\lefri{\xi_{b},\xi_{c},\xi_{a}}\xi_{c}\phi_{i} + \dot{\xi}_{a}\phi_{i} - \xi_{a}\dot{\phi}_{i}} + & \nonumber \\
& \left. \lefri{I_{1}\varphi\lefri{\xi_{a},\xi_{b},\xi_{c}}}\lefri{-2\lefri{1+ \MNorm{\xi}{2}^{2}}^{-1}\varphi\lefri{\xi_{a},\xi_{b},\xi_{c}}\xi_{c}\phi_{i} + \xi_{b}\dot{\phi}_{i} - \dot{\xi}_{b}\phi_{i}} \right] & = 0, \nonumber
\end{align}
\end{subequations}%
for \(i = 2,...,n-1\), and where we have suppressed the \(t\) argument on all of our functions. Solving these equations, along with the condition \(\xi^{1} = 0\) and the discrete Euler-Poincar\'{e} equations%
\begin{align}
D_{1}L_{d}\lefri{R_{k},R_{k+1}} + D_{2}L_{d}\lefri{R_{k-1},R_{k}} = 0, \label{DEPBoundCond}
\end{align}%
which we will discuss in \S \ref{MomMatching}, yields \(\tilde{R}\lefri{t}\), the stationary point of the discrete action. Using this stationary point, computing \(\tilde{R}\lefri{h} = R_{k+1}\) gives us the next step of our one-step map.%
\subsubsection{Momentum Matching} \label{MomMatching}

As we mentioned in our general derivation of the discrete Euler-Poincar\'{e} equations, (\ref{DEPBoundCond}) must be treated with care. We described in an expedient method for computing \(D_{1}L_{d}\lefri{R_{k},R_{k+1}}\) so that the result is compatible with our change of natural charts. We will provide an explicit example below.

We already know the expression for \(D_{1}L_{d}\lefri{R_{k},R_{k+1}}\) for the coordinates in the current natural chart, the vector of the form (\ref{DEPone}) -- (\ref{DEPthree}), with \(i = 1\). This is the map \(\frac{\partial L_{d}}{\partial \xi_{k}}\) described in \S \ref{DEPSection}. Now, we need to compute an expression for \(\lambda_{k}\) and \(\frac{\partial \xi_{k}}{\partial \lambda_{k}}\). Given \(\xi_{0} = \lefri{\xi_{a}^{0}, \xi_{b}^{0}, \xi_{c}^{0}}\) and \(\xi_{k} = \lefri{\xi_{a}, \xi_{b}, \xi_{c}}\), we compute \(\lambda\) by%
\begin{align*}
\hat{\lambda} = \Phi^{-1}\lefri{\Phi\lefri{\hat{\xi_{0}}}\Phi\lefri{\hat{\xi_{k}}}} 
\end{align*}%
which gives in coordinates \(\lambda = \lefri{\lambda_{a},\lambda_{b},\lambda_{c}}\),%
\begin{align*}
\lambda_{a} &= \frac{-\xi_{a} - \xi_{a}^{0} + \xi_{c}\xi_{b}^{0} - \xi_{b}\xi_{c}^{0}}{-1 + \xi_{a}^{0}\xi_{a} + \xi_{b}^{0}\xi_{b} + \xi_{c}^{0}\xi_{c}}\\
\lambda_{b} &= \frac{-\xi_{b} - \xi_{b}^{0} + \xi_{a}\xi_{c}^{0} - \xi_{c}\xi_{a}^{0}}{-1 + \xi_{a}^{0}\xi_{a} + \xi_{b}^{0}\xi_{b} + \xi_{c}^{0}\xi_{c}}\\
\lambda_{c} &= \frac{-\xi_{c} - \xi_{c}^{0} + \xi_{b}\xi_{a}^{0} - \xi_{a}\xi_{b}^{0}}{-1 + \xi_{a}^{0}\xi_{a} + \xi_{b}^{0}\xi_{b} + \xi_{c}^{0}\xi_{c}}.
\end{align*}%
Now, we recompute \(\xi_{k}\) in terms of \(\lambda\),%
\begin{align*}
\hat{\xi}_{k} = \Phi^{-1}\lefri{\lefri{\Phi\lefri{\hat{\xi}_{0}}}^{-1}\Phi\lefri{\hat{\lambda}}}
\end{align*}%
which, when expressed in coordinates \(\xi_{k} = \lefri{\xi_{a},\xi_{b},\xi_{c}}\), gives%
\begin{align*}
\xi_{a} &= \frac{\lambda_{a} - \xi_{a}^{0} + \lambda_{c}\xi_{b}^{0} - \lambda_{b}\xi_{c}^{0}}{1 + \lambda_{a}\xi_{a}^{0} + \lambda_{b}\xi_{b}^{0} + \lambda_{c}\xi_{c}^{0}}\\
\xi_{b} &= \frac{\lambda_{b} - \xi_{b}^{0} + \lambda_{a}\xi_{c}^{0} - \lambda_{c}\xi_{a}^{0}}{1 + \lambda_{a}\xi_{a}^{0} + \lambda_{b}\xi_{b}^{0} + \lambda_{c}\xi_{c}^{0}}\\
\xi_{c} &= \frac{\lambda_{c} - \xi_{c}^{0} + \lambda_{b}\xi_{a}^{0} - \lambda_{a}\xi_{b}^{0}}{1 + \lambda_{a}\xi_{a}^{0} + \lambda_{b}\xi_{b}^{0} + \lambda_{c}\xi_{c}^{0}}.
\end{align*}%
So, to compute \(D_{1}L_{d}\lefri{R_{k},R_{k+1}}\) = \(\lefri{\frac{\partial L_{d}}{\partial \lambda_{a}},\frac{\partial L_{d}}{\partial \lambda_{a}},\frac{\partial L_{d}}{\partial \lambda_{a}}}\), we can take the easily computed expression \(\frac{\partial L_{d}}{\partial \xi_{k}}\) and apply a change of coordinates computation,
\begin{align*}
\frac{\partial L_{d}}{\partial\lambda_{a}} &= \frac{\partial L_{d}}{\partial \xi_{a}} \frac{\partial \xi_{a}}{\partial \lambda_{a}} + \frac{\partial L_{d}}{\partial \xi_{b}} \frac{\partial \xi_{b}}{\partial \lambda_{a}} + \frac{\partial L_{d}}{\partial \xi_{c}} \frac{\partial \xi_{c}}{\partial \lambda_{a}} \\
\frac{\partial L_{d}}{\partial\lambda_{b}} &= \frac{\partial L_{d}}{\partial \xi_{a}} \frac{\partial \xi_{a}}{\partial \lambda_{b}} + \frac{\partial L_{d}}{\partial \xi_{b}} \frac{\partial \xi_{b}}{\partial \lambda_{b}} + \frac{\partial L_{d}}{\partial \xi_{c}} \frac{\partial \xi_{c}}{\partial \lambda_{b}} \\
\frac{\partial L_{d}}{\partial\lambda_{c}} &= \frac{\partial L_{d}}{\partial \xi_{a}} \frac{\partial \xi_{a}}{\partial \lambda_{c}} + \frac{\partial L_{d}}{\partial \xi_{c}} \frac{\partial \xi_{b}}{\partial \lambda_{c}} + \frac{\partial L_{d}}{\partial \xi_{c}} \frac{\partial \xi_{c}}{\partial \lambda_{c}},
\end{align*}%
which is the momentum matching condition expressed so that it is compatible with the change of natural charts.%
\section{Numerical Experiments} \label{NumExp}

Thus far, we have discussed the construction of Lie group Galerkin variational integrators, and established bounds on their rate of convergence. We will now turn to several numerical examples to demonstrate that our methods behave in practice as our theory predicts.

\subsection{Cayley Transform Method for the Rigid Body}
In \S \ref{CayCon} we have discussed in great detail a specific construction of a Lie group Galerkin variational integrators for the free rigid body based on the Cayley transform. Based on the convergence results from Theorems \ref{LG_OptConv} and \ref{LG_GeoConv}, we would expect our construction to converge geometrically with \(n\)-refinement and optimally with \(h\)-refinement.

Using MATLAB, we implemented the Lie group Galerkin variational integrator described in \S \ref{CayCon}, using a finite-difference Newton method as a root finder. We used the parameters%
\begin{align*}
J_{d} &= \mbox{diag}\lefri{1.3, 2.1, 1.2}\\
R\lefri{0} &= I \\
R^{T}\lefri{0}\dot{R}\lefri{0} &= \widehat{\lefri{2.0, -1.9, 1.0}}^{T}.
\end{align*}
To establish convergence, we first computed a numerical solution using a low-order splitting method with a very small time step, and once we established that the Lie group Galerkin variational integrator's solution and the splitting method's numerical solutions agreed, we used a Lie group Galerkin variational integrator solution with \(n = 26\) and \(h = 0.5\) as a high-order approximation to the exact solution, and established convergence to this solution. We made this choice of parameters for our approximate exact solution because it appeared that for this choice of parameters, the residual from the nonlinear solver was the dominant source of error, and neither \(h\) nor \(n\) refinement improved our numerical solution.

The results, which are summarized in Figures \ref{fig:GeoConvRigidBody} -- \ref{fig:ConservedQuantitiesRigidBody}, establish the rates of convergence predicted in Theorems \ref{LG_OptConv} and \ref{LG_GeoConv}. For \(n\)-refinement, we see that our integrator did indeed achieve geometric convergence, as can be seen in Figure \ref{fig:GeoConvRigidBody}. However, unlike the vector space method (see \citet{HaLe2012}), we did not observe the difference in convergence rates of the continuous approximation and the one-step map. We suspect this is because until very high accuracy is achieved, the inaccurate boundary conditions due to the one-step map error dominates the continuous approximation error, and the threshold at which the continuous approximation error is greater than the one-step error is related to the time step. While we can take extremely large time steps in the vector space case, in the Lie group case the time step length is limited by the natural chart, and hence we never observe the lower convergence rate of the continuous approximation. WE explore convergence with \(h\)-refinement, see Figure \ref{fig:hConvRigidBody}, and observe the optimal rate of convergence for our construction for even \(n\). However, for odd \(n\), we see convergence at a rate of \(n-1\). We do not have a clear explanation for this.

Now, considering the geometric invariants related to the rigid body, we see that the Cayley transform based method has excellent conservation properties. Figure \ref{fig:ConservedQuantitiesRigidBody} shows one of the classic depictions of geometric invariants for the rigid body, that is the intersection of the two hypersurfaces in momentum space given by the two geometric invariants \(C\lefri{y} = \frac{1}{2}\sum_{i=1}^{3} y_{i}^{2}\) and \(H\lefri{y} = \frac{1}{2} \sum_{i=1}^{3} I^{-1}_{i}y_{i}^{2}\) where \(y\) is the angular momentum of the rigid body. These invariants correspond to the norm of the body fixed angular momentum and the energy, respectively. Discussions of these invariants, and comparable behavior of other methods can be found in \citet{MaRa1999} and \citet{HaLuWa2006} (specifically, see \citet{HaLuWa2006} for a comparison to other numerical methods). Our method has nearly perfect conservation of these invariants.

Additionally, while it is not perfectly conserved, the energy behavior of our method is oscillatory and remains bounded even for very long integration times, as can be seen in Figure \ref{fig:EnergyBehaviorRigidBody}. This type of behavior is typical for variational integrators, and can be understood in terms of backwards error analysis.

\begin{figure}[htpb]
  \centering
  \includegraphics[width = 0.75\textwidth]{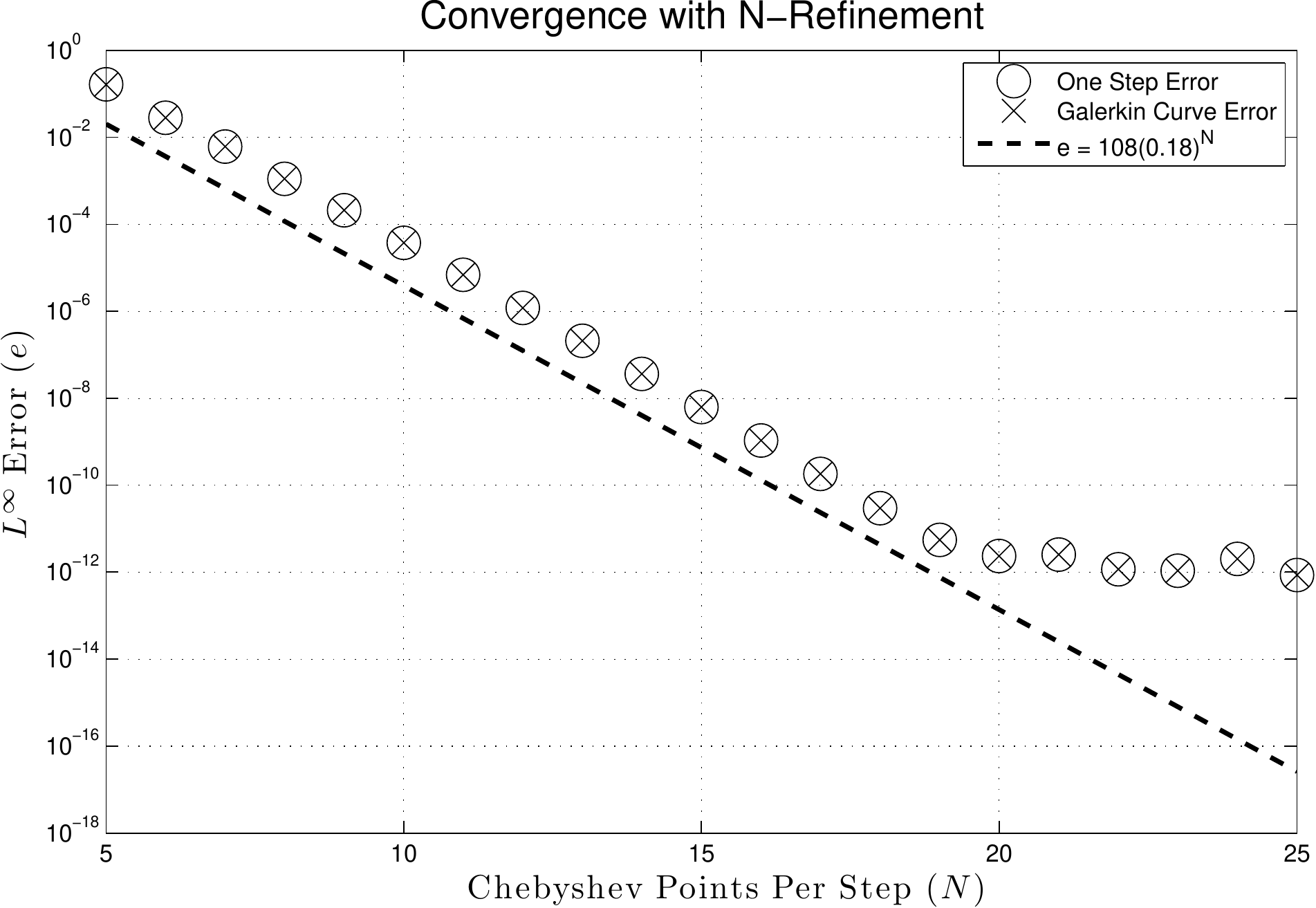}
  \caption{Geometric convergence of the Lie group spectral variational integrator based on the Cayley transform for the rigid body. We use a constant step-size \(h  = 0.5\). Note that the Galerkin curves have the same error as the one-step map, even though they have a theoretical lower rate of convergence.}
  \label{fig:GeoConvRigidBody}
\end{figure}

\begin{figure}[htpb]
  \centering
  \includegraphics[width = 0.75\textwidth]{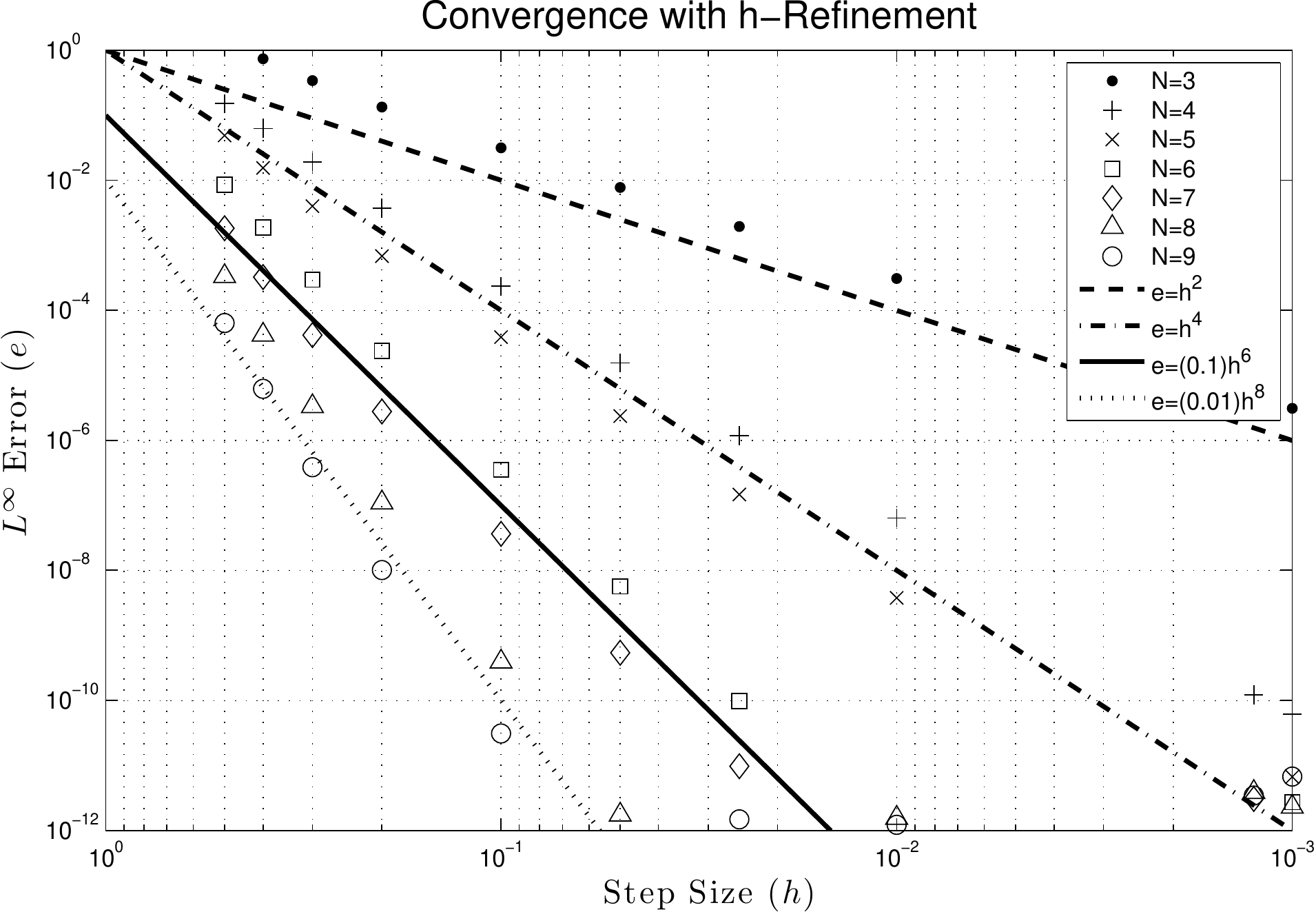}
  \caption{Order optimal convergence of the Lie group Galerkin variational integrator based on the Cayley transform for the rigid body.}
  \label{fig:hConvRigidBody}
\end{figure}

\begin{figure}[htpb]
  \centering
  \includegraphics[width = 0.75\textwidth]{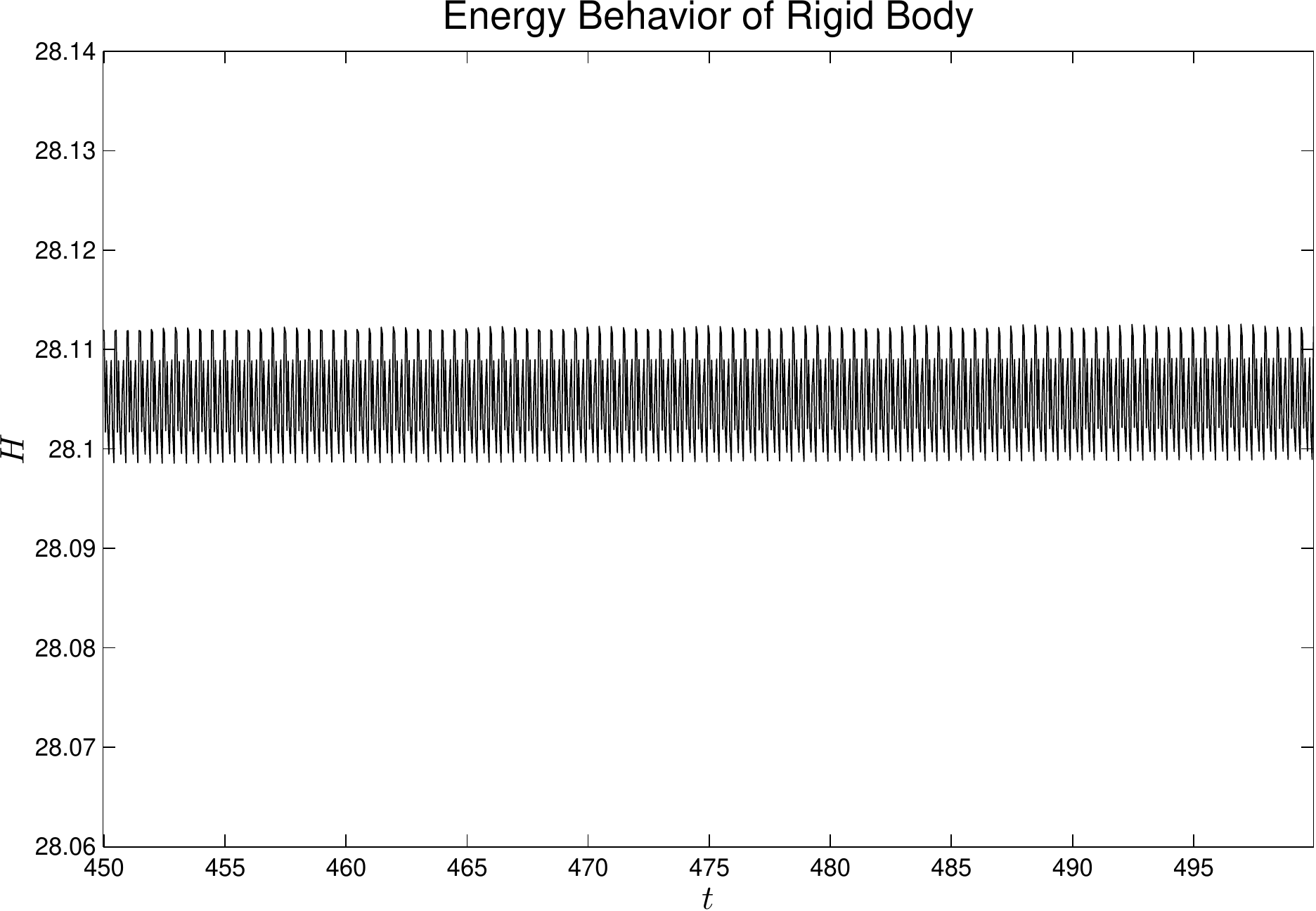}
  \caption{Energy behavior of the Lie group Galerkin variational integrator based on the Cayley transform for the rigid body. This is from a simulation starting at \(t_{0} = 0.0\), and using the parameters \(n = 12\), \(h = 0.5\) for the integrator.}
  \label{fig:EnergyBehaviorRigidBody}
\end{figure}

\begin{figure}[htpb]
  \centering
  \includegraphics[width = 0.75\textwidth]{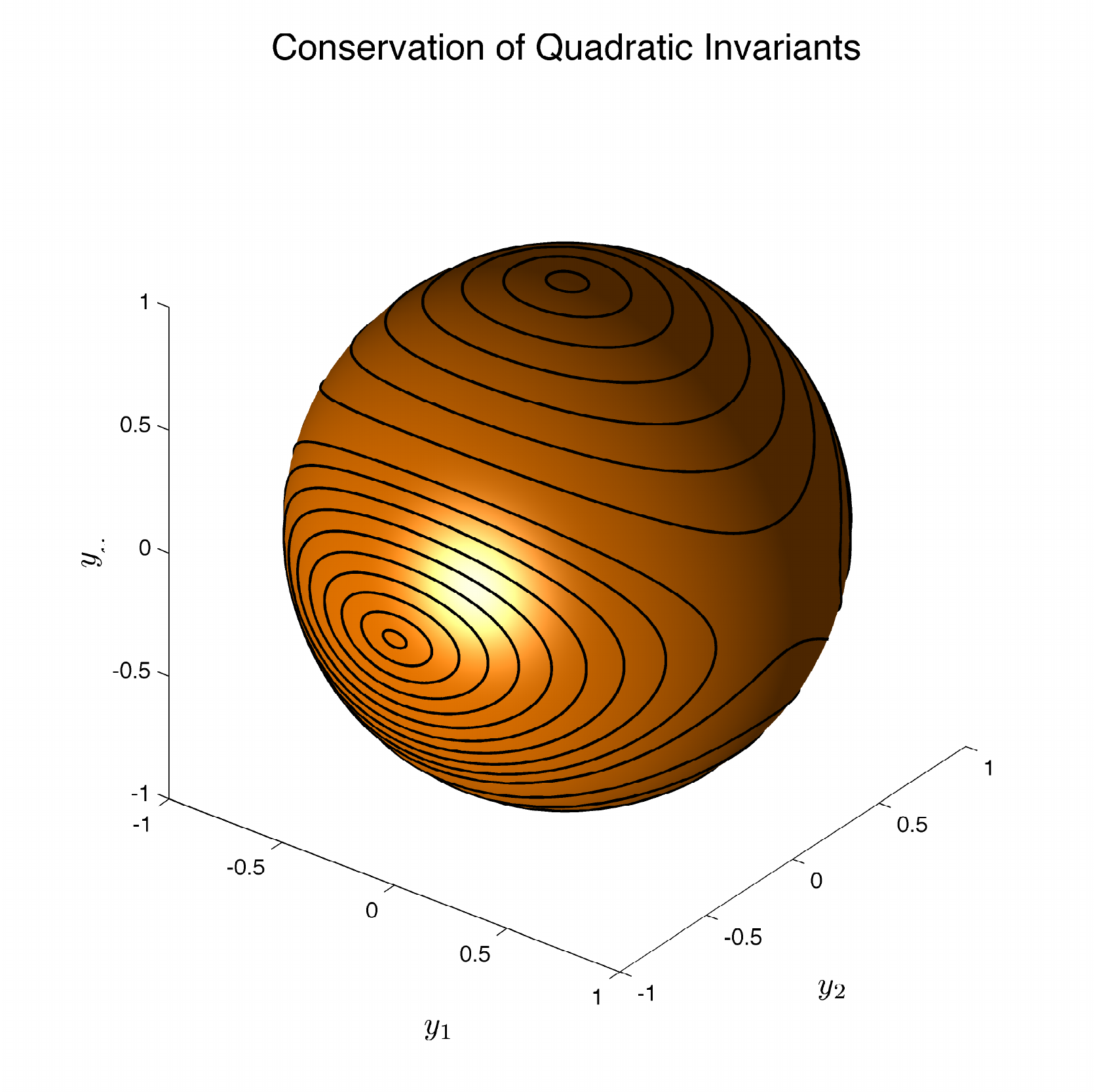}
  \caption{Conserved quantities for the Lie group Galerkin variational integrator based on the Cayley transform for the rigid body. This is from a series of computations using the parameters \(n = 8\), and \(h = 0.5\), from a variety of initial conditions. Note that the trajectories computed by the Lie group Galerkin variational integrators, which are the black curves, lie on the intersections of \(\sum_{i=1}^{3}y^{2}_{i} = 1\), \(\sum_{i=1}^{3}I_{i}^{-1}y^{2}_{i} = 2\), which are the norm of angular momentum and energy, respectively.}
  \label{fig:ConservedQuantitiesRigidBody}
\end{figure}

\subsection{Cayley Transform Method for the 3D Pendulum}

For a second numerical experiment, we examine the 3D pendulum. The 3D pendulum is the rigid body with one point fixed and under the influence of gravity, and its Lagrangian is:%
\begin{align*}
L\lefri{R,\dot{R}} &= \frac{1}{2}\mbox{tr}\lefri{\dot{R}^{T}RJ_{d}R^{T}\dot{R}} + mge_{3}^{T}R\rho\\
J_{d} &= \mbox{diag}\lefri{1, 2.8, 2}\\
\rho & = \lefri{0,0,1}^{T}
\end{align*}%
where \(\rho\) is center of mass for \(R = I\), \(m\) is the mass of the pendulum and \(g\) is the gravitational constant. We consider two sets of initial conditions, the first,%
\begin{align*}
R\lefri{0} &= I \\
R\lefri{0}^{T}\dot{R}\lefri{0} & = \widehat{\lefri{0.5,-0.5,0.4}^{T}},
\end{align*}%
which is a slight perturbation from stable equilibrium, and the second
\begin{align*}
R\lefri{0} &= \mbox{diag}\lefri{-1, 1, -1}\\
R\lefri{0}^{T}\dot{R}\lefri{0} & = \widehat{\lefri{0.5,-0.5,0.4}^{T}}
\end{align*}%
which is the pendulum slightly perturbed from its unstable equilibrium. 

We construct the variational integrator for this system using the Cayley transform. This involves adding the term \(V\lefri{\lgc} = mge_{3}^{T}\LGForm{\lgc}\rho\) to the discrete action in equations (\ref{CayDEP}), and finding the stationarity conditions of this new discrete action, which gives us the new internal stage discrete Euler-Poincar\'{e} equations. These have the same form as equations (\ref{DEPone}) -- (\ref{DEPthree}), with added terms for the potential.

For the first set of initial data, which are near the stable equilibrium, we see exactly the expected convergence with both \(h\) and \(N\) refinement, as is illustrated in Figures \ref{fig:GeoConvPend} and \ref{fig:hConvPend}. Furthermore, we see bounded oscillatory energy behavior over the length of the integration, as in Figure \ref{fig:EnergyBehaviorStablePend}.

For the second set of initial data, this system evolves chaotically, so convergence of individual trajectories is not of great interest. What is more important is the conservation of geometric invariants as the system evolves. As can be seen from Figures \ref{fig:DynamicsUnstablePen} and \ref{fig:EnergyBehaviorUnstablePen}, the energy of the system is nearly conserved, even with very aggressive time stepping. Of particular note is that even though there are many steps where the solution undergoes a change that approaches the limit on the conditioning of the natural chart, the energy error remains small.

\begin{figure}[htpb]
  \centering
  \includegraphics[width = 0.75\textwidth]{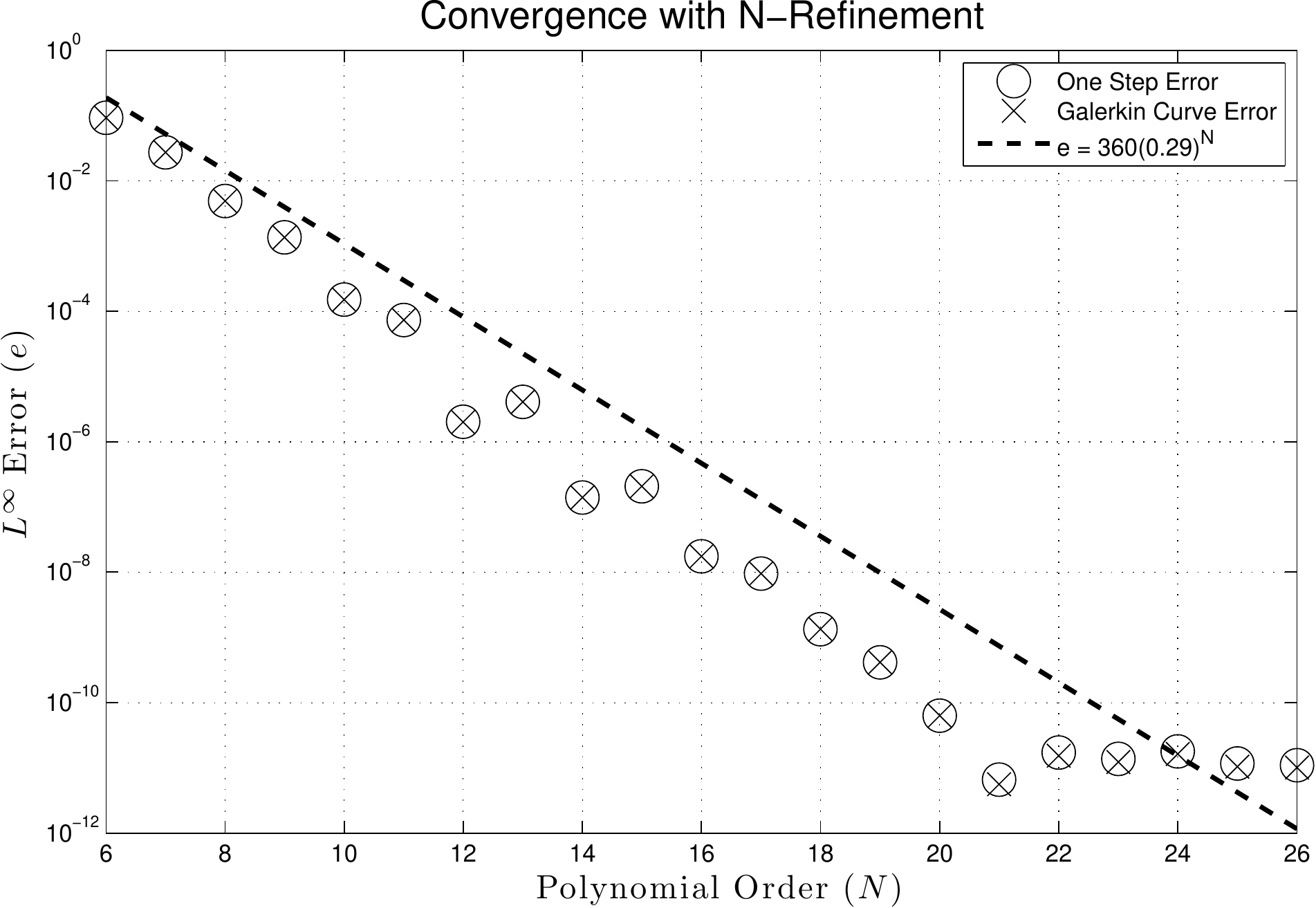}
  \caption{Geometric convergence of the Lie group spectral variational integrator based on the Cayley transform for the 3D pendulum for a small perturbation from the stable equilibrium. We use the time step \(h  = 0.5\). Note that, once again, the Galerkin curves have the same error as the one-step map.}
  \label{fig:GeoConvPend}
\end{figure}

\begin{figure}[htpb]
  \centering
  \includegraphics[width = 0.75\textwidth]{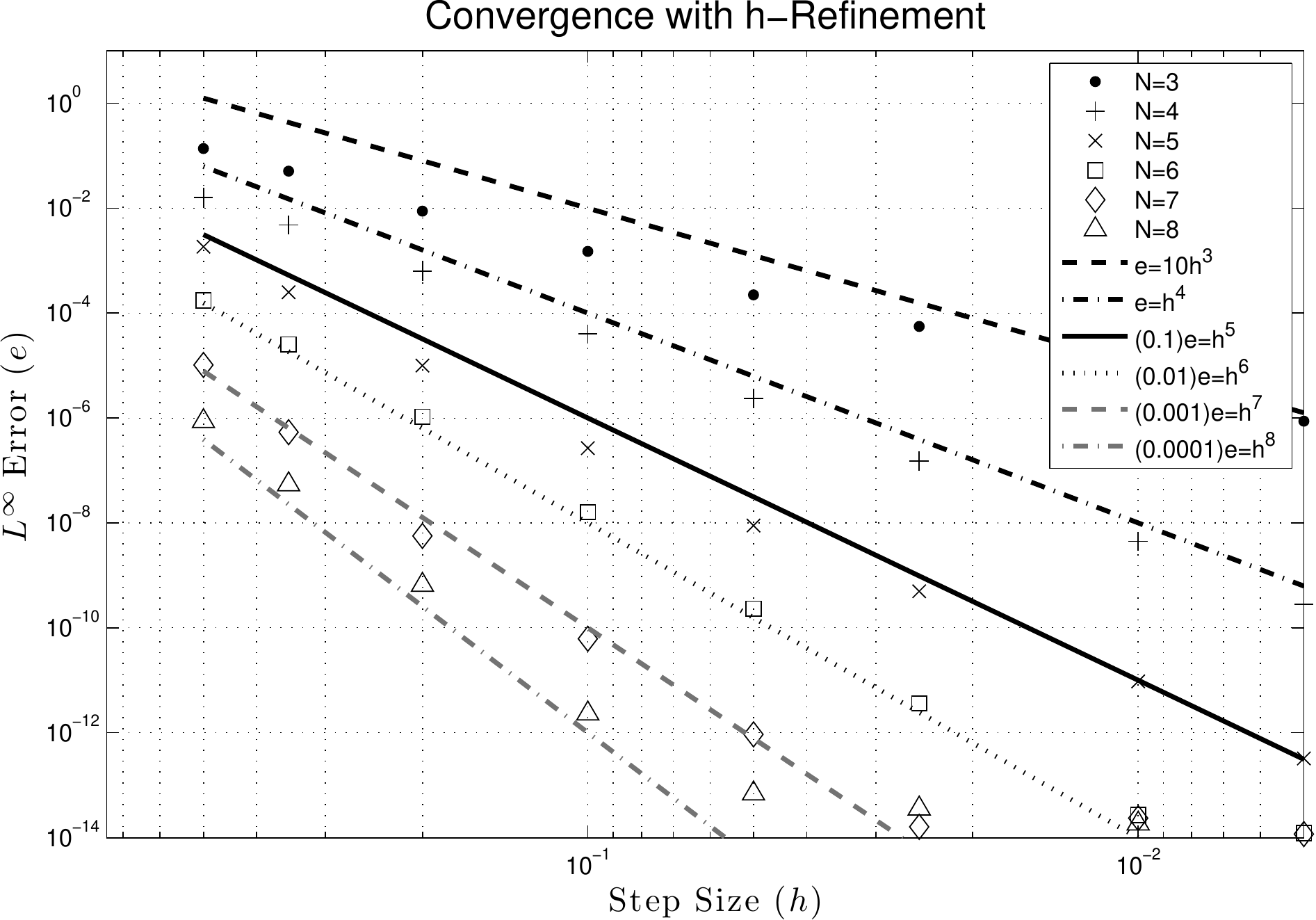}
  \caption{Order optimal convergence of Lie group Galerkin variational integrator based on the Cayley transform for the 3D pendulum for a small perturbation from the stable equilibrium. Note that we have almost exactly order optimal convergence.}
  \label{fig:hConvPend}
\end{figure}

\begin{figure}[htpb]
  \centering
  \includegraphics[width = 0.75\textwidth]{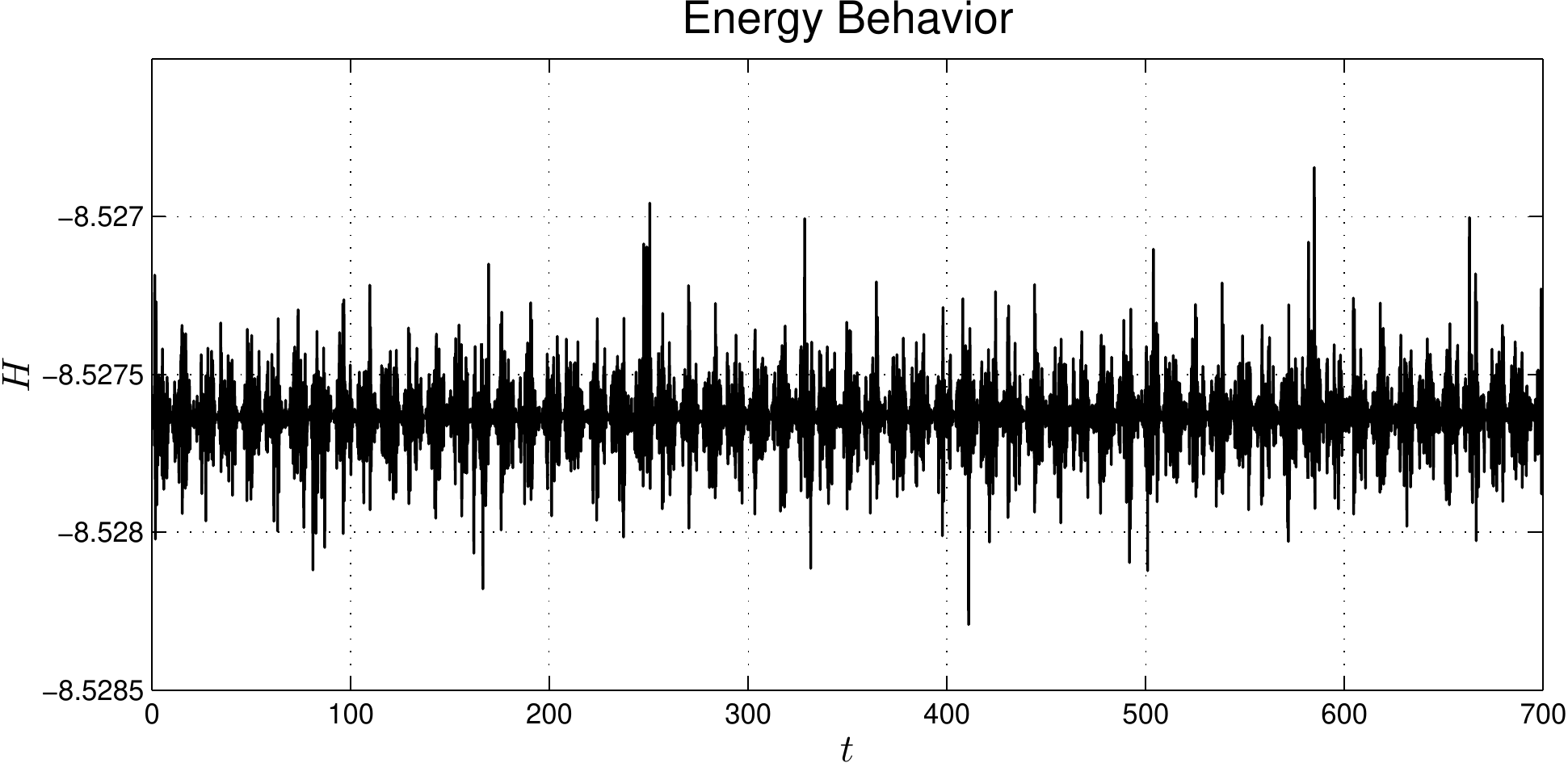}
  \caption{Energy behavior of the Lie group Galerkin variational integrator based on the Cayley transform for the 3D pendulum for a small perturbation from the stable equilibrium. This is the behavior of an integrator constructed with parameters \(n = 8\), step size \(h = 1.5\). Note that the error is both small and oscillatory, but not increasing.}
  \label{fig:EnergyBehaviorStablePend}
\end{figure}

\begin{figure}[htpb]
  \centering
  \includegraphics[width = 0.75\textwidth]{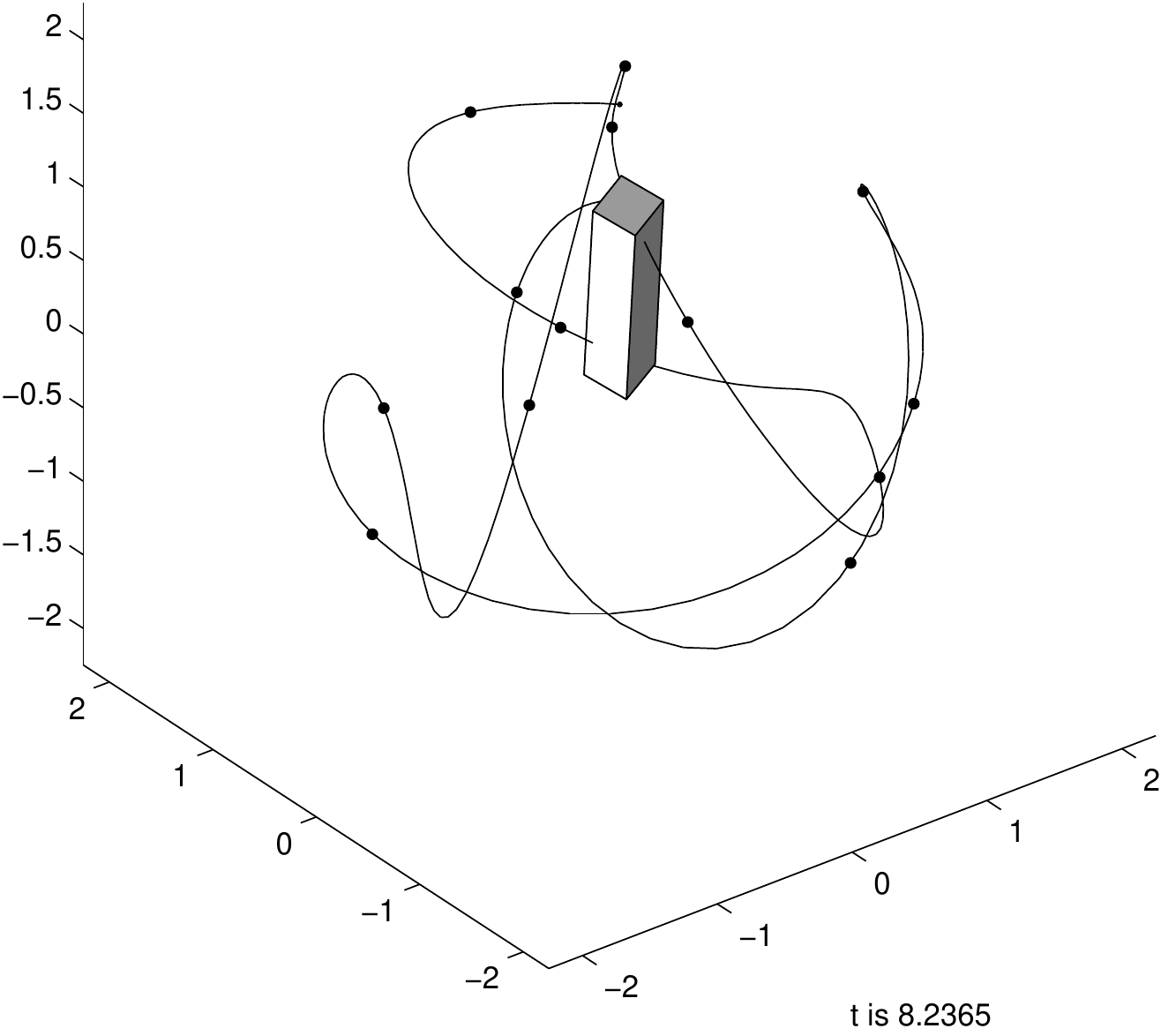}
  \caption{Dynamics of the numerical simulation of the 3D pendulum constructed from a Lie group Galerkin variational integrator. These dynamics were constructed from an integrator with \(n = 20\), \(h = 0.6\). The black dots each represent a single step of the one-step map, and the solid lines are the Galerkin curves. Note that some of the steps are almost through an angle of length \(\pi\), which is the limit of the conditioning of the natural chart.}
  \label{fig:DynamicsUnstablePen}
\end{figure}

\begin{figure}[htpb]
  \centering
  \includegraphics[width = 0.75\textwidth]{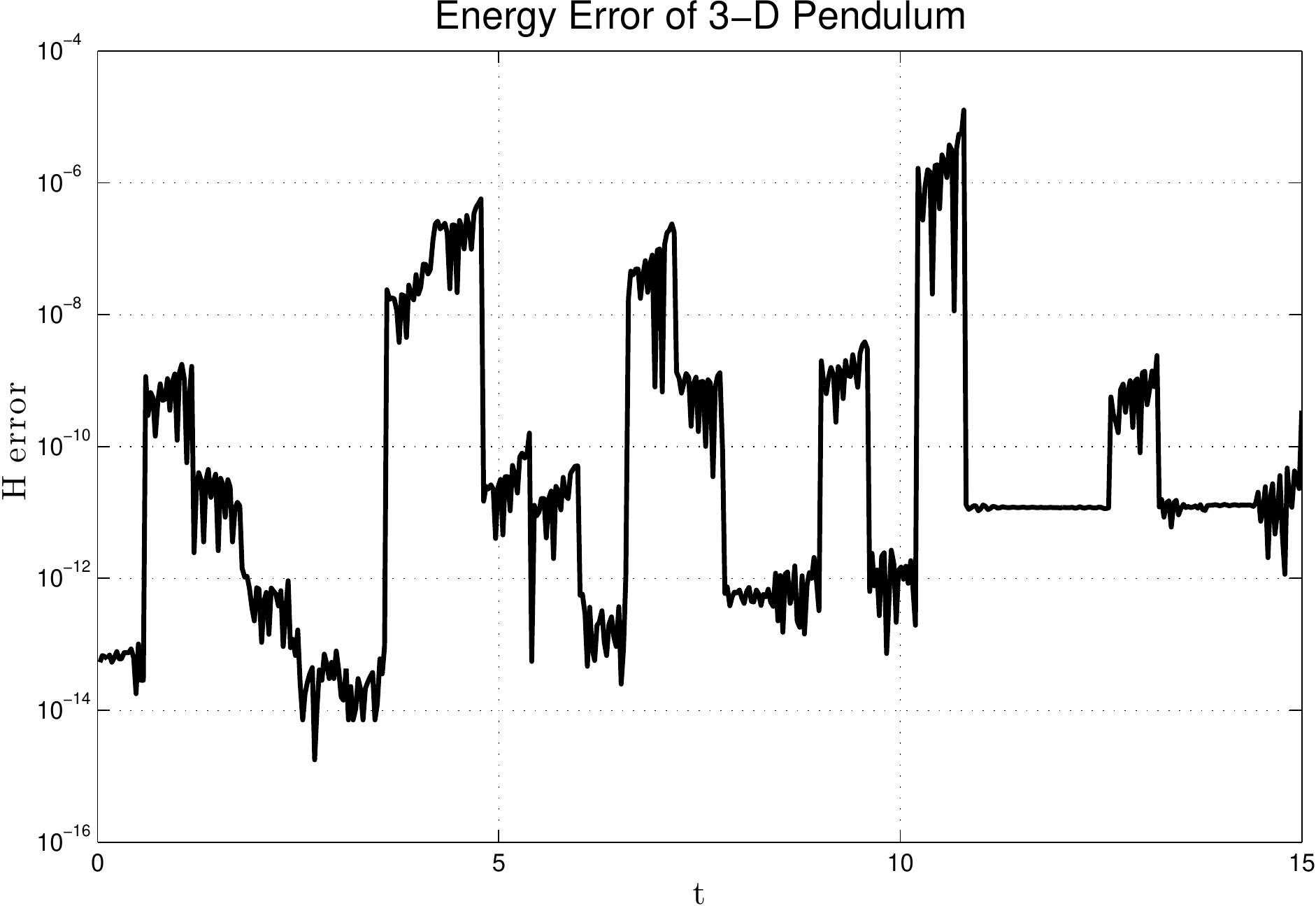}
  \caption{Energy error of the dynamics depicted in Figure \ref{fig:DynamicsUnstablePen}. The large jumps in error are associated with time steps that almost exceed the conditioning of the natural chart.}
  \label{fig:EnergyBehaviorUnstablePen}
\end{figure}

\section{Conclusions and Future Work}

In this paper, we have presented a new numerical method for Lagrangian problems on Lie groups. Specifically, we used a Galerkin construction to create variational integrators of arbitrarily high-order, and also Lie group spectral variational integrators, which converge geometrically. We demonstrated that in addition to inheriting the excellent geometric properties common to all variational integrators, which include conservation of the symplectic form, and conservation of momentum, that such integrators also are extremely stable even for large time steps, can be adapted for a large class of problems, and yield highly accurate continuous approximations to the true trajectory of the system.

We also gave an explicit example of a Lie group Galerkin variational integrator constructed using the Cayley transform. Using this construction, we demonstrated the expected rates of convergence on two different example problems, the rigid body and the 3D pendulum. We also showed that these methods both have excellent energy and momentum conservation properties. Additionally, we provided explicit expressions for the internal stage discrete Euler-Poincar\'{e} equations for the free rigid body, which form the foundation of a numerical method for a variety of problems.

\subsection{Future Work}

Symplectic integrators continue to be an area of interest, and there has been considerable success in developing high-order structure-preserving methods and applying such methods to relevant problems. While we have developed a significant amount of the theory of Lie group Galerkin variational integrators, there is considerable future work to be done.%
\subsubsection{Choice of Natural Charts} In our construction, we chose the Cayley transform to construct our natural chart. While this choice made the derivation of the resulting integrator simpler, it also introduced a limitation on the conditioning of the natural chart. A possible extension of our framework would be constructions based on natural charts constructed from other functions. An obvious choice is the exponential map, which was the choice of chart function used in earlier works that proposed this construction. A comparison of the behavior of integrators constructed from other choices of natural chart functions would be interesting further work. %
\subsubsection{Novel Variational Integrators:} One of the attractive features of our work is that we establish an optimality result for arbitrary approximation spaces. Because of this, our results hold for a variety of different possible constructions of variational integrators. It would be interesting to investigate the behavior of variational integrators constructed from novel approximation spaces, such as wavelets, or for variational integrators that make use of specialized function spaces, such as spaces that include both high and low frequency functions for problems with components that evolve on different time scales.%
\subsubsection{Larger classes of Problems} In this paper, we have focused most of our attention on the rigid body and problems that evolve on \(\SoT\). However, there are many examples of Lie group problems that evolve on other spaces. Our analysis suggests that the Galerkin approach would be effective for these problems. It would be interesting to examine Galerkin variational integrators for problems that evolve on other Lie groups, and apply our methods to other interesting applications.%
\subsubsection{Parallel Implementation and Computational Efficiency} Perhaps our method's greatest flaw is that it requires the solution of a large number of nonlinear equations at every time step. This problem is further exasperated by the fact that assembling the Newton matrix at every time step requires the repeated application of a high-order quadrature rule. While the fact that our method is stable even for very large time steps helps to overcome this computational difficulty, it would be interesting to see how much our method could be accelerated by assembling our Newton matrix in parallel.%
\subsubsection{Multisymplectic Variational Integrators} Multisymplectic geometry has become an increasingly popular framework for extending much of the geometric theory from classical Lagrangian mechanics to Lagrangian PDEs. The foundations for a discrete theory have been laid, and there have been significant results achieved in geometric techniques for structured problems such as elasticity, fluid mechanics, non-linear wave equations, and computational electromagnetism. However, there is still significant work to be done in the areas of construction of numerical methods, analysis of discrete geometric structure, and especially error analysis. Galerkin type methods have become a standard in classical numerical PDE methods, popular examples include Finite-Element, Spectral, and Pseudospectral methods. The variational Galerkin framework could provide a natural framework for extending these classical methods to structure-preserving geometric methods for PDEs.

\section{Appendix}

In \S \ref{SecGeoConv}, we stated Theorem \ref{LG_GeoConv} but did not provide a proof. This is because the proof is essentially the same as that for optimal convergence, with slight and obvious modifications. For completeness, we will provide the proof here. 
\begin{theorem} Given an interval \(\left[0,h\right]\), and a Lagrangian \(L:TG \rightarrow \mathbb{R}\), suppose that \(\trug\lefri{t}\) solves the Euler-Lagrange equations on that interval exactly. Furthermore, suppose that the exact solution \(\trug\lefri{t}\) falls within the range of the natural chart, that is:%
\begin{align*}
\trug\lefri{t} = \LGForm{\trula\lefri{t}}
\end{align*}
for some \(\trula \in C^{2}\lefri{\left[0,h\right],\mathfrak{g}}\). For the function space \(\AFdFSpace\) and the quadrature rule \(\mathcal{G}\), define the Galerkin discrete Lagrangian \(\GDL{g_{0}}{g_{1}} \rightarrow \mathbb{R}\) as%
\begin{align}
\GDLh{g_{0}}{g_{1}}{n} = \ext_{\Ggalargsf{g_{0}}{g_{1}}} h \sum_{j=1}^{m}b_{j}L\lefri{g_{n}\lefri{c_{j}h},\dot{g}_{n}\lefri{c_{j}h}} = h \sum_{h=1}^{m}b_{j}L\lefri{\galgn\lefri{c_{j}h},\dgalgn\lefri{c_{j}h}} \label{DiscAct}
\end{align}%
where \(\galgn\lefri{t}\) is the extremizing curve in \(\GalSpace\). If:
\begin{enumerate}
\item there exists an approximation \(\optla \in \AFdFSpace\) such that,%
\begin{align*} 
\RMetric{\trula - \optla}{\trula - \optla}^{\frac{1}{2}} & \leq \ApproxC \ApproxK^{n}\\
\RMetric{\dtrula - \doptla}{\dtrula - \doptla}^{\frac{1}{2}} &\leq \ApproxCi \ApproxK^{n},
\end{align*}%
for some constants \(\ApproxC \geq 0\) and \(\ApproxCi \geq 0\), \(0 < \ApproxK < 1\) independent of \(n\),
\item the Lagrangian \(L\) is Lipschitz in the chosen error norm in both its arguments, that is:%
\begin{align*}
\left|L\lefri{g_{1},\dot{g}_{1}} - L\lefri{g_{2},\dot{g}_{2}}\right| \leq \LagLipC \lefri{\GroupE{g_{1}}{g_{2}} + \AlgeE{\dot{g}_{1}}{\dot{g}_{2}}} 
\end{align*}%
\item the chart function \(\Phi\) is well-conditioned in \(\GroupE{\cdot}{\cdot}\) and \(\AlgeE{\cdot}{\cdot}\), that is (\ref{CharCond1}) and (\ref{CharCond2}) hold,
\item there exists a sequence of quadrature rules \(\left\{\mathcal{G}_{n}\right\}_{n=1}^{\infty}\), \(\mathcal{G}_{n}\lefri{f} = h \sum_{j=1}^{m_{n}}\bnj f\lefri{\cnjh} \approx \int_{0}^{h} f\lefri{t}\dt\), and there exists a constant \(0 < \QuadK < 1\) independent of \(n\) such that,%
\begin{align*}
\left|\int_{0}^{h}L\lefri{g_{n}\lefri{t},\dot{g}_{n}\lefri{t}}\dt - h \sum_{j=1}^{m_{n}}\bnj L\lefri{g_{n}\lefri{\cnjh},\dot{g}_{n}\lefri{\cnjh}}\right| \leq \QuadC \QuadK^{n}
\end{align*}%
for any \(g_{n}\lefri{t} = L_{g_{0}}\Phi\lefri{\xi\lefri{t}}\) where \(\xi \in \AFdFSpace\).
\item the stationary points of the discrete action and the continuous action are minimizers,
\end{enumerate}
then the variational integrator induced by \(\GDLh{g_{0}}{g_{1}}{n}\) has error \(\mathcal{O}\lefri{\SpecK^{n}}\) for some constant \(\SpecK\) independent of \(n\), \(0 < \SpecK < 1\). 
\end{theorem}

\begin{proof}%
We begin by rewriting the exact discrete Lagrangian and the Galerkin discrete Lagrangian:%
\begin{align*}
\left|\EDLh{g_{0}}{g_{1}}{n} - \GDLh{g_{0}}{g_{1}}{h}\right| &= \left|\int_{0}^{h}L\lefri{\trug,\dtrug} \dt - h\sum_{j=1}^{m_{n}}\bnj L\lefri{\galgn\lefri{\cnjh},\dgalgn\lefri{\cnjh}}\right|,
\end{align*}%
where we have introduced \(\galgn\), which is the stationary point of the local Galerkin action (\ref{DiscAct}). We introduce the solution in the approximation space which takes the form \(\optgn\lefri{t} = \LGForm{\optla\lefri{t}}\), and compare the action on the exact solution to the action on this solution:%
\begin{align*}
\left|\int_{0}^{h}L\lefri{\trug,\dtrug}\dt - \int_{0}^{h}L\lefri{\optgn,\doptgn}\dt\right| &= \left|\int_{0}^{h}L\lefri{\trug,\dtrug} - L\lefri{\optgn,\doptgn}\dt\right| \\
& \leq \int_{0}^{h}\left|L\lefri{\trug,\dtrug} - L\lefri{\optgn,\doptgn}\right|\dt.
\end{align*}%
Now, we use the Lipschitz assumption to establish the bound  
\begin{align*}
\int_{0}^{h}\left|L\lefri{\trug,\dtrug} - L\lefri{\optgn,\doptgn}\right|\dt & \leq \int_{0}^{h}\LagLipC \lefri{\GroupE{\trug}{\optgn} + \AlgeE{\dtrug}{\doptgn}}\dt \\
& = \int_{0}^{h} \LagLipC \left(\GroupE{\LGForm{\trula}}{\LGForm{\optla}} + \right. \\
& \hspace{5em} \left. \AlgeE{\DLGForm{\trula}}{\DLGForm{\optla}} \right) \dt,
\end{align*}%
and the chart conditioning assumptions to see
\begin{align*}
\int_{0}^{h}\left|L\lefri{\trug,\dtrug} - L\lefri{\optgn,\doptgn}\right|\dt & \leq \int_{0}^{h} \LagLipC \left(\GroupC\RMetric{\trula - \optla}{\trula - \optla}^{\frac{1}{2}} + \AlgeC\RMetric{\dtrula - \doptla}{\dtrula - \doptla}^{\frac{1}{2}} + \right. \\
& \hspace{5em} \left. \phantom{\RMetric{\dtrula - \doptla}{\dtrula - \doptla}^{\frac{1}{2}}} \AlgeGC\RMetric{\trula - \optla}{\trula - \optla}^{\frac{1}{2}}\right) \dt \\
& \leq \int_{0}^{h} \LagLipC \lefri{\GroupC \ApproxC \ApproxK^{n} + \AlgeC \ApproxCi \ApproxK^{n} + \AlgeGC \ApproxC \ApproxK^{n}} \dt \\
& = \LagLipC\lefri{\lefri{\GroupC + \AlgeGC}\ApproxC + \AlgeC \ApproxCi}\ApproxK^{n}.
\end{align*}%
This establishes a bound between the action evaluated on the exact discrete Lagrangian and the optimal solution in the approximation space. Considering the Galerkin discrete action,
\begin{align}
h\sum_{j=1}^{m_{n}} \bnj L\lefri{\galgn,\galgn} & \leq h \sum_{j=1}^{m_{n}} \bnj L\lefri{\optgn,\doptgn} \nonumber \\  
& \leq \int_{0}^{h}L\lefri{\optgn,\doptgn}\dt + \QuadC \QuadK^{n} \nonumber \\
& \leq \int_{0}^{h}L\lefri{\trug,\dtrug}\dt + \QuadC \QuadK^{n} + \LagLipC\lefri{\lefri{\GroupC + \AlgeGC}\ApproxC + \AlgeC\ApproxCi}\ApproxK^{n} \label{GeoIneq1}
\end{align}%
where we have used the assumption that the Galerkin approximation minimizes the Galerkin discrete action and the assumption on the accuracy of the quadrature. Now, using the fact that \(\trug\lefri{t}\) minimizes the action and that \(\GalSpace \subset \CG\),
\begin{align}
h\sum_{j=1}^{m_{n}}\bnj L\lefri{\galgn,\dgalqn} & \geq \int_{0}^{h}L\lefri{\galgn,\dgalgn}\dt - \QuadC \QuadK^{n} \nonumber \\
& \geq \int_{0}^{h} L\lefri{\trug,\trug}\dt - \QuadC \QuadK^{n}. \label{GeoIneq2}
\end{align}%
Combining inequalities (\ref{GeoIneq1}) and (\ref{GeoIneq2}), we see that,%
\begin{align*}
\int_{0}^{h}L\lefri{\trug,\dtrug} \dt - \QuadC \QuadK^{n} \leq h \sum_{j=1}^{m_{n}}\bnj L\lefri{\galgn,\dgalgn} \leq \int_{0}^{h} L\lefri{\trug,\dtrug} \dt + \QuadC \QuadK^{n} + \LagLipC\lefri{\lefri{\GroupC + \AlgeGC} \ApproxC + \AlgeC \ApproxCi}\ApproxK^{n}
\end{align*}%
which implies%
\begin{align}
\left|\int_{0}^{h} L\lefri{\trug,\dtrug} \dt - h \sum_{j=1}^{m_{n}}\bnj L\lefri{\galgn,\dgalgn}\right| \leq \lefri{\QuadC + \LagLipC\lefri{\lefri{\GroupC + \AlgeGC}\ApproxC + \AlgeC\ApproxCi}} \SpecK^{n} \label{GeoIneq3}
\end{align}%
where \(\SpecK = \max\lefri{\ApproxK,\QuadK}\). The left hand side of (\ref{GeoIneq3}) is exactly \(\left|\EDLh{g_{0}}{g_{1}}{h} - \GDLh{g_{0}}{g_{1}}{n}\right|\), and thus%
\begin{align*}
\left|\EDLh{g_{0}}{g_{1}}{h} - \GDLh{g_{0}}{g_{1}}{n}\right| \leq \lefri{\QuadC + \LagLipC\lefri{\lefri{\GroupC + \AlgeGC}\ApproxC + \AlgeC\ApproxCi}} \SpecK^{n}.
\end{align*}%
This states that the Galerkin discrete Lagrangian approximates the exact discrete Lagrangian with error \(\mathcal{O}\lefri{\SpecK^{n}}\), and by Theorem (\ref{ErrorThm}) this further implies that the Lagrangian update map has error \(\mathcal{O}\lefri{\SpecK^{n}}\).
\end{proof}%
\bibliographystyle{agsm}
\bibliography{LG_SVI_Bibliography}

\end{document}